\documentclass[12pt,reqno]{amsart}
\usepackage{amssymb,latexsym,amsmath,amsthm,amscd,epsfig}
\usepackage[active]{srcltx}
\usepackage[latin1]{inputenc}
\usepackage{pdfsync}

\def\cal{\mathcal}

\def\bT{\mathbf{T}}
\def\bZ{\mathbf{Z}}

\def\cD{\mathcal{D}}
\def\cF{\mathcal{F}}

\def\cG{\mathcal{G}}
\def\cP{\mathcal{P}}

\def\cT{\mathcal T}

\def\bpm{\begin{pmatrix}}
\def\epm{\end{pmatrix}}

\newcommand{\ttt}{{\rm Top}}

\newcommand{\sel}{{\rm Sel}}
\newcommand{\sss}{{\rm Stop}}
\newcommand{\tree}{{\rm Tree}}
\newcommand{\term}{{\rm Term}}
\newcommand{\rf}[1]{{\eqref{#1}}}

\newcommand{\fiproof}{{\hspace*{\fill} $\square$ \vspace{2pt}}}
\newcommand{\D}{{\mathbb D}}  
\newcommand{\wt}[1]{{\widetilde{#1}}}
\newcommand{\wh}[1]{{\widehat{#1}}}

\newcommand{\maxbad}{{\rm Bad}}
\newcommand{\eee}{{\rm End}}

\newtheorem{thm}{Theorem}[section]

\newtheorem{lemma}[thm]{Lemma}

\newtheorem{definition}[thm]{Definition}
\newtheorem{theorem}[thm]{Theorem}

\newtheorem*{theorema}{Theorem A}

\newtheorem{lem}[thm]{Lemma}

\theoremstyle{remark}
\newtheorem{remark}[thm]{\bf Remark}

\numberwithin{equation}{section}

\def\R{\mathbb R}
\def\C{\mathbb C}
\def\D{\mathbb D}
\def\H{\mathcal H}

\def\cB{\mathcal B}
\def\cG{\mathcal G}

\def\Z{\mathbb Z}
\def\diam{\text{diam}}
\def\Lip{\operatorname{Lip}}
\def\dist{\operatorname{dist}}
\def\supp{\operatorname{supp}}
\def\diam{\operatorname{diam}}

\newcommand{\ve}{{\varepsilon}}

\title{Quasiconformal maps, analytic capacity, and non linear potentials}

\author[X. Tolsa]{Xavier Tolsa}
\address{Instituci\'{o} Catalana de Recerca i Estudis Avan\c{c}ats (ICREA) and Departament de Matem\`{a}tiques,  Universitat Aut\`{o}noma de Barcelona, 08193 Bellaterra (Barcelona), Catalonia}
\email{{\tt xtolsa@math.uab.cat}}
\urladdr{http://mat.uab.es/~xtolsa}

\author[I. Uriarte-Tuero]{Ignacio Uriarte-Tuero}
\address{Department of Mathematics, Michigan State University, East Lansing, MI 48824, USA}
\email{{\tt ignacio@math.msu.edu}}

\thanks{2000 {\em Mathematical Subject Classification.}
.}


\thanks{X.\ T.\ is partially supported by grants MTM2007-62817, MTM2010-16232 (Spain), and
2009-SGR-420 (Catalonia). I.\ U.\ 
was a postdoctoral fellow at the University of Missouri, Columbia, USA, and at Centre de Recerca Matem\`{a}tica, Barcelona, Spain, for some periods of time during the elaboration of this paper. During the elaboration of this paper, he has been partially supported by grants DMS-0901524, CAREER DMS-1056965 (US NSF), Sloan Research Foundation, and MTM2010-16232, MTM2009-14694-C02-01 (Spain)}

\begin{document}


\begin{abstract}
In this paper we prove that if $\phi:\C\to\C$ is a $K$-quasiconformal map, with $K>1$, 
and $E\subset \C$ is a compact set contained in a ball $B$, then
$$\frac{\dot C_{ \frac{2K}{2K+1},\frac{2K+1}{K+1}}(E)}{\diam(B)^{\frac2{K+1}}}  \geq 
c^{-1} \left(\frac{\gamma(\phi(E))}{\diam(\phi(B))}\right)^{\frac{2K}{K+1}},$$
where $\gamma$ stands for the analytic capacity and $\dot C_{ \frac{2K}{2K+1},\frac{2K+1}{K+1}}$
is a capacity associated to a non linear Riesz potential. 
As a consequence, if $E$ is not $K$-removable (i.e. removable for bounded $K$-quasiregular maps), it has positive capacity $\dot C_{ \frac{2K}{2K+1},\frac{2K+1}{K+1}}$. This improves previous results that assert that $E$ must have
non $\sigma$-finite Hausdorff measure of dimension $2/(K+1)$.
We also show that the indices $\frac{2K}{2K+1}$, $\frac{2K+1}{K+1}$ are sharp, and that Hausdorff gauge functions do not appropriately discriminate which sets are $K$-removable. So essentially we solve the problem of finding sharp ``metric" conditions for $K$-removability.
\end{abstract}

\maketitle

\thispagestyle{empty}


\section{Introduction}

A homeomorphism $\phi:\Omega\to \Omega'$ between planar domains is called $K$-quasiconformal
if it preserves orientation, it belongs to the Sobolev space $W^{1,2}_{\rm loc}(\Omega)$, and satisfies
$$\max_\alpha|\partial_\alpha\phi| \leq K\,\min_\alpha|\partial_\alpha\phi|\quad \mbox{ a.e. in $\Omega$.}$$ If one does not ask $\phi$ to be a homeomorphism, then one says that 
$\phi$ is quasiregular. When $K=1$, the class of quasiregular maps coincides with the
one of analytic functions.

A compact set $E\subset \C$ is said to be removable for bounded $K$-quasiregular
maps (or, $K$-removable) if for every open set $\Omega\supset E$, every bounded
$K$-quasiregular map $f:\Omega\setminus E\to\C$ admits a $K$-quasiregular extension
to $\Omega$. By Stoilow's theorem, it turns out that $E$ is $K$-removable if, and only if, for every planar $K$-quasiconformal map $\phi$,  
$\phi(E)$ is removable for bounded analytic functions (i.e. $\phi(E)$ is $1$-removable). 
The Painlev\'{e} problem for $K$-quasiregular mappings consists in describing $K$-removable
sets in metric and geometric terms, in analogy to the classical Painlev\'{e} problem of characterizing removable
sets for bounded analytic functions in metric and geometric terms (i.e.\ precisely the case $K =1$). In this case ($K=1$), a ``solution'' in terms of the so called curvature
of measures was obtained in \cite{tolsasemiadditivityanalyticcapacity} (see also \cite{tolsabilip}).

The analytic capacity of a compact set $E\subset \C$ is defined by
$$\gamma(E)  =\sup_f |f'(\infty)|,$$
where the supremum is taken over all bounded analytic functions $f:\C\setminus E\to \C$
with $\|f\|_\infty\leq 1$, and
$$f'(\infty)= \lim_{z\to\infty}z(f(z)-f(\infty)).$$
This set function was introduced by Ahlfors in order to study the Painlev\'{e} problem for 
bounded analytic functions. He showed that $E$ is removable for these functions if and
only if $\gamma(E)=0$. By the relationship between $1$-removable and $K$-removable sets 
explained above, it turns out that $E$ is $K$-removable if and only if $\gamma(\phi(E))=0$
for all planar $K$-quasiconformal maps.

An old theorem of Painlev\'{e} shows that $\gamma (E) \lesssim \H^1 (E)$ for any compact set $E$, where $\H^1 (E)$ denotes the one-dimensional Hausdorff measure (length) of $E$.
This result gives the sharp condition in terms of Hausdorff measures for removability
for bounded analytic functions, since sets of zero length are removable, but there are sets of positive length which are not removable (e.g. a line segment.)

By Painlev\'{e}'s theorem, if $\gamma(\phi(E))>0$, then $\phi(E)$ has positive
length, and so it has Hausdorff dimension at least $1$. By the celebrated theorem
of Astala on the distortion of area \cite{astalaareadistortion}, this forces the Hausdorff
dimension of $E$ to be at least $2/(K+1)$. Quite recently, in \cite{ACMOU} it was shown 
that, in fact, the $\frac{2}{K+1}$-Hausdorff measure $\H^{\frac2{K+1}}(E)$ must be positive and, moreover, non $\sigma$-finite, or, equivalently, if $\H^{\frac2{K+1}}(E)$ is $\sigma$-finite, then $E$ is $K$-removable.
To get this result, the authors proved, on the one hand, 
that if $\H^1(\phi(E))$ is non $\sigma$-finite, then $\H^{\frac2{K+1}}(E)$ is also non 
$\sigma$-finite, equivalently, if $\H^{\frac2{K+1}}(E)$ is $\sigma$-finite, then $\H^1(\phi(E))$ is $\sigma$-finite (see \cite{Lacey-Sawyer-Uriarte} for related recent results). On the other hand, if $\H^1(\phi(E))$ is $\sigma$-finite and 
$\gamma(\phi(E))>0$, from David's solution of Vitushkin's conjecture 
\cite{davidunrectifiable1setszeroanalyticcapacity} and 
the countable semiadditivity of analytic capacity \cite{tolsasemiadditivityanalyticcapacity}, 
it turns out that $\phi(E)$ contains some
rectifiable subset of positive length.
Using improved distortion estimates for the dimension of rectifiable sets, the authors showed
that in this case the Hausdorff dimension of $E$ must be strictly larger that $2/(K+1)$, and
so $\H^{\frac2{K+1}}(E)$ is also non $\sigma$-finite in this case.

Moreover, in \cite{ACMOU} it was also shown (see Theorem \ref{NonRemovabilityBoundedKQRACMOUT} below) that for any measure function $h(t)=t^\frac{2}{K+1}\,\varepsilon(t)$ such that
$$
  \int_0 \frac{\varepsilon(t)^{1+1/K}}{t} dt < \infty \ ,
$$
there is a compact set $E\, $ which is not $K$-removable and such that $0<\H^h(E)<\infty$. In particular, whenever  $\varepsilon(t)$ is chosen so that in addition for every  $\alpha >0$ we have  $t^\alpha/\varepsilon(t) \to 0$  as  $t \to 0$, then there is a non-$K$-removable set $E$ with $\dim(E)=\frac{2}{K+1}$. Note that there is a small ``gap" between the sufficient condition for the gauge function $h$ for the existence of non-$K$-removable sets just mentioned, and the sufficient condition for $K$-removability of sets also from \cite{ACMOU} mentioned before (namely that $\H^{\frac{2}{K+1}} (E)$ be $\sigma$-finite, see Theorem \ref{RemovabilityBoundedKQRACMOUT} below.)

\medskip 
The main result of this paper improves on the preceding results, giving the sharp ``metric" condition for $K$-removability in terms of Riesz capacities (i.e. closing the aforementioned ``gap"):

\begin{theorem}\label{distorgamma}
Let $E\subset\C$ be compact and $\phi:\C\to\C$ a $K$-quasiconformal mapping, $K>1$. 
If $E$ is contained in a ball $B$,
then
$$\frac{\dot C_{ \frac{2K}{2K+1},\frac{2K+1}{K+1}}(E)}{\diam(B)^{\frac2{K+1}}}  \geq 
c^{-1} \left(\frac{\gamma(\phi(E))}{\diam(\phi(B))}\right)^{\frac{2K}{K+1}}.$$
\end{theorem}

In this theorem, the constant $c$ depends only on $K$. On the other hand,  $\dot C_{ \frac{2K}{2K+1},\frac{2K+1}{K+1}}$
is a Riesz capacity associated to a non linear potential. Recall that, for 
$\alpha>0$, $1<p<\infty$ with $0<\alpha p <2$, the Riesz capacity $\dot C_{\alpha,p}$
of $F$ is defined as
$$\dot C_{\alpha,p}(F) = \sup_\mu \mu(F)^p,$$
where the supremum runs over all positive measures $\mu$ supported on $F$ such that
$$I_\alpha(\mu)(x) = \int \frac1{|x-y|^{2-\alpha}}\,d\mu(x)$$
satisfies $\|I_\alpha(\mu)\|_{p'}\leq 1$, where as usual $p'=p/(p-1)$. 

It is easy to check that $\dot C_{\alpha, p}$ is a homogeneous capacity of degree $2-\alpha p$, that is,
$$\dot C_{\alpha, p}(\lambda F)
= |\lambda|^{2-\alpha p}\, \dot C_{ \alpha, p}(F)$$
for any compact set $F\subset \C$ and $\lambda\in\C$.
Therefore, 
 $\dot C_{ \frac{2K}{2K+1},\frac{2K+1}{K+1}}$ has  homogeneity $2/(K+1)$. 
The indices $\alpha=\frac{2K}{2K+1}$, $p=\frac{2K+1}{K+1}$, are sharp and cannot be improved in the theorem. See Theorem \ref{teosharp} below for a more precise statement. 

It is well known that sets with
positive capacity $\dot C_{\alpha,p}$ have non $\sigma$-finite
Hausdorff measure $\H^{2-\alpha p}$. So as a direct corollary of Theorem \ref{distorgamma}
one recovers the result of \cite{ACMOU} that asserts that if $\gamma(\phi(E))>0$, then $\H^{\frac2{K+1}}(E)$ is non $\sigma$-finite, or, as phrased in \cite{ACMOU}, if $\H^{\frac2{K+1}}(E)$ is $\sigma$-finite, then $\gamma(\phi(E))=0$.

On the other hand, not all sets with non $\sigma$-finite $\H^{\frac2{K+1}}$ measure
have positive capacity $\dot C_{ \frac{2K}{2K+1},\frac{2K+1}{K+1}}$. 
So Theorem \ref{distorgamma} provides new examples of $K$-removable sets. More precisely, for any increasing nonnegative function $h$ on $[0,\infty)$, with 
$$
\int_0^1 \left( \frac{h(r)}{r^{\frac{2}{K+1}}}  \right)^{1+\frac{1}{K}} \; \frac{dr}{r} = \infty \ ,
$$
there is a compact set $E \subset \C$ such that the generalized Hausdorff measure $\H^h(E)>0$ and $\dot C_{\frac{2K}{2K+1}, \frac{2K+1}{K+1}}(E) =0$. See Theorem \ref{Theorem5.4.2AdamsHedberg} below. 

However, for any positive function $h$ on $(0,\infty)$ such that
$$\varepsilon (r) = \frac{h(r)}{r^{\frac{2}{K+1}}} \to 0 \quad\mbox{ as $r\to0$,}$$
there is a compact set $E \subset \C$ such that $\H^h(E)=0$ and a $K$-quasiconformal map $\phi$ such that $\gamma(\phi E) >0$ (and hence $\dot C_{ \frac{2K}{2K+1},\frac{2K+1}{K+1} }(E) >0$, due to Theorem \ref{distorgamma}.) See
 Section~\ref{SectionExamples} (in particular, 
Theorem \ref{ForAnyGaugeFunctionMeasuringSetsStrictlyLargerThan2/K+1ThereIsANonRemovableSet}) for more details and examples.

The results just mentioned strongly suggest 1) that the ``gap" left in \cite{ACMOU} in terms of characterizing sharp ``metric" conditions for $K$-removability cannot be closed in terms of Hausdorff gauge functions, 2) that the language of (Riesz) capacities is more appropriate to obtain sharp ``metric" conditions for $K$-removability, and 3) that (in view of the aforementioned Theorem \ref{teosharp} below and the fact that we recover the previously known sufficient conditions for $K$-removability from \cite{ACMOU}), we close the ``gap" by obtaining sharp ``metric" conditions for $K$-removability in terms of Riesz capacities.

For our purposes, the description of Riesz capacities in terms of Wolff
potentials is more useful than the above definition of $\dot C_{\alpha,p}$. Consider
$$\dot W^\mu_{\alpha,p}(x) = \int_0^\infty \biggl(\frac{\mu(B(x,r))}{r^{2-\alpha p}}\biggr)^{p'-1}\,\frac{dr}r.$$
A well known theorem of Wolff asserts that
$$\dot C_{\alpha,p}(F) \approx \sup_\mu \mu(F),$$
where the supremum is taken over all measures $\mu$ supported on $F$ such that
$\dot W_{\alpha,p}^\mu(x)\leq 1$ for all $x\in F$. See \cite[Chapter 4]{adamshedberg},
for instance.
Notice that for the indices  $\alpha=\frac{2K}{2K+1}$, $p=\frac{2K+1}{K+1}$,
we have 
$$\dot W_{\alpha,p}^\mu(x) = \int_0^\infty \biggl(\frac{\mu(B(x,r))}{r^{\frac2{K+1}}}\biggr)^{\frac{K+1}K}\,\frac{dr}r.$$

We will also prove the following result in this paper.

\begin{theorem} \label{teocap}
Let  $1<p<\infty$, $E\subset\C$ be compact and $\phi:\C\to\C$ a $K$-quasiconformal mapping. Then,
\begin{itemize}
\item[(a)] If $E$ is contained in a ball $B$,
\begin{equation}\label{eq11}
\frac{\dot C_{\frac{2K}{2Kp-K+1},\frac{2Kp-K+1}{K+1}}(E)}{\diam(B)^{\frac2{K+1}}} \gtrsim 
\left(\frac{\dot C_{1/p,p}(\phi(E))}{\diam(\phi(B))}\right)^{\frac{2K}{K+1}}.
\end{equation}
\item[(b)]
If $\phi$ is conformal outside $E$, $K$-quasiconformal in $\C$, and moreover,
$|\phi(z) - z| = O(1/|z|)$ as $z\to\infty$, then
\begin{equation}\label{eq12}
\dot C_{1/p,p}(E)\approx \dot C_{1/p,p}(\phi(E)) 
\end{equation}
\end{itemize}
The constants in \rf{eq11} and \rf{eq12} only depend on $p$, $K$.
\end{theorem}

Notice that the capacity $\dot C_{1/p,p}$ is homogeneous of degree $1$, while 
$\dot C_{\frac{2K}{2Kp-K+1},\frac{2Kp-K+1}{K+1}}$
is homogeneous of degree $2/(K+1).$

To understand the relationship between analytic capacity and 
non linear potentials, we need to recall the characterization of analytic capacity in
terms of curvature. For $x\in\C$, denote
\begin{equation}\label{eqdfcm}
c^2_\mu(x) := \iint \frac1{R(x,y,z)^2}\,d\mu(y)d\mu(z),
\end{equation}
where $R(x,y,z)$ stands for the radius of the circle through $x,y,z$ (with $R(x,y,z)
=\infty$ if the points are colinear).

\begin{theorema} \label{teogamcur}
For any compact $E\subset\C$ we have
 $$\gamma(E)\simeq \sup\mu(E),$$
where the supremum is taken over all Borel measures $\mu$
supported on $E$ such that $\mu(B(x,r))\leq r$ for all $x\in \C$,
$r>0$ and $c^2_\mu(x)\leq 1$ for all $x\in \C$.
\end{theorema}

The inequality $\gamma(E)\gtrsim \sup\mu(E)$ is due to Melnikov \cite{Melnikov},
while the (more difficult) converse was proved in 
\cite{tolsasemiadditivityanalyticcapacity}.

It is easy to check that 
$$\biggl(\sup_{r>0}\frac{\mu(B(x,r))}{r}\biggr)^2\ + c^2_\mu(x)\leq C\sum_{k\in\Z}\biggl(\frac{\mu(B(x,2^k))}{2^k}\biggr)^2\leq
C\,\dot W^\mu_{2/3,3/2}(x).$$
From this fact, one infers that
\begin{equation}\label{eqbeta0}
\gamma(F) \geq c^{-1}\,\dot C_{2/3,3/2}(F)
\end{equation}
for every compact set $F$.
On the other hand, Theorem \ref{teocap} tells us that
\begin{equation}\label{eq4g}
\dot C_{ \frac{2K}{2K+1},\frac{2K+1}{K+1}}(E)  \geq c^{-1}\, \left( \dot C_{2/3,3/2}(\phi(E)) \right)^{\frac{2K}{K+1}}
\end{equation}
(assuming $\diam(B)=\diam(\phi(B))=1$). If the estimate $\gamma(F) \approx \dot C_{2/3,3/2}(F)$ were true, then Theorem \ref{distorgamma} would follow from this and \rf{eq4g}. However, the comparability of $\gamma$ and $\dot C_{2/3,3/2}$ is false (for instance,
if $F$ is a segment, $\gamma(F)>0$, while $\dot C_{2/3,3/2}(F)=0$).

Nevertheless, for Cantor type sets $F$ such as the ones considered 
in \cite{mattilaanalyticcapacitycantorsets} and \cite{MTV} it is true
that $\gamma(F) \approx \dot C_{2/3,3/2}(F)$. 
So for this type of sets, the estimate
$$\frac{\dot C_{ \frac{2K}{2K+1},\frac{2K+1}{K+1}}(\phi^{-1}(F))}{\diam(B)^{\frac2{K+1}}}  \geq c^{-1} 
\left(\frac{\gamma(F)}{\diam(\phi(B))}\right)^{\frac{2K}{K+1}}
$$
is a direct consequence of Theorem \ref{teocap}. On the other hand, by the results
in \cite{ACMOU}, if $F$ is rectifiable (and thus $\gamma(F)>0$), then the Hausdorff dimension of $\phi^{-1}(F)$ is strictly larger than 
$2/(K+1)$, and so 
$$\dot C_{ \frac{2K}{2K+1},\frac{2K+1}{K+1}}(\phi^{-1}(F))  >0.$$
The proof of Theorem \ref{distorgamma}
for general sets $E$ follows by combining the arguments in Theorem \ref{teocap}  with quantitative estimates for the distortion of rectifiable 
sets (more precisely, for the distortion of sub-arcs of chord arc curves).  
To this end, we will
 need to use a corona type construction similar to the one used in 
 \cite{tolsabilip} to prove the bilipschitz invariance
 of analytic capacity, modulo multiplicative estimates.

The relationship between capacities $\gamma_\beta$ associated to Calder\'on-Zygmund kernels of 
the form $x/|x|^{\beta+1}$ in $\R^n$ and the capacities $\dot C_{\alpha,p}$ was first 
observed by 
Mateu, Prat and Verdera \cite{mateupratverdera}. In this paper the authors proved
 that if $0<\beta<1$, then
\begin{equation}\label{eqbeta}
\gamma_\beta\approx \dot C_{(n-\beta)2/3,3/2}.
\end{equation}
An immediate consequence is that
 sets of positive but finite $\beta$-Hausdorff measure are removable for $\gamma_\beta$, as shown previously by Prat \cite{prat1}.
In the case $n=2$, $\beta=1$, the capacity $\gamma_\beta$ coincides with the
analytic capacity $\gamma$, modulo multiplicative constants, and the comparability
\rf{eqbeta} fails. Instead, only the inequality \rf{eqbeta0} holds. It is
an open problem to prove (or disprove) that \rf{eqbeta} holds for every non integer $\beta\in(0,n)$. 

Other results concerning the relationship between the capacities
$\dot C_{\alpha,p}$ and the so called Lipschitz harmonic capacity (which can be
considered as a generalization of analytic capacity to higher dimensions) have been
obtained recently in \cite{ENV}. See \cite{Paramonov} and \cite{volberg} for
more details on Lipschitz harmonic capacity.

The plan of the paper is the following:
in next section we prove Theorem \ref{teocap}, while Sections 3-7 are devoted to the proof
Theorem \ref{distorgamma}. 
Section \ref{SectionExamples} contains examples and results that illustrate the sharpness
of our results, and in Section \ref{sec9} there are some final remarks.
As usual, the letters $c,C$ denote constants (often, absolute constants) that may change 
at different occurrences, while constants with subscript, such as $C_1$, 
retain their values. The notation $A\lesssim B$ means that there is a positive constant $C$ such that $A\leq CB$; and $A\approx B$ means that $A\lesssim B\lesssim A$.

\bigskip


\section{Distortion estimates for non linear potentials}  \label{sec1}

\subsection{Strategy for the proof of Theorem \ref{teocap}}


Note that, that \rf{eq11} holds for all $K$-quasiconformal maps is equivalent to
\begin{equation}\label{eqd44}
\frac{\dot C_{ \frac{2K}{2Kp-K+1},\frac{2Kp-K+1}{K+1}}(\phi(E))}{\diam(\phi(B))^{\frac2{K+1}}}  \geq c^{-1} 
\left(\frac{\dot 
C_{1/p,p}(E)
}{\diam(B)}\right)^{\frac{2K}{K+1}} \quad \mbox{ if $E\subset B$,}
\end{equation}
 for all $K$-quasiconformal maps $\phi$. The difference between \rf{eq11} and \rf{eqd44} is mostly ``psychological": think of $\phi$ as taking a set $E$ of dimension $1$ to a set of dimension $\frac{2}{K+1}$, as opposed to the other way round (which would be looking at $\phi^{-1}$). The equivalence between \rf{eq11} and \rf{eqd44} uses quasisymmetry.

Let $\mu$ 
be a measure supported on $E$ such that $\dot W_{1/p,p}^\mu(x)\leq 1$ for all $x\in\C$.
In a sense, we want to show how $\mu$ is distorted.
A first attempt might consist in obtaining suitable estimates for
the Wolff potentials associated to the image measure $\phi\mu$. However,
we have not been able to follow this approach.

Instead, to prove \rf{eqd44}, we have transformed our original problem of estimating
distortion in terms of Riesz capacities into another involving ``Hausdorff-like'' 
measures or contents. As far as we know, this way of studying the Riesz capacities is new. The advantage of this 
 approach is that arguments  using covering lemmas and related techniques are
better suited for Hausdorff contents than for non linear potentials. In this way, some of the arguments in 
\cite{ACMOU} will be adapted to estimate Riesz capacities via the intermediate contents of Hausdorff type
that we will construct below.

\medskip
Throughout all this section we suppose that $\mu$ is a finite Borel measure supported on $E$ such that  $\dot W_{1/p,p}^\mu(x)\leq 1$ for all $x\in\C$.
In particular, notice that this implies that $\theta_\mu(B):=\mu(B)/r(B)\lesssim 1$ for any ball $B\subset \C$ with radius $r(B)$. 
We plan to introduce Hausdorff-like measures associated to $\mu$. To this end, first
we need to define suitable gauge functions on all the balls in $\C$.
Given a parameter $a>0$, we consider the function 
\begin{equation}\label{eqpsia}
\psi_a(x) = \frac1{|x|^{1+a}+1},\qquad x\in\C.
\end{equation}
For the ball $B=B(x,t)$ we define
\begin{equation}\label{defvex}
\ve_{\mu,a}(x,t) = \ve_{\mu,a}(B) := \frac1t \int \psi_a\Bigl(\frac{y-x}t\Bigr)d\mu(y),
\end{equation}
and we consider the gauge function
\begin{equation}\label{defhx}
h_{\mu,a}(x,t) = h_{\mu,a}(B) := t\ve_{\mu,a}(B).
\end{equation}
Notice that $\ve_{\mu,a}(B)$ and $h_{\mu,a}(B)$ can be considered as smooth versions of 
$\theta_\mu(B)$ and $\mu(B)$, respectively. One of the advantages of 
$\ve_{\mu,a}(x,t)$ over $\theta_\mu(x,t)$ (where, of course, $\theta_\mu(x,t):=\theta_\mu(B(x,t))$) is that 
 $\ve_{\mu,a}(x,2t)\leq C\ve_{\mu,a}(x,t)$ for any $x$ and $t>0$, which fails in general for $\theta_\mu(x,t)$. Analogously, we have $h_{\mu,a}(x,2t)\leq C\,h_{\mu,a}(x,t)$, while
 $\mu(B(x,t))$ and $\mu(B(x,2t))$ may be very different.

Observe that, decomposing the integrals into annuli, for all $x\in\C$ we get 
\begin{align*}
\int_0^\infty \ve_{\mu,a}(x,t)^{p'-1}\frac{dt}t &= \int_0^\infty\frac1{t^{p'-1}} \biggl(\int \psi_a\Bigl(\frac{y-x}t\Bigr)d\mu(y)\biggr)^{p'-1}\frac{dt}t\\
& \leq C \sum_{j\in \Z} 2^{-(p'-1)j} \Bigl(\sum_{k>j}\mu(B(x,2^k)) 2^{(1+a)(j-k)}\Bigr)^{p'-1}\\
& \leq C \sum_{j\in \Z} 2^{-(p'-1)j} \sum_{k>j}\mu(B(x,2^k))^{p'-1} 2^{(p'-1)(1+\frac a2)(j-k)},
\end{align*}
where we applied H\"older's inequality for $p'-1>1$, and the fact that $(c+d)^{p'-1}\leq c^{p'-1}+ d^{p'-1}$ otherwise.
Thus,
\begin{equation}\label{ed**}
\int_0^\infty \ve_{\mu,a}(x,t)^{p'-1}\frac{dt}t 
 \lesssim
\sum_{k\in \Z}  \mu(B(x,2^k))^{p'-1}\, 
 2^{-(p'-1)(1+\frac a2)k}
 \sum_{j<k}  2^{(p'-1)\frac a2 j} \lesssim  \dot W^\mu_{1/p,p}(x)\lesssim 1.
\end{equation}


\subsection{The contents $M^h$ and the families $\cG_1$ and $\cG_2$}

Let $\cB$ denote the family of all closed balls contained in $\C$.
We consider a function $\ve:\cB:\to[0,\infty)$ (for instance, we can take $\ve=\ve_{\mu,a}$), and
we define $h(x,r)=r\,\ve(x,r)$. We assume that $\ve,h$ are such that $h(x,r)\to0$ as $r\to0$, for all $x\in\C$·.
We introduce the measure $H^h$ following Carathéodory's construction (see \cite{mattila}, p.54):
given $0<\delta\leq\infty$ and a set $F\subset\C$, we consider
$$H^h_\delta(F) = \inf\sum_i h(B_i),$$
where the infimum is taken over all coverings $F\subset \bigcup_i B_i$ with balls $B_i$ with radii smaller that $\delta$. Finally, we
define 
$$H^h(F) = \lim_{\delta\to0} H^h_\delta(F).$$
Recall that $H^h$ is a Borel regular measure (see \cite{mattila}), although it is not a ``true'' Hausdorff measure.
For the $h$-content, we use the notation $M^h(E):=H_\infty^h(E)$.

We say that the function $\ve$ belongs to $\cG_1$ if it verifies the following properties for all balls 
$B(x,r)$, $B(y,s)$:
there exists a constant $C_0$ such that if $|x-y|\leq 2r$ and $r/2\leq s\leq 2r$, then
\begin{equation}\label{eqeq1}
C_0^{-1}\,\ve(x,r)\leq \ve(y,s)\leq C_0\,\ve(x,r).
\end{equation}
If moreover, there exists $C_0'$ such that
\begin{equation}\label{eqsum4}
\sum_{k\geq0} 2^{-k}\,\ve(x,2^kr)\leq C_0' \,\ve(x,r),
\end{equation}
then we set $\ve\in\cG_2$.

Notice that \rf{eqeq1} also holds with a different constant $C_0$ if one assumes $|x-y|\leq Cr$ and $C^{-1}r\leq s\leq Cr$.

It is easy to check that the function $\ve_{\mu,a}$ introduced above belongs to $\cG_1$ for all $a>0$, and to $\cG_2$ if $0<a<1$ (see Lemma \ref{lemtec5}
below for a stronger statement).
Moreover, we have:

\begin{lemma}
If $\ve\in\cG_1$ and $h(x,r)=r\,\ve(x,r)$, then Frostman's Lemma holds for $H^h$. That is to say,
given a compact set $F\subset \C$, the following holds: $M^h(F)>0$ if and only if
there exists a Borel measure $\nu$ supported on $F$ such that $\nu(B)\leq h(B)$ for any ball $B$. Moreover, one can find $\nu$ such that $\nu(F)\geq c^{-1} M^h(F)$.
\end{lemma}

The proof is almost the same as the one of the usual Frostman's Lemma (for instance, see \cite{mattila}, p.112), taking into account the regularity properties of the gauge functions $h\in\cG_1$.

For $h=h_{\mu,a}$, we have the following.

\begin{lemma}\label{lem2.2}
For any Borel set $A\subset\C$, we have
$$M^{h_{\mu,a}}(A) \geq C^{-1}\mu(A).$$
\end{lemma}

\begin{proof}
Given any $\eta>0$, consider a covering $A\subset \bigcup_i B_i$ by balls so that
$$\sum_i h_{\mu,a}(B_i) \leq  M^{h_{\mu,a}}(A) +\eta.$$ 
Since $\mu(B_i)\leq C h_{\mu,a}(B_i)$, we have
$$\mu(A)\leq \sum_i \mu(B_i)\leq C\sum_i h_{\mu,a}(B_i) \leq CM^{h_{\mu,a}}(A) + C\eta.$$ 
\end{proof}


Now, for technical reasons we need to extend the function $\ve(\cdot)$ defined on $\cB$ to the whole family of 
bounded sets. Given an arbitrary bounded set $A\subset\C$, let $B$ a ball with minimal diameter that contains
$A$. We define $\ve(A):=\ve(B)$. If $B$ is not unique, it does not matter. In this case, for definiteness we can choose
the infimum of the values $\ve(B)$ over all balls $B$ with minimal diameter containing $A$, for instance.
Analogously, if $h(x,r)=r\,\ve(x,r)$, we define $h(A)$ as the infimum of the $h(B)$'s.

It was mentioned above that $\ve_{\mu,a}\in\cG_2$. Our next objective consists in showing that if $\phi$ is
 a $K$-quasiconformal planar homeomorphism, then the function defined by
$$\ve(B) = \ve_{\mu,a}(\phi(B))$$ 
for any ball $B\subset\C$, also belongs to $\cG_2$. In fact, because of the geometric properties of quasiconformal
mappings and the smoothness of $\psi_a$, it is easily seen that $\ve$ satisfies \rf{eqeq1}.   
To show that \rf{eqsum4} also holds requires some more effort.
First we need a technical result, whose proof follows from an elementary calculation that we leave for the reader:

\begin{lemma}\label{lemtec1}
Let $a,b>0$, $a\neq b$, and denote $m=\min(a,b)$. For all $z\in\C$, we have
$$\sum_{k\geq 0} 2^{-bk}\,\frac1{\bigl(2^{-k}|z|\bigr)^{a} + 1}\leq \frac C{|z|^{m} + 1},$$
with $C$ depending only on $a,b$. 
\end{lemma}

\begin{lemma}\label{lemtec5}
Let $\phi:\C\to\C$ be a $K$-quasiconformal mapping. If $0<a<C_1b$ (where $C_1$ is a small positive constant
depepending only on $K$), then,
$$\sum_{j\geq0}\frac{\ve_{\mu,a}(\phi(B(x,2^jr)))}{2^{bj}}\leq C(K)\,\ve_{\mu ,a}(\phi(B(x,r))).$$
In particular, if $a$ is chosen small enough, the function $\ve$ defined by $\ve(B) = \ve_{\mu,a}(\phi(B))$ for 
any ball $B$, belongs to $\cG_2$.  
\end{lemma}

\begin{proof}
We denote $d_j={\rm \diam}(\phi(B(x,2^jr)))$. We have
\begin{align*}
S & :=\sum_{j\geq0}\frac{\ve_{\mu ,a}(\phi(B(x,2^jr)))}{2^{bj}} \lesssim \sum_{j\geq0}\frac{\ve_{\mu ,a}(B(\phi
(x),d_j))}{2^{bj}}\\
&  \lesssim \sum_{k\geq0}
\sum_{j:d_0 2^{k}\leq d_j< d_02^{k+1}}\frac{\ve_{\mu ,a}(B(\phi(x),2^kd_0))}{2^{bj}}.
\end{align*}
For each $j\geq0$ we have
$$\frac{d_j}{d_0} = \prod_{i=1}^j \frac{d_i}{d_{i-1}} \leq C \prod_{i=1}^j \frac{{\rm \diam}(\phi
(B(x,2^ir)))}{{\rm \diam}(\phi(B(x,2^{i-1}r)))}
\leq C(K)^j = 2^{C_2j},$$
with $C_2$ depending on $K$. Thus, for $j,k$ such that $d_0 2^{k}\leq d_j< d_02^{k+1}$,
$$2^j\geq \Bigl(\frac{d_j}{d_0}\Bigr)^{1/C_2}\approx 2^{k/C_2}.$$
Then we obtain
$$S\lesssim \sum_{k\geq0}
\sum_{j:d_0 2^{k}\leq d_j< d_02^{k+1}} \frac{\ve_{\mu ,a}(B(\phi(x),2^kd_0))}{2^{C_1bk}} \leq C 
\sum_{k\geq0} \frac{\ve_{\mu ,a}(B(\phi(x),2^kd_0))}{2^{C_1bk}},$$
with $C_1=1/C_2$.
From Lemma \ref{lemtec1}, if $0<a<C_1b$, we infer that
\begin{align*}
\sum_{k\geq0} \frac{\ve_{\mu ,a}(B(\phi(x),2^kd_0))}{2^{C_1bk}} & = \sum_{k\geq0} \frac{1}{2^{(1+C_1b)k}d_0}\int \frac1{\biggl(\dfrac{|\phi(x)-y|}{2^kd_0}\biggr)^{1+a} + 1}\,
d\mu(y) \\
& \lesssim \frac1{d_0}\int 
 \frac 1{\biggl(\dfrac{|\phi(x)-y|}{d_0}\biggr)^{1+a} + 1}\,d\mu(y) \\ & =
\ve_{\mu ,a}(B(\phi(x),d_0))\lesssim \ve_{\mu ,a}(\phi(B(x,r))).
\end{align*}
\end{proof}

Another result that shows that some properties of the functions from $\cG_1$ are preserved under composition with quasiconformal maps is the following.

\begin{lemma} \label{lemcg}
Let  $\phi:\C\rightarrow\C$ be a $K$-quasiconformal mapping, and 
$\ve_0\in\cG_1$. 
Define $\ve(B) = \ve_0(\phi(B))$ for any ball $B\subset\C$. 
For any $s>0$ we have
$$\int_0^\infty \ve(x,r)^s\,\frac{dr}r\leq C(K,s)\int_0^\infty \ve_0(\phi(x),r)^s\,\frac{dr}r.
$$
\end{lemma}

\begin{proof}
We have
$$
\int_0^\infty 
\ve_0(\phi(B(x,r)))^s\, \frac{dr}r 
\leq C(s)\sum_{j\in\Z} \ve_0(\phi(B(x, 2^j)))^s.$$
Denote now $r_j=\diam(\phi(B(x, 2^j))$. We obtain
\begin{align*}
\sum_{j\in\Z} \ve_0(\phi(B(x, 2^j))^s & = \sum_{k\in\Z} \sum_{j:2^k\leq r_j< 2^{k+1}}\!\!\!\!\ve_0(\phi(B(x, 2^j))^s \\ &\lesssim \sum_{k\in\Z} \sum_{j:2^k\leq r_j< 2^{k+1}}\!\!\!\!\ve_0(B(\phi(x), r_j))^s \\
& \lesssim C(K) \sum_{k\in\Z} \ve_0(B(\phi(x), 2^k))^s\leq C(K,s)\int_0^\infty 
\ve_0(\phi(x),r)^s\, \frac{dr}r,
\end{align*}
where we took into account that $\#\{j:2^k\leq r_j< 2^{k+1}\}\leq C(K)$ because of the geometric properties of quasiconformal mappings.
\end{proof}

\bigskip


\subsection{The space $\Lip^q(\ve)$}
We wish to obtain distortion estimates by quasiconformal maps in terms of the contents $M^h$, from which
we will derive the corresponding estimates in terms of non linear potentials. To study the distortion in terms of
$M^h$, our arguments below are inspired by the ones of \cite{ACMOU}.

Given $1\leq q<\infty$ and a function $\ve:\cB\to[0,\infty)$, we define $\Lip^q(\varepsilon)$
 as the class of all functions $f:\C\to\C$ for which there is some constant $M$ such that
$$\biggl(\frac{1}{|B|}\int_B |f-f_B|^q\biggr)^{1/q}\leq M\,\ve(B)$$
for all balls $B$. In the sufficient part of the definition, one can replace the average $f_B=|B|^{-1
}\int_B f$ by any constant $c_B$, getting the same class of functions. The infimum of all these constants $M$ is denoted by $\|f\|_{\Lip^q(\varepsilon)}$.

Let us look at the behaviour of a function in $\Lip(\varepsilon) = \Lip^1(\varepsilon)$ under a $K$-quasi\-conformal mapping.

\begin{lem}\label{PVMO}
Let $\ve_0\in\cG_1$ and let  $\phi:\C\rightarrow\C$ be a $K$-quasiconformal mapping. Set
$$\ve(B):= \ve_0(\phi(B))$$
with $a>0$, 
and $h(x,r)=r\,\ve(x,r)$.
Then, given $q$ with $K<q<\infty$, for all $f\in\operatorname{Lip}^q(\ve_0)$, we have  $f\circ\phi\in\operatorname{Lip}(\ve)$ and $$\|f\circ\phi\|_{\operatorname{Lip}(\ve)}\leq C(q,K)\,\|f\|_{\operatorname{Lip}^q(\ve_0)}.$$
\end{lem}

\begin{proof}
We will follow the techniques used in \cite{reimann}. Given a ball $B=B(x,t)$,
we can find a ball $B_0$ centered at $\phi(x)$ such that $B_0\supset \phi(B)$ and $|B|\leq|\phi^{-1}(B_0)|\leq C(K)\,|B|$, where $C(K)$ depends only on $K$. We have:
$$\aligned
\frac{1}{|B|}&\int_B|f\circ\phi(z)-c_B|\,dm(z)=\frac{1}{|B|}\int_{\phi(B)}|f(w)-
c_B|\,J\phi^{-1}(w)\,dm(w)\\
&\leq\frac{1}{|B|}\int_{B_0}|f(w)-c_B|\,J\phi^{-1}(w)\,dm(w)\\
&\leq\,C(K)\,\frac{1}{|B_0|}\int_{B_0}|f(w)-c_B|\,J\phi^{-1}(w)\,dm(w)\,\,\frac{
|B_0|}{|\phi^{-1}(B_0)|}\\
&\leq\,C(K)\,\left(\frac{1}{|B_0|}\int_{B_0}|f(w)-c_B|^{q}dm(w)\right)^{1/q}
\frac{\left(\frac{1}{|B_0|}\int_{B_0}J\phi^{-1}(w)^{q'}\,dm(w)\right)^{1/q'}}{\left(
\frac{1}{|B_0|}\int_{B_0}J\phi^{-1}(w)\,dm(w)\right)}\\
&\leq C(K,q)\,\|f\|_{\Lip^q(\ve_0)}\,\varepsilon_0(B_0),
\endaligned
$$
where the last inequality follows from the fact that the Jacobian satisfies the reverse H\"older inequality 
$$
\left(\frac{1}{|B_0|}\int_{B_0}J\phi^{-1}(w)^{q'}\,dm(w)\right)^{1/q'}\leq
 C(K,q) \frac{1}{|B_0|}\int_{B_0}J\phi^{-1}(w)\,dm(w)$$
for $q'<K/(K-1)$, by \cite[p.37]{astalaiwaniecsaksman}. 

Since $\ve_0\in\cG_1$, we have $\ve_0(B_0)\approx\ve_0(\phi(B))=\ve(B)$, and then
$$
\frac{1}{|B|}\int_B|f\circ\phi(z)-c_B|\,dm(z)\leq\,C(K,q)\,\|f\|_{\Lip^q(\varepsilon_0
)}\,\varepsilon(B).
$$
Thus $\|f\circ \phi\|_{Lip(\ve)}\leq C(K,q) \|f\|_{Lip^q(\ve_0)}$.
\end{proof}

\bigskip


\subsection{The capacities $\gamma_{h,q}$}

Given $1<q<\infty$, for a bounded set $F\subset\C$ and a function $h:\cB\to[0,\infty)$, with $h(x,r)=r\,\ve(x,r)$,
we set
$$\gamma_{h,q}(F)=\sup|\langle\overline\partial f, 1\rangle|=\sup|f'(\infty)|$$
where the supremum is taken over all $\Lip^q(\varepsilon)$ functions with $\|f\|_{
\Lip^q(\varepsilon)}\leq 1$, $f(\infty)=0$ and such that $\overline\partial f$ is 
a distribution supported on $F$.

\begin{lem}\label{CH}
Let $E$ be a compact set and $\ve\in\cG_1$. For $1\leq q<2$ we have
\begin{enumerate}
\item[(a)] $\gamma_{h,q}(E)\leq C\, M^h(E)$.
\item[(b)] If moreover $\ve\in\cG_2$, then $M^{h}(E)\leq C(q)\,\gamma_{h,q}(E)$.
\end{enumerate}
\end{lem}

\begin{proof} First we show (a).
Fix a real number $\eta>0$ and take a covering of $E$ by balls $B_j$, with 
radius $r_j$, such that $\sum_jh(B_j)\leq M^h(E)+\eta$. Consider a partition of unity
 associated to this covering, that is, for each $j$ we take an infinitely differentiable
  function $\varphi_j$ supported on $2B_j$ with 
  $\|\nabla\varphi_j\|_\infty\leq\frac{C}{r_j}$, and so that $\sum_j\varphi_j=1$ on a neighbourhood of $E$. Then, if $\|f\|_{
\Lip^q(\varepsilon)}\leq 1$,
$$\aligned
|\langle\overline\partial f, 1\rangle|
&=|\langle\overline\partial f,\sum_j\varphi_j\rangle|=|\sum_j\langle\overline\partial(f-f_{2B_j}),\varphi_j\rangle|\\
&\leq\sum_j\int_{2B_j}|f-f_{2B_j}|\,|\overline\partial\varphi_j|\,dm\leq\sum_j\frac{
C}{r_j}\int_{2B_j}|f-f_{2B_j}|\,dm\\
&\leq C\,\sum_jr_j\,\varepsilon(2B_j)\leq C\,\sum_jr_j\,\varepsilon(B_j)\leq C\,(M^h(E)+\eta).
\endaligned$$
Hence, $\gamma_{h,q}(E)\leq C\,M^h(E)$.

To prove (b), we suppose that $\ve\in\cG_2$. If $M^{h}(E)>0$ then by Frostman's Lemma
there exists a positive measure $\nu$, supported on $E$, such that $\nu(B(x,r)
)\leq h(x,r)$ and $\nu(E)\geq C\,M^{h}(E)$. The function $f=\nu\ast\frac{1}{z}$ is 
analytic outside $E$, $f(\infty)=0$ and $\langle\overline\partial f, 1\rangle=\nu(E)$. Now, we will check that $f\in\Lip^q(\varepsilon)$. Fix a ball $B=B(z_0,r)$ and $c_B=\int_{\C\setminus 2B}\frac{d\nu(w)}{w
-z_0}$. We have
\begin{align}\label{DecomposeIntegralInsideAndOutside2}
 \frac{1}{|B|}&\int_B|f(z)-c_B|^q\,dm(z)\leq \nonumber \\
&\leq  \frac{1}{|B|}\int_B\left(\int_{2B}\frac{1}{|w-z|}\,d\nu(w)+\int_{\C\setminus 2B}\left|\frac{1}{w-z}-\frac{1}{w-z_0}\right|d\nu(w)\right)^qdm(z).\nonumber \\
\end{align}

For the first term on the right side, by H\"older's inequality and Fubini's Theorem, since $q<2$,
\begin{align*}
\frac{1}{|B|}\int_{B}\biggl(\int_{2B}\frac{1}{|w-z|}\,d\nu(w)\biggr)^q\,dm(z)&\leq 
\frac{\nu(2B)^{q-1}}{|B|}\int_{2B}\int_{B}\frac{1}{|w-z|^q}\,dm(z)\,d\nu(w)\\
&\leq C \frac{\nu(2B)^q}{r^q}\leq
C\,\varepsilon
(2B)^q\leq C\ve(B)^q.
\end{align*}

Since $|w-z|\simeq|w-z_0|$ for $z\in B$ and $w\in\C\setminus 2B$, we have
\begin{align*}
\int_{\C\setminus 2B}\frac{|z-z_0|}{|w-z_0|^2}\,d\nu(w) &
\leq Cr\,\sum_{j=1}^\infty\int_{2^{j+1}B\setminus 2^jB}\frac{d\nu(w)}{|w-z_0|^2}\\
&\leq Cr\,\sum_{j=1}^\infty\frac{h(z_0,2^{j+1}r)}{(2^jr)^2}
\leq C\,\sum_{j=0}^\infty\frac{\varepsilon(2^jB)}{2^j}\leq C\ve(B),
\end{align*}
using the fact that $\ve\in\cG_2$.
Thus we get
$$\frac{1}{|B|}\int_B\left(\int_{\C\setminus 2B}\left|\frac{1}{w-z}-\frac{1}{w-z_0}\right|d\nu(w)\right)^qdm(z)\leq C\ve(B)^q,
$$
and so (b) follows.
\end{proof}

Recall that a quasiconformal mapping $\phi:\C\to\C$ is called principal if it is conformal
outside a compact set and $|\phi(z) - z| = O(1/|z|)$ as $z\to\infty$.

\begin{lem}\label{VMO}
Let $E$ be a compact set, and $\phi:\C\rightarrow\C$ a principal $K$-quasiconformal mapping, conformal on $\C\setminus E$. Given $\ve_0\in\cG_1$,
define
$$\ve(x,r) = \ve_0(\phi(B(x,r))$$
and $h(x,t)=r\,\ve(x,r)$. 
For $q>K$, we have
$$\gamma_{h_0,q}(\phi(E))\leq C\,\gamma_{h,1}(E).$$
\end{lem}

\begin{proof}
Consider $f\in\Lip^q(\varepsilon_0)$ which is analytic in $\C\setminus\phi(E)$, 
$\|f\|_{\Lip^q(\varepsilon_0)}\leq 1$ and $f(\infty)
=0$. Set $g=f\circ\phi$. 
Then $g$ is analytic on $\C\setminus E$ and, by Lemma \ref{PVMO}, 
$g\in\Lip^1(\varepsilon)$ and $\|g\|_{\Lip^1(\varepsilon)}\leq C(K)\,\|f\|_{\Lip^q(\varepsilon_0)}\leq C(K)$. 
So, we have $|g'(\infty)|\leq C(K)\,\gamma_{h,1}(E)$. Moreover, since $\phi$ is
principal, $\phi'(\infty)=1$ and so
$$|g'(\infty)|=|f'(\infty)|\,|\phi'(\infty)|=|f'(\infty)|.$$
Consequently $\gamma_{h_0,q}(\phi(E))\leq C(K)\,\gamma_{h,1}(E)$.
\end{proof}

\bigskip


\subsection{Distortion of $h$-contents}

From Lemmas \ref{CH} and \ref{VMO} we get:

\begin{lem}\label{lemkpetit}
Let $E$ be a compact set, and $\phi:\C\rightarrow\C$ a principal $K$-quasiconformal mapping, with $K<2$, conformal on $\C\setminus E$. Given $\ve_0\in\cG_2$,
define
$$\ve(x,r) = \ve_0(\phi^{-1}(B(x,r))$$
and $h(x,r)=r\,\ve(x,r)$. We have
$$M^{h_0}(E)\leq C\,M^{h}(\phi(E)).$$
\end{lem}

\begin{proof}Take $q$ such that $K<q<2$.
By Lemma \ref{CH}, we have $M^{h_0}(E)\leq C\gamma_{h_0,q}(E)$. By Lemma
\ref{VMO} (applied to $\phi^{-1}$), $\gamma_{h_0,q}(E)\leq C \gamma_{h,1}(\phi(E))$.
Finally, by
Lemma \ref{CH} again, $\gamma_{h,1}(\phi(E))\leq C M^{h}(\phi(E))$.
\end{proof}

Our next objective in this section is to extend Lemma \ref{lemkpetit} to the case $K\geq2$. 

\begin{lemma} \label{lemkgran}
Let $\ve:\cB\to[0,\infty)$ be a function from $\cG_1$, and set
$h(x,r)=r\,\ve(x,r)$. Suppose that for any principal $K_0$-quasiconformal mapping $\phi:\C\to\C$ 
conformal on $\C\setminus E$, with $E$ compact, and with $K_0\leq K$,
the function $\ve_\phi:\cB\to[0,\infty)$ defined by $\ve_\phi(B)=\ve(\phi^{-1}(B))$ 
is in $\cG_2$. Then,  
$$M^h(E)\leq C(K)\,M^{h_\phi}(\phi(E))$$
for any compact set $E\subset \C$, where $h_\phi(x,r)=r\,\ve_\phi(x,r)$.
\end{lemma}

\begin{proof}
We factorize $\phi$ so that $\phi = \phi_n\circ\cdots \circ\phi_1$, where
$\phi_i$ are $K^{1/n}$-quasiconformal mappings conformal on $\C\setminus \phi_{n-1}(E)$, with $n$ big enough so that $K^{1/n}<2$.
So we have
$$E = E_0 \stackrel{\phi_1}{\longrightarrow} E_1
\stackrel{\phi_2}{\longrightarrow} \ldots \stackrel{\phi_{n-1}}{\longrightarrow} E_{n-1} \stackrel{\phi_n}{\longrightarrow} E_n
=\phi(E).$$
By Lemma \ref{lemkpetit}, we have
$$M^{h}(E) = M^h(E_0)\leq C\,M^{h_1}(\phi_1(E_0)) = C\,M^{h_1}(E_1),$$
where $\ve_1(B) =\ve(\phi_1^{-1}(B))$ (notice that $\ve\in \cG_2$ by hypothesis).

Denote now $\ve_2 (B) = \ve_1(\phi_2^{-1}(B))=\ve(\phi_1^{-1}(\phi_2^{-1}(B)))$ and $h_2(x,r)=r\, \ve_2(x,r)$.
Since $\ve_1\in\cG_2$ by the hypotheses above, 
by Lemma \ref{lemkpetit} again, $$M^{h_1}(E_1)\leq C M^{h_2}(E_2).$$
Going on in this way, after $n$ steps we obtain
$$M^h(E_0)\leq C M^{h_1}(E_1)\leq  \cdots \leq
C M^{h_n}(E_n),$$
with $E=E_0$, \, $E_n= \phi(E)$,\, $h_n(x,r)=r\,\ve_n(x,r)$, and
$$\ve_n(B) = \ve(\phi_1^{-1}(\phi_2^{-1}(\cdots \phi_n^{-1}(E)))) \approx 
\ve(\phi^{-1}(B)) = \ve_\phi(B).$$ 
\end{proof}

\bigskip


\subsection{Proof of Theorem \ref{teocap} (b)}

Recall that $\mu$ is a Borel measure supported on $E$ such that  $\dot W_{1/p,p}^\mu(x)\leq 1$ for all $x\in\C$ and such that $\dot C_{1/p,p}(E)\approx\mu(E)$.
We know that $M^{h_{\mu,a}}(E)\geq C^{-1}\mu(E)$, with $h_{\mu,a}$ defined in \rf{defhx} and $a$ small enough. Set 
$$\ve(x,r) = \ve_{\mu,a}(\phi^{-1}(B(x,r)))$$
and $h(x,r)=r\,\ve(x,r)$. 
By Lemma \ref{lemtec5}, $\ve\circ\phi^{-1}\in\cG_2$ for any $K_0$-quasiconformal
mapping such that $K_0\leq K$. Then,
 by Lemma \ref{lemkgran} we have
$$M^{h_{\mu,a}}(E)\leq C\,M^h(\phi(E)).$$
 By Frostman's Lemma there exists
some measure $\nu$ supported on $\phi(E)$ such that $\nu(\phi(E))\geq C^{-1}
M^{h}(\phi(E))$ with $\nu(B)\leq {h}(B)$ for all balls $B$.  
Recall that
$$\int_0^\infty \ve_{\mu,a}(x,r)^{p'-1}\frac{dr}r\lesssim 1 \quad\mbox{ for all $x\in\C$,}$$
by \rf{ed**}. From Lemma \ref{lemcg} we deduce that this also holds with $\ve$ instead
of $\ve_{\mu,a}$, and thus
$$\int_0^\infty \biggl(\frac{\nu(B(x,r))}{r}\biggr)^{p'-1}
\,\frac{dr}r \leq \int_0^\infty \biggl(\frac{h(x,r)}{r}\biggr)^{p'-1}
\,\frac{dr}r \leq C.$$
In terms of Wolff's potentials, this is the same as saying that
$$\dot W^\nu_{1/p,p}(x)\leq C$$
for all $x\in\C$.
Therefore, 
$$\dot C_{1/p,p}(\phi(E)) \gtrsim \nu(\phi(E)) \gtrsim
M^{h}(\phi(E))
\gtrsim M^{h_{\mu,a}}(E)\gtrsim \mu(E)
\gtrsim \dot C_{1/p,p}(E).$$
\fiproof

\bigskip

\subsection{The main lemma on $h$-contents}

\begin{lemma}\label{mainlem}
  Let $\ve_0\in\cG_1$ and set $h_0(x,r)=r\,\ve_0(x,r)$.
  Suppose that for any $K_0$-quasiconformal mapping $\psi:\C\to\C$, with $K_0\leq K$, we have
  $\ve_0\circ\psi \in \cG_2$.
Let $E\subset B(0,1/2)$ be compact
and $\phi:\C\to\C$ a principal $K$-quasiconformal mapping,
conformal on $\C\setminus \bar \D$.
Denote 
\begin{equation}\label{defhed}
\ve(x,r):= \ve_0(\phi^{-1}(B(x,r)))^{2K/(K+1)},\qquad
h(x,r):= r^{2/(K+1)}\ve(x,r).
\end{equation}
Then we have
$$M^{h_0}(E) \leq C(K) \,M^h(\phi(E))^{(K+1)/2K}.$$
\end{lemma}

\begin{proof}
Consider an arbitrary covering $\phi(E)\subset\bigcup_i B_i$ by a finite number of balls $B_i:=B(x_i,t_i)$
(recall that $\phi(E)$ is compact).
For each $i$, take also a ball $D_i$ centered at $\phi^{-1}(x_i)$ which contains
$\phi^{-1}(B_i)$ and which has comparable diameter.

 We denote 
$\Omega = \bigcup_i D_i$. Notice that $E\subset \Omega$. Then we consider the decomposition
$\phi=\phi_2\circ\phi_1$, where $\phi_1$, $\phi_2$ are principal $K$-quasiconformal
mappings. Moreover, we require $\phi_1$ to be conformal outside $\overline{\Omega}$ and $\phi_2$ conformal 
in $\phi_1(\Omega) \cup (\C
\setminus \D)$.

By Lemma \ref{lemkgran}, we have
$$M^{h_0}(E)\leq M^{h_0}(\Omega)\leq C\,M^{\wt h}(\phi_1(\Omega)),$$
with 
$$\wt h(x,r)= r\,\wt\ve(x,r):= r\,\ve_0(\phi_1^{-1}(B(x,r))).$$

Now we will estimate
$M^{\wt h}(\phi_1(\Omega))$ in terms of $M^h(\phi(E))$. For each $i$,
let $\wt D_i$ be a ball centered at $\phi_2^{-1}(x_i)$ containing $\phi_1(D_i)$
and such that $\diam(\wt D_i)\approx \diam(\phi_1(D_i))$. Notice that we also have
\begin{equation}\label{eqsa1}
B_i\subset \phi_2(\wt D_i)\qquad\mbox{and}\qquad \diam(B_i)\approx \diam(\phi_2(\wt D_i)).
\end{equation}
By Vitali's covering lemma there exists a subfamily of disjoint balls $\wt D_j$, $j\in
J$, 
such that $\phi_1(E)\subset \cup_j 5\wt D_j$. Applying H\"older's inequality twice (once for the intergral and once for the sum),
the fact that $J(\phi_2^{-1})\in L^{K/(K-1)}(\phi(\Omega))$ because of the improved borderline integrability of the Jacobian $J(\phi_2^{-1})$ under the assumption that 
$\phi_2:\phi_1(\Omega)\to\phi(\Omega)$ is conformal (by \cite[Lemma 5.2]{astalanesi}), 
and \eqref{eqsa1} (these last two facts yield the last inequality), we obtain
\begin{align*}
M^{\wt h}(\phi_1(\Omega)) &
\leq \sum_{j\in J} \wt h(5\wt D_j)\leq C \sum_{j\in J} \wt h(\wt D_j) = C \sum
_{j\in J}\diam(\phi_1(D_j))\, \wt \ve(\wt D_j)\\
& \leq C \sum_{j\in J} \biggl(\int_{\phi(D_j)} J(\phi_2^{-1})\,dm\biggr)^{1/2}\wt\ve(\wt D_j)\\
&\leq 
 C \sum_{j\in J} \biggl(\int_{\phi(D_j)} J(\phi_2^{-1})^{K/(K-1)}\,dm\biggr)^{(K-1)/2K}
 \diam(\phi(D_j))^{1/K} \,\wt\ve(\wt D_j)\\
 & \leq 
  C \biggl(\sum_{j\in J} \int_{\phi(D_j)} J(\phi_2^{-1})^{K/(K-1)}\,dm\biggr)^{(K-1)/2K}\\
  &\quad \times
  \biggl(\sum_{j\in J}
 \diam(\phi(D_j))^{2/(K+1)} \,\wt\ve(\wt D_j)^{2K/(K+1)}\biggr)^{(K+1)/2K}\\
& \leq C \biggl(\sum_{j\in J}
 \diam(B_j)^{2/(K+1)} \,\wt\ve(\wt D_j)^{2K/(K+1)}\biggr)^{(K+1)/2K} \ .
\end{align*}
Notice now that
$$\wt\ve(\wt D_j) = \ve_0(\phi_1^{-1}(\wt D_j))\approx\ve_0(D_j)\approx\ve_0(\phi^{-1}(B_j)).$$
Recalling that 
$$\ve(x,t):= \ve_0(\phi^{-1}(B(x,t)))^{2K/(K+1)}\quad \mbox{ and }\quad h(x,t)= t^{2/(K+1)}\ve(x,t),$$
 we deduce
$$
M^{h_0}(E)\leq C M^{\wt h}(\phi_1(\Omega))  \leq C
\biggl(\sum_{j\in J} h(B_j)\biggr)^{(K+1)/2K}.
$$
If we take the infimum over all coverings of $\phi(E)$ by balls $B_j$, 
then we get
$$M^{h_0}(E)\leq C M^h(\phi(E))^{\frac{K+1}{2K}}.$$
\end{proof}

\bigskip


\subsection{Proof of Theorem \ref{teocap} (a)}

By standard methods (see e.g. \cite{Lacey-Sawyer-Uriarte}, p.289, proof of lemma 2.1), we may assume that $\phi$ is a principal quasiconformal mapping,
conformal on $\C\setminus \bar \D$,
and that $E\subset B(0,1/2)=:\frac12 B$ (and so $\diam(\phi(B))\approx 1$).

Let 
$\mu$ be a Borel measure supported on $E$ such that  $\dot W_{1/p,p}^\mu(x)\leq 1$ for all $x\in\C$ and such that $\dot C_{1/p,p}(E)\approx\mu(E)$.
We know that $M^{h_{\mu,a}}(E)\geq C^{-1}\mu(E)$. If $0<a<1$ is small enough,
then $\ve_0:=\ve_{\mu,a}$ satisfies the assumptions of Lemma \ref{mainlem}, and so
$$M^{h_{\mu,a}}(E) \leq C(K) M^h(\phi(E))^{(K+1)/2K},$$
with $h$ given by \eqref{defhed} (replacing $\ve_0$ there by $\ve_{\mu,a}$). 
By the definition of $\ve$ and Lemma~\ref{lemcg},
\begin{align*}
\int_0^\infty \ve(x,r)^{\frac{(p'-1)(K+1)}{2K}}
\,\frac{dr}r 
& = \int_0^\infty\ve_{\mu,a}(\phi^{-1}(B(x,r)))^{p'-1} \,\frac{dr}r \\& \leq C
\int_0^\infty\ve_{\mu,a}(\phi^{-1}(x),r)^{p'-1} \,\frac{dr}r
\leq C
\end{align*}
for all $x\in\C$.

Now we apply Frostman's lemma again, and we deduce that there exists
some measure $\nu$ supported on $\phi(E)$ such that $\nu(\phi(E))\geq C^{-1}
M^h(\phi(E))$ with $\nu(B)\leq h(B)$ for all balls $B$. 
So we have
$$\int_0^\infty \biggl(\frac{\nu(B(x,r))}{r^{\frac2{K+1}}}\biggr)^{\frac{(p'-1)(K+1)}{2K}}
\,\frac{dr}r \leq C$$
for all $x\in\C$.
In terms of Wolff's potentials, this is the same as saying that
$$\dot W^\nu_{\alpha,q}(x)\leq C$$
for all $x\in\C$, with
$$\alpha= \frac{2K}{2Kp-K+1},\qquad q=\frac{2Kp-K+1}{K+1}.$$
Therefore, 
$$\dot C_{\alpha,q}(\phi(E)) \gtrsim \nu(\phi(E)) \gtrsim \dot C_{1/p,p}(E)^{\frac{2K}{K+1}}.$$
\fiproof


\bigskip

\section{Strategy for the proof of Theorem \ref{distorgamma}}\label{secstrat}

Sections 3-7 are devoted to the proof of Theorem \ref{distorgamma}. 
An equivalent way of formulating this theorem consists in saying that if 
 $E\subset B(0,1/2)$ is compact and $\phi:\C\to\C$ a principal $K$-quasiconformal mapping,
 conformal on $\C\setminus \bar B(0,1)$,
 then 
\begin{equation}\label{eqdistor}
\dot C_{ \frac{2K}{2K+1},\frac{2K+1}{K+1}}(\phi(E))  \geq c^{-1} \gamma(E)^{\frac{2K}{K+1}},
\end{equation}
by appropriate normalizations.
To prove this result we will use the following tools:
\begin{itemize}
\item the characterization of analytic capacity in terms of
curvature in Theorem~A,
\item a corona type decomposition for measures with finite curvature analogous 
to the one used in \cite{tolsabilip} to study the behavior of analytic capacity under bilipschitz maps,
\item the main Lemma \ref{mainlem} on the distortion of $h$-contents under quasiconformal
maps,
\item improved quantitative estimates for the distortion of sub-arcs of chord arc curves.
\end{itemize}

\bigskip
Let us describe the arguments to prove \rf{eqdistor} in more detail.
Given, $E\subset\C$ with $\gamma(E)>0$, let $\mu$ be a measure supported on $E$
such that $\mu(E)\approx\gamma(E)$, $\mu(B(x,r))\leq r$ for all $x\in\C$, $r>0$, and $c^2_\mu(x)\leq 1$ for all $x\in\C$. 
As in the preceding section, for each $a>0$ we construct the measure $\H^{h_a}$ associated to $\mu$, with $h_a(x,t)=t\,\ve_a(x,t)$, where 
$$\ve_a(x,t) = \frac1t \int \psi_a\Bigl(\frac{y-x}t\Bigr)d\mu(y),$$
and $\psi_a$ is defined as in \rf{eqpsia}. To simplify the notation, now 
we will write $\ve_a$ and $h_a$ instead of $\ve_{\mu,a}$ and $h_{\mu,a}$.
The main Lemma \ref{mainlem} on the distortion
of $h$-contents tells us that
$$M^h(\phi(E))\gtrsim M^{h_a}(E)^{\frac{2K}{K+1}}\gtrsim \mu(E)^{\frac{2K}{K+1}},$$
where $h$ is the gauge function defined by
$$ h(x,t):= t^{2/(K+1)}\ve(x,t),\quad \ve(x,t):= \ve_a(\phi^{-1}(B(x,t)))^{2K/(K+1)},$$
with $a>0$ small enough.

By Frostman's Lemma we deduce that there exists a measure $\nu$ supported on $\phi(E)$ satisfying $\nu(\phi(E))\approx M^h(\phi(E))$ and $\nu(B(x,r))\leq h(x,t)$.
However, from the last estimate we cannot infer that
\begin{equation}\label{eqclau8}
\dot W^\nu_{ \frac{2K}{2K+1},\frac{2K+1}{K+1}}(x)\leq C\quad \mbox{ for all $x\in \C$},
\end{equation}
as in the proof of Theorem \ref{teocap}, because now the estimate
$$\dot W^\mu_{2/3,3/2}(x)\leq C\quad \mbox{ for all $x\in \C$}$$
may be false.

To obtain a measure $\nu$ supported on $\phi(E)$ satisfying \rf{eqclau8} we will use
the information on the curvature of $\mu$. Indeed, by \cite[Main Lemma 3.1]{tolsabilip}, 
there exists some collection of squares $\ttt(\mu)$ such that
\begin{equation}\label{eqcor90}
\sum_{Q\in \ttt(\mu)} \theta_\mu(Q)^2 \mu(Q) 
\leq C \biggl(\mu(E)+\int c^2_\mu(x)\,d\mu(x)\biggr)\leq C\mu(E),
\end{equation}
where $\theta_\mu(Q)=\mu(Q)/\ell(Q)$ (here $\ell(Q)$ stands for the side length of $Q$), and the last inequality is a consequence of the normalization $c^2_\mu(x)\leq 1$ for all $x\in\C$. 
For each square $Q\in \ttt(\mu)$ there exists some chord arc curve $\Gamma_Q$ (or a fixed finite number of chord arc curves) satisfying some precise properties. 
Roughly speaking, if a dyadic square $P$  intersects $E$ and $\ell(P)\leq \diam(E)$, then it belongs to some ``tree'' with ``root'' $Q\in\ttt(\mu)$ and $P$ is close to the curve $\Gamma_Q$. For more precise information, see \cite{tolsabilip}.

It is easy to check (but we will not need it for the proof, see however Lemma \ref{lempac0} and Section \ref{secconsnu} which contains the actual proof), that \rf{eqcor90} implies that
$$\sum_{Q\in \ttt(\mu)} \ve_a(Q)^2 \mu(Q) \leq C \mu(E).$$
By Tchebytchev, we infer that for all $x$ in a subset $E_0\subset E$ with $\mu(E_0)\geq \mu(E)/2$,
$$\sum_{Q\in \ttt(\mu):\,x\in Q} \ve_a(Q)^2 \leq C.$$
Arguing as in the preceding section (but again we will not need this for the proof, Section \ref{secconsnu} contains the actual proof), this implies that
\begin{equation}\label{eqsx1}
\sum_{Q\in \phi(\ttt(\mu)):x\in Q} \biggl(\frac{h(Q)}{\ell(Q)^{2/(K+1)}}\biggr)^{\frac{K+1}{K}} \leq C
\end{equation}
for all $x\in\phi(E_0)$.
By Frostman's Lemma, we deduce that there exists a measure $\nu$ supported on $\phi(E)$ with $\nu(2Q)\leq h(2Q)\lesssim h(Q)$ for all the squares $Q$, and so
$$\sum_{Q\in \phi(\ttt(\mu)):x\in Q} \biggl(\frac{\nu(2Q)}{\ell(Q)^{2/(K+1)}}\biggr)^{\frac{K+1}{K}} \leq C.$$
In this inequality, if instead of summing over all the squares $Q\in \phi(\ttt(\mu))$ containing $x$ we summed over all $Q\in\phi(\cD)$ containing $x$, then we would
obtain \rf{eqclau8}, and thus
$$
\dot C_{ \frac{2K}{2K+1},\frac{2K+1}{K+1}}(\phi(E))  \gtrsim \nu(\phi(E)) \gtrsim \mu(E_0)^{\frac{2K}{K+1}}\gtrsim \gamma(E)^{\frac{2K}{K+1}}.$$

In a sense, to extend the sum in \rf{eqsx1} from the squares in $\phi(\ttt)$ to the entire collection of $Q\in\phi(\cD)$,
we can use the geometric properties of the corona decomposition (i.e. 
different scales). To be able to use this information, we have to obtain improved distortion estimates for
subsets of chord arc curves, in a more quantitative way
 than the ones of \cite[Section 3]{ACMOU} for rectifiable sets. This is what we do in next section. 
 
To tell the truth, in the arguments above, when we apply Tchebytchev to obtain the
subset $E_0\subset E$, some of the delicate properties of the corona decomposition for $\mu$
are destroyed, and so we will follow a somewhat different approach, although similar
in spirit to the one outlined above. Because of this reason, we will need to obtain
a corona decomposition for $\mu$ slightly different to the one in \cite{tolsabilip}. 
We carry out this task in Section \ref{seccorona0}. The required measure $\nu$ is
constructed in Section \ref{secconsnu}. A direct application of Frostman Lemma is not
enough, and we will have to use a more sophisticated argument more adapted to the corona
decomposition. Finally, in Section \ref{secproof} we prove that the key estimate \rf{eqclau8} holds for $\nu$.


\bigskip
\section{Distortion of sub-arcs of chord arc curves} \label{secarcs}

Our arguments are inspired by the ones used in \cite{ACMOU} to obtain improved
distortion results for rectifiable sets. However, we need more precise
quantitative estimates. Some notation we will use is as follows. The length of an interval $I \subset \partial\D$ is denoted by $\ell(I)$. The side length of a square $Q$ is denoted by $\ell(Q)$. Given a bijective mapping $\phi:\C\to\C$ and a square $Q$, one defines the side length of $\phi(Q)$ as $\ell(\phi(Q)) := \diam(\phi(Q))$. Analogously, $\ell(\phi(I)) := \diam(\phi(I))$.

\begin{lemma}\label{lemdistca1}
Let $\ve>0$ and let $\phi:\C\to\C$ be a $(1+\ve)$-quasiconformal mapping which is conformal on $\D$, such that $\phi'(0)=1$. 
Denote $\alpha_0=1-c_0\ve^2$, where $c_0$ will be a sufficiently large constant.
Let $\{I_n\}_n\subset\partial\D$ be a collection of pairwise disjoint dyadic intervals. 
\begin{itemize}
 \item[(a)] If $c_0\geq 40$, we have
$$\sum_n \ell(\phi(I_n))^{\alpha_1} \leq C\Bigl(\sum_n \ell(I_n)^{\alpha_0}\Bigr)^b,$$
where $\alpha_1=1-\frac14 c_0 \ve^2$, and $b>0$ depends only on $c_0$; and $C$ 
on $c_0$ and~$\ve$.
\item[(b)] If 
$$\sum_n \ell(I_n)^{\alpha_0}\geq \delta,$$
then
$$\sum_n \ell(\phi(I_n))^{\alpha_2}\geq\delta',$$
where $\alpha_2=1-(2c_0+10)\ve^2$ and $\delta'>0$ depends on $\delta, c_0, \ve$.
\end{itemize}
\end{lemma}

\begin{proof}
{\bf (a)}
We drop the factor $2\pi$ in all calculations. Let $\cD_j$ be the collection of the dyadic intervals of length $2^{-j}$ of $\partial \D$, and set $\{I_n\}=\{I_n^j\}_{j,n}$, with $I_n^j\in \cD_j$. 
Consider Whitney squares $\{Q_n^j\}_{j,n}\subset \D$ so that $\ell(Q_n^j)\approx \ell(I_n^j)\approx\dist(Q_n^j,I_n^j)$. Denote by
$z_n^j$ the center of $Q_n^j$. By Koebe's distortion theorem, we have
\begin{equation}\label{eqfr65}
\ell(\phi(Q_n^j))\approx |\phi'(z_n^j)|\ell(Q_n^j) \approx |\phi'(z_n^j)|\,(1-|z_n^j|) \approx |\phi'(z)|\,(1-|z|),
\end{equation}
for all $z\in Q_ n^j$.
Denoting $\ell_j=\ell(Q_n^j)=2^{-j}$, $r_j= 1-\ell_j$, and $N_j = \# \{I_n^j\}_n$, using H\"older's inequality we get
\begin{align*}
\ell_j \sum_{n=1}^{N_j} \ell(\phi(I_n^j))^{\alpha_1} & \approx \ell_j \sum_{n=1}^{N_j} \ell(\phi(Q_n^j))^{\alpha_1}\lesssim \int_{\bigcup_n I_n^j}\ell(I_n^j)^{\alpha_1} \,|\phi'(r_j\,e^{it})|^{\alpha_1}\,dt\\
& = \ell_j^{\alpha_1} \int_{\bigcup_n I_n^j} |\phi'(r_j\,e^{it})|^{\alpha_1}\,dt\\
& \leq N_j^{1/p'}\,\ell_j^{\alpha_1+\frac1{p'}} \biggl[\int_{\bigcup_n I_n^j} |\phi'(r_j\,e^{it})|^{\alpha_1 p}\,dt\biggr]^{1/p}
\end{align*}
for $1<p<\infty$. Since $\phi$ is $(1+\ve)$-quasiconformal,  we have the following estimate for the integral means:
$$\int_{\bigcup_n I_n^j} |\phi'(r_j\,e^{it})|^q\,dt \leq \frac{C_\beta}{\ell_j^{\beta}},$$
with $\beta>\beta(q)$, for any $q \in \R$, by definition of $\beta(q)$ (see e.g Chapter 8 in \cite{pommerenke}). Recall also (e.g.\cite[page 182]{pommerenke}) that, for any $q \in \R$,
$$\beta(q) \leq 9\biggl(\frac{K-1}{K+1}\biggr)^2\,q^2.$$
So if we choose $\beta =9\ve^2q^2,$ we get
$$\sum_{n=1}^{N_j} \ell(\phi(I_n^j))^{\alpha_1} \lesssim N_j^{1/p'}\,\ell_j^{\alpha_1-1+\frac1{p'} -9\ve^2\alpha_1^2p}.$$
Replacing
$N_j = \dfrac{1}{\ell_j^{\alpha_0}}\sum_{n= 1}^{N_j}\ell(I_n^j)^{\alpha_0},$
we obtain
\begin{equation}\label{eq245}
\sum_{n=1}^{N_j} \ell(\phi(I_n^j))^{\alpha_1} \lesssim
\biggl(\sum_{n= 1}^{N_j}\ell(I_n^j)^{\alpha_0}\biggr)^{1/p'}
\ell_j^{\alpha_1-1+\frac1{p'} -\frac{\alpha_0}{p'} - 9\ve^2\alpha_1^2p}.
\end{equation}
Since $\alpha_1\leq1$, if we set $\alpha_0=1-c_0\ve^2$ and $\alpha_1=1-c_1\ve^2$, we get
$$\alpha_1-1+\frac1{p'} -\frac{\alpha_0}{p'} - 9\ve^2\alpha_1^2p \geq
\alpha_1-1+\frac1{p'} -\frac{\alpha_0}{p'} - 9\ve^2 p \geq \ve^2(\frac{c_0}{p'}-c_1-9p) =:a.$$
Since $c_0\geq 10$, we can choose $p\in(1,\infty)$ and $c_1>0$ such that
$$\frac{c_0}{p'}-c_1-9p>0,$$
and so $a>0$ (e.g. $p=2$, $c_1 = \frac{c_0}{4}$, and $c_0 \geq 40$). By \eqref{eq245} and H\"older's inequality we get
\begin{align*}
\sum_j \sum_n \ell(\phi(I^j_n))^{\alpha_1} \lesssim \sum_j
\biggl(\sum_{n= 1}^{N_j}\ell(I_n^j)^{\alpha_0}\biggr)^{1/p'}
\ell_j^a\leq \biggl(
\sum_j \sum_{n= 1}^{N_j}\ell(I_n^j)^{\alpha_0}\biggr)^{1/p'}
\biggl(\sum_j \ell_j^{ap}\biggr)^{1/p}.
\end{align*}
Since $a>0$, we have $\sum_j \ell_j^{ap}\leq C(c_0,\ve)$, and the statement (a) in the lemma follows.

\bigskip
\noindent {\bf (b)} We use the same notation as in (a). Let $\ell_{max}=\max_n \ell(I_n)$, and denote
$$Z_j = \{I_n^j: |\phi'(z_n^j)|\leq \ell(I_n^j)^\gamma\},$$
where $\gamma>0$ is some small constant to be chosen below. Then we have
$$\ell_j^{1-\gamma}\,\# Z_j \lesssim \int_{\bT} \frac{dt}{|\phi'(r_j e^{it})|}\leq \frac{C(\beta)}{\ell_j^{\beta}},$$
for $\beta>\beta(-1)$.
So we infer that
$$\sum_{I\in Z_j} \ell(I)^{\alpha_0} =
\ell_j^{\alpha_0}\,\# Z_j \leq C(\beta)\ell_j^{\gamma - \beta +\alpha_0-1}.
$$
Assuming that
\begin{equation}\label{eqhip4}
 \gamma - \beta +\alpha_0-1 >0,
\end{equation}
summing on $j\geq0$ and setting $Z=\bigcup_j Z_j$, we get
$$\sum_{I\in Z} \ell(I)^{\alpha_0} \leq C(\beta)\sum_{j\geq0} \ell_j^{\gamma - \beta +\alpha_0-1} \leq C(\beta,\gamma,\ve)\, \ell_{max}^{\gamma - \beta +\alpha_0-1}.$$
Therefore, if $\ell_{max}$ is small enough (depending on $\beta,\gamma,\delta,\ve$) we infer that
$\sum_{I\in Z}\ell(I)^{\alpha_0}\leq \frac\delta2,$
and so
$$\sum_{I\not\in Z}\ell(I)^{\alpha_0}\geq \frac\delta2.$$

For the intervals $I\not\in Z$ we use \eqref{eqfr65}, and we obtain
$$\ell(I)\approx \frac1{|\phi'(z_I)|}\,\ell(\phi(I)) \leq \frac{\ell(\phi(I))}{\ell(I)^\gamma},$$
where $z_I=z_n^j$ if $I= I_n^j$. We deduce
\begin{equation}\label{eqft56}
\frac\delta2 \leq \sum_{I\not\in Z} \ell(I)^{\alpha_0} \leq C \sum_{I\not\in Z}\ell(\phi(I))^{\alpha_0/(1+\gamma)}.
\end{equation}
Therefore, (b) holds if $\ell_{\max}$ is small enough and 
we choose $\beta$ and $\gamma$ such that \eqref{eqhip4} is true, that is, if 
\begin{equation}\label{eqhip43}
\gamma>\beta + c_0\ve^2
\end{equation}
Using the estimate (see e.g. \cite[page 182]{pommerenke})
$$\beta(-q) \leq 9\biggl(\frac{K-1}{K+1}\biggr)^2\,q^2,$$
we derive $\beta(-1)\leq 9 \ve^2$. Thus, \rf{eqhip43} holds if we choose
$$\gamma = (10 + c_0)\ve^2,$$
say. Then we have
$$\frac{\alpha_0}{1+\gamma} = \frac{1-c_0\ve^2}{1+(10+c_0)\ve^2} \geq (1-c_0\ve^2)(1- (10+c_0)\ve^2)
\geq 1- (10+2c_0)\ve^2.$$
 From \rf{eqft56} we deduce
$$C \delta \leq \sum_n \ell(\phi(I_n))^{1- (10+2c_0)\ve^2},$$
if $\ell_{max}$ is small enough, i.e. if $\ell_{max}\leq l_0$, where $l_0$ is some constant depending on $c_0,\ve, \delta$.

The case where $\ell_{max}$ is not small follows easily from the preceding estimates. Indeed, let $\cF$ be the family of dyadic intervals
obtained by splitting  each interval $I_n$ into $2^N$ pairwise disjoint dyadic intervals, with $N$ big enough so that each interval from
$\cF$ has length smaller than $l_0$. If we have
$$I_n = I_1' \cup \ldots \cup I_{2^N}',$$
with $I_j'\in \cF$, then we get
$$\ell(I_n)^{\alpha_0} \leq \sum_{j=1}^{2^N}\ell(I_j')^{\alpha_0},$$
and thus,
$\delta \leq \sum_{I\in \cF} \ell(I)^{\alpha_0}.$
So we infer that
$$C \delta \leq \sum_{I\in \cF} \ell(\phi(I))^{1- (10+2c_0)\ve^2}\leq 2^N \sum_n \ell(\phi(I_n))^{1- (10+2c_0)\ve^2},$$
with $N$ depending on $c_0,\ve, \delta$.
\end{proof}

\begin{lemma}\label{lemdistca2}
Let $\ve>0$ and let $\phi:\C\to\C$ be a $(1+\ve)$-quasiconformal mapping. 
Denote $\alpha_0=1-c_0\ve^2$.
Let $\{I_n\}_n\subset\partial\D$ be a collection of pairwise disjoint dyadic intervals 
such that
$$\sum_n \ell(I_n)^{\alpha_0}\geq \delta_0 .$$
Then we have
$$\sum_n \ell(\phi(I_n))^{\alpha}\geq\delta\,\diam(\phi(\D))^\alpha,$$
where $\alpha=1-C \ve^2$ (and $C$ depends on $c_0$), and $\delta>0$, which depends on $\delta_0, c_0, \ve$ (i.e. $\delta = \delta (\delta_0, c_0, \ve)$).
\end{lemma}

\begin{proof}
The lemma follows by combining (a) and (b) in the preceding lemma: arguing as
in \cite{ACMOU}, we write $\phi=f\circ g^{-1} \circ h$, so that $f,g,h$ are
$(1+C\ve)$-quasiconformal and moreover
$h$ is principal and conformal on $\C\setminus\D$ (and so $\diam(h(\D))\approx1$), $f,g$ are conformal on $\D$, and $f(\D) =
\phi(\D)$ and $g(\D)= h(\D)$. So
$$\D \stackrel{h}{\longrightarrow} h(\D)
\stackrel{g^{-1}}{\longrightarrow} \D \stackrel{f}{\longrightarrow}\phi(\D).$$
From (b) in Lemma \ref{lemdistca1} we infer that
$$\sum_n \ell(h(I_n))^{\alpha'}\geq\delta',$$
with $\alpha' = 1-C'\ve^2$. Here $C' = C'(c_0)$, and $\delta' = \delta' (\delta_0, c_0, \ve)$. By (a) in the same lemma we get
$$\sum_n \ell(g^{-1}\circ h(I_n))^{\alpha''}\geq\delta'',$$
with $\alpha'' = 1-C''\ve^2$. Here $C'' = C''(C')$, and $\delta'' = \delta'' (\delta', C', \ve)$. And by (b) again,
$$\sum_n \ell(f\circ g^{-1}\circ h(I_n))^{\alpha'''}\geq\delta'''\diam(\phi(\D))^{\alpha'''},$$
where $\alpha''' = 1-C'''\ve^2$. Here $C''' = C'''(C'')$, and $\delta''' = \delta''' (\delta'', C'', \ve)$.
\end{proof}

\begin{lemma} \label{lemdistca3}
Let $\phi:\C\to\C$ be a $K$-quasiconformal mapping. Let $\{I_n\}_n$ be a family of pairwise disjoint sub-arcs of $\partial \D$ such that
$$\sum_n \ell(I_n)\geq \delta,$$
with $\delta>0$.
Then,
$$\sum_n \ell(\phi(I_n))^\alpha \geq \delta'\,\diam(\phi(\D))^\alpha,$$
where $\delta'$ is a positive constant depending only on $K$, $\delta$; and $\alpha$ depends only on $K$ and verifies
$$\frac2{K+1}< \alpha<1.$$
\end{lemma}

\begin{proof}
By appropriate standard arguments, we may assume that $\diam(\phi(\D))=1$.
 We factorize $\phi=\phi_2\circ\phi_1$ so that 
$\phi_i$, $i=1,2$ are $K_i$-quasiconformal, with $K_1= 1+\ve$ and
$K_2=K/K_1$, and so that $\diam(\phi_1(\D))=1$.
By quasi-symmetry we may assume that the intervals $I_n$ are dyadic. (See e.g. \cite[Chapter 3]{astalaiwaniecmartin} for the definition of quasisymmetry and its equivalence with quasiconformality. A consequence of this definition is that both $\delta$ and $\delta'$ in the statement only change by a constant factor which only depends on $K$ if each of the intervals $I_n$ is replaced by a ``nearby" dyadic interval which has comparable size to $I_n$.)
By Lemma \ref{lemdistca2} we have
$$\sum_n \ell(\phi_1(I_n))^{\alpha_1} \geq \delta_1,$$
with 
\begin{equation}\label{eqd59}
\eta:=\alpha_1 - \frac2{K_1+1} >0
\end{equation}
if $\ve$ is small enough.

To estimate the distortion of the arcs $\phi_1(I_n)$, we consider a family of pairwise
disjoint balls $B_n$ centered on $\phi_1(I_n)$ with
radii $r_n\approx \ell(\phi_1(I_n))$, and so that $\diam(\phi_2(B_n))\approx \ell(\phi(I_n))$. Take a constant $K_2'>K_2$ to be fixed below. By H\"older's inequality, we have
\begin{align*}
\sum_n \ell(\phi_1(I_n))^{\alpha_1} & \approx
\sum_n r_n^{\alpha_1}  \lesssim \sum_n \biggl(\int_{\phi_2(B_n)} J(\phi_2^{-1})\,dx\biggr)^{\alpha_1/2}\\
& \leq \sum_n \biggl(\int_{\phi_2(B_n)} J(\phi_2^{-1})^{\frac{K_2'}{K_2'-1}}\,dx\biggr)^{\frac{\alpha_1  (K_2'-1)}{2K_2'}} \ell(\phi(I_n))^{\frac{\alpha_1}{K_2'}}\\
& \leq \biggl(\sum_n \int_{\phi_2(B_n)} J(\phi_2^{-1})^{\frac{K_2'}{K_2'-1}}\,dx\biggr)^{1/p} \biggl(\sum_n \ell(\phi(I_n))^{\frac{\alpha_1 p'}{K_2'}}\biggr)^{1/p'},
\end{align*}
where we chose
\begin{equation}\label{eqd60}
\frac1p = \frac{\alpha_1  (K_2'-1)}{2K_2'}.
\end{equation}
Notice that
$K_2'/(K_2'-1)<K_2/(K_2-1)$ and then
$$\sum_n \int_{\phi_2(B_n)} J(\phi_2^{-1})^{\frac{K_2'}{K_2'-1}}\,dx\leq \int J(\phi_2^{-1})^{\frac{K_2'}{K_2'-1}}\,dx \leq C(K_2, K_2') <\infty.$$
The last estimate follows from our normalizations ($\diam(\phi_1(\D))=\diam(\phi(\D))=1$) and the higher integrability of quasiconformal maps (see e.g. equation (13.24) in \cite{astalaiwaniecmartin}), which is in turn a consequence of the celebrated area distortion theorem proved by K. Astala in \cite{astalaareadistortion}. So we get
$$\delta_1\leq \sum_n \ell(\phi_1(I_n))^{\alpha_1} \leq C \biggl(\sum_n \ell(\phi(I_n))^{\frac{\alpha_1 p'}{K_2'}}\biggr)^{1/p'}.$$

To show that the lemma holds in this particular case, it is enough to take
$$\alpha:= \frac{\alpha_1 p'}{K_2'},$$
and then it remains to check that $2/(K+1)< \alpha <1$.
To this end, observe that, by \rf{eqd59} and \rf{eqd60},
$$\frac1p > \frac{K_2'-1}{(K_1+1) K_2'},$$
and thus
$$\frac1{p'} < \frac{K_1 K_2' +1}{(K_1+1)K_2'}.$$
From this estimate and \rf{eqd59} we obtain
$$\frac{\alpha_1 p'}{K_2'} > \biggl(\frac2{K_1+1} + \eta\biggr) \frac{K_1+1}{K_1 K_2' + 1} = \frac{2}{K_1K_2' +1} + \eta\,\frac{K_1+1}{K_1K_2'+1}.$$
From this inequality (with given $K= K_1K_2$ and $\eta$) it is clear that if $K_2'$ is close enough to $K_2$ (with $K_2'>K_2$), then
$$\alpha= \frac{\alpha_1 p'}{K_2'} > \frac2{K+1}.$$

To show that $\alpha<1$, notice that \rf{eqd60} implies that
$$\frac1p < \frac{K_2'-1}{2K_2'},$$
and then one easily gets 
$$p'< \frac{2K_2'}{K_2'+1},$$
and thus
$$\alpha= \frac{\alpha_1 p'}{K_2'} < \frac{p'}{K_2'} < \frac{2}{K_2'+1}
<1,$$
since $K_2'>K_2\geq1$.
\end{proof}

\begin{remark}\label{RemarkQuantitativeCompressionOfLine} The preceding arguments show that, choosing a suitable 
$K_2'$, one gets
$$\alpha \geq \frac{2}{K +1} + \frac\eta2\,\frac{K_1+ 1}{K+1}\geq 
\frac{2}{K +1} + \frac\eta2\,\frac{1}{K+1}.$$ 
\end{remark}

\begin{lemma} \label{lemdistca4}
Let $\phi:\C\to\C$ be a principal $K$-quasiconformal mapping, and let $\Gamma\subset\C$ be a chord arc curve. 
Let $\{I_n\}_n$ be a family of pairwise disjoint subarcs of $\Gamma$ such that
$$\sum_n \ell(I_n)\geq \delta\,\diam(\Gamma),$$
with $\delta>0$.
If the chord arc constant $C_\Gamma$ is close enough to $1$, that is, $|C_\Gamma -1|\leq
\ve_0$ with $\ve_0=\ve_0(K)$, then
$$\sum_n \ell(\phi(I_n))^\alpha \geq \delta'\, \diam(\phi(\Gamma))^\alpha,$$
where $\delta'$ is a positive constant depending only on $K$, $\delta$, and the chord arc constant; and $\alpha$ depends only on $K$ and verifies
$$\frac2{K+1}< \alpha.$$
\end{lemma}

Recall that a chord arc curve is the bilipschitz image of an interval. The chord arc constant is the bilipschitz constant (or rather the infimum over all the possible bilipschitz 
constants).

Notice that the above result can be understood as a quantitative version of the result of 
\cite{ACMOU} which asserts that 
if $F$ is rectifiable and of positive length, then $\dim(\phi(F))>2/(K+1)+ c(K)$, where $c(K)$ is some positive constant depending only on $K$.

\begin{proof}[Proof of Lemma \ref{lemdistca4}]
If $\Gamma$ is an arc of a circumference or a segment, then the result follows   
from Lemma \ref{lemdistca3} by appropriate normalization.

In the case of a general chord arc curve with small constant, we consider
a bilipschitz parametrization
$f:J\to\Gamma$, where $J$ is a segment with $\ell(J) = \diam(\Gamma)$, so 
that the 
 bilipschitz constant $C_f$ of $f$ is very close to $1$:
$$|C_f - 1|\leq c(K)\qquad\mbox{with $c(K) \to 0$ as $\ve_0(K)\to 0$.}$$
By a theorem of V\"{a}is\"{a}l\"{a} \cite{vaisala}, $f$ can be extended to a bilipschitz 
mapping $\widetilde f:\C\to\C$ with constant $C_{\widetilde f}$ depending on $C_f$ very close to $1$ too. In particular $\widetilde f$ is quasiconformal
with constant $K_{\widetilde f}\to 1$ as $\ve_0(K)\to 0$.

Using the auxiliary mapping $\phi_0 = \phi\circ f$, we deduce from Lemma \ref{lemdistca3} that
$$\sum_n \ell(\phi(I_n))^\alpha \geq \delta'\, \diam(\Gamma)^\alpha,$$
with $\alpha$ such that
$\alpha> \dfrac2{K K_{\widetilde f}+1}.$ By Remark \ref{RemarkQuantitativeCompressionOfLine}, 
for $K_{\widetilde f}$ close enough to $1$, we have
$$\alpha> \frac2{K+1}.$$ 
\end{proof}

\bigskip


\section{A corona type decomposition for measures with finite curvature and linear growth}\label{seccorona0}

Throughout all this section we suppose that $\mu$ is supported on $E\subset B(0,1/2)$, and satisfies
$$\mu(B(x,r))\leq r\quad \mbox{for all $x\in\C$, $r>0$; }\qquad\;\;\; c^2_\mu(x)\leq 1
\quad\mbox{for all $x\in\C$.}$$
 As explained in Section \ref{secstrat}, our objective is to construct a
corona type decomposition for $\mu$, which has some similarities with the one of \cite{tolsabilip}. This corona type decomposition will be used in Section \ref{secconsnu}
to find a measure $\nu$ supported on $\phi(E)$ with bounded potential
$\dot W_{ \frac{2K}{2K+1},\frac{2K+1}{K+1}}^\nu$.


\subsection{Additional notation and terminology}\label{subsecterm}

By a square we mean a square with sides parallel to
the axes. Moreover, we assume the squares to be half closed - half
open. The side length of a square $Q$ is denoted by $\ell(Q)$. If $\ell(Q)=2^{-n}$,
then we write $J(Q)=n$.
Given $a>0$, $aQ$ denotes the square concentric
with $Q$ with side length $a\ell(Q)$.
A square $Q\subset \C$ is called $4$-dyadic if it is of the form
$[j2^{-n},\, (j+4)2^{-n}) \times [k2^{-n},\, (k+4)2^{-n})$, with
$j,k,n\in \Z$. So a $4$-dyadic square with side length
$4\cdot2^{-n}$ is made up of $16$ dyadic squares with side length
$2^{-n}$. 

Given $a,b>1$, the square $Q$ is
$(a,b)$-doubling if $\mu(aQ)\leq b\mu(Q)$. If we don't want to
specify the constant $b$, we say that $Q$ is $a$-doubling.
If $h_a$ is the function defined in \rf{defhx},
we say that $Q$ is $(h_a,b)$-doubling
if
$$h_a(Q)\leq b\mu(Q),$$
which is equivalent to $\ve_a(Q)\leq b\theta_\mu(Q).$
Notice that if $Q$ is $(h_a,b)$-doubling, then, for all $c>1$ there exists some $d>0$ depending only on $a,b,c$ such that $Q$ is $(c,d)$-doubling. 

Given a bijective mapping $\phi:\C\to\C$ and a square $Q$,  one says that that $\phi(Q)$ is a $\phi$-square, and then
one defines its side length as $\ell(\phi(Q)) := \diam(Q)$. 
If $Q_0$ is a dyadic (or $4$-dyadic) square, we
say that $\phi(Q_0)$ is a dyadic (or $4$-dyadic)
$\phi$-square. If $Q= \phi(Q_0)$ is a $\phi$-square, we
denote $\lambda Q = \phi(\lambda Q_0)$, for
$\lambda>0$.  

 An Ahlfors regular curve is a curve
$\Gamma$ such that $\H^1(\Gamma\cap B(x,r))\leq C r$ for all
$x\in \Gamma$, $r>0$, and some fixed $C>0$. Recall that $\Gamma$ is a chord arc curve if it is a bilipschitz
image of an interval in $\R$. If the bilipschitz constant of the map
is $L$, we say that $\Gamma$ is an $L$-chord arc curve.

The total Menger curvature of $\mu$ is
$$c^2(\mu)= \int c^2_\mu(x)\,d\mu(x),$$
with $c^2_\mu(x)$ defined by \rf{eqdfcm}.
The curvature operator $K_\mu$ is
$$K_\mu(f)(x)=\int k_\mu(x,y) f(y) d\mu(y),\hspace{8mm}
f\in L^1_{loc}(\mu),\,x\in\C,$$ where $k_\mu(x,y)$ is the kernel
$$k_\mu(x,y)=\int \frac1{R(x,y,z)^2}\, d\mu(z),\hspace{8mm}
x,y\in\C.$$
 For $j\in\Z$, the truncated operators
$K_{\mu,j}$, $j\in\Z$, are defined as
$$K_{\mu,j} f(x) = \int_{|x-y|>2^{-j}}
k_\mu(x,y)\,f(y)\,d\mu(y),\hspace{8mm} f\in
L^1_{loc}(\mu),\,x\in\C.$$
Notice that
$c^2_\mu(x) = K_\mu(\chi_E)(x).$

\bigskip


\subsection{Properties of $(h_a,b)$-doubling squares}

\begin{remark}
Let $Q$ be a square and $x$ its center. For $N\geq 1$, we have
\begin{align*}
\ve_a(Q)  & \approx \frac1{\ell(Q)} \int \frac1{\left(\frac{|x-y|}{\ell(Q)}\right)^{1+a} + 1}\,
d\mu(y)\\ & \leq C \sum_{j=0}^N \frac1{\ell(Q)}\,\frac{\mu(2^jQ)}{2^{j(1+a)}} + \frac1{\ell(Q)}\int_{\C\setminus Q_N} \frac1{\left(\frac{|x-y|}{\ell(Q)}\right)^{1+a}}\,
d\mu(y),
\end{align*}
where $Q_N:=2^NQ$ and the constant $C$ depends on $a$ but on $N$. Since 
$$ \frac1{\ell(Q)}\int_{\C\setminus Q_N} \frac1{\left(\frac{|x-y|}{\ell(Q)}\right)^{1+a}}\,
d\mu(y) = \frac{2^{-aN}}{\ell(Q_N)}\int_{\C\setminus Q_N} \frac1{\left(\frac{|x-y|}{\ell(Q_N)}\right)^{1+a}}\,
d\mu(y) \leq C(a) 2^{-aN}\,\ve_a(Q_N),$$
we deduce
\begin{equation}\label{eqsum69}
\ve_a(Q) \leq C(a)\biggl( \sum_{j=0}^{N-1} 2^{-aj}\theta_\mu(2^jQ) + 2^{-aN}\,\ve_a(Q_N)\biggr).
\end{equation}
The converse inequality is also true, but we will not
need it.
\end{remark}
\bigskip

\begin{lemma}\label{lemnodob}
Given $a>0$, let $b>0$ be some big enough constant.
Let $Q$ be a square, and suppose that $2^{-j} Q$ is not $(h_a,b)$-doubling for 
$0\leq j \leq N$. Then,
\begin{equation}\label{eqd99}
\theta_\mu(2^{-j}Q)\leq 2^{-aj/2}\,\ve_a(Q)\quad \mbox{ for $0\leq j\leq N$,}
\end{equation}
and
\begin{equation}\label{eqd99.5}
\sum_{j=0}^N \ve_a(2^{-j} Q)^2 \leq C\ve_a(Q)^2,
\end{equation}
with $C$ independent of $N$.
\end{lemma}

\begin{proof}
By \rf{eqsum69}, the fact that $2^{-j} Q$ is not $(h_a,b)$-doubling for 
$0\leq j \leq N$ implies that
\begin{equation}\label{eqd100}
\theta_\mu(2^{-j} Q)\leq \frac1b \ve_a(2^{-j}Q)\leq \frac{C_3}{b}\biggl( \sum_{k=1}^{j-1} 2^{-ak}\theta_\mu(2^{-j+k}Q) + 2^{-aj}\ve_a(Q)
\biggr),
\end{equation}
where $C_3$ depends on $a$. Notice that the sum above starts with $k=1$, while the one in \rf{eqsum69} starts with $j=0$ (we used the fact that $\theta_\mu(2^{-j}Q)\leq C\theta_\mu(2^{-j+1}Q)$).

We prove \rf{eqd99} by induction on $j$. For $j=0$, this is a direct consequence of
the definition of $(h_a,b)$-doubling squares.
Suppose that \rf{eqd99} holds for $0\leq h\leq j$, with $j\leq N-1$, and consider the case $j+1$. Using \rf{eqd100} and the induction hypothesis we get
\begin{align*}
\theta_\mu(2^{-j-1} Q) &\leq \frac{C_3}{b}\biggl( \sum_{k=1}^{j} 2^{-ak}
\theta_\mu(2^{-j-1+k}Q) + 2^{-a(j+1)}\ve_a(Q)\biggr)\\
& \leq \frac{C_3}{b}\biggl( \sum_{k=1}^{j} 2^{-ak}2^{(-j-1+k)a/2}\,\ve_a(Q) + 2^{-a(j+1)}\ve_a(Q)\biggr)
\end{align*}
Since $$\sum_{k=1}^{j} 2^{-ak}2^{(-j-1+k)a/2}\leq C(a) 2^{-aj/2},$$
we obtain
$$\theta_\mu(2^{-j-1} Q) \leq\frac{C_3C(a)}{b}\,\bigl( 2^{-aj/2}   + 2^{-a(j+1)}\bigr)
\ve_a(Q).$$
If $b$ is chosen big enough, we get
$$\theta_\mu(2^{-j-1} Q) \leq 2^{-a(j+1)/2}\ve_a(Q).$$

The estimate \rf{eqd99.5} is a straightforward consequence of \rf{eqd99}, using Cauchy-Schwartz
inequality. We leave the details for the reader.
\end{proof}

Let $b=b(a)>0$ be big enough so that \rf{eqd99} and \rf{eqd99.5} hold. 
It is immediate to check that if $Q$ is $(h_a,b)$-doubling and $R\supset Q$ is a square
such that $\ell(R)\leq 4\ell(Q)$, then $R$ is $(h_a,C_4b)$-doubling.  
We say that a square
$R$ is $h_a$-doubling if it is $(h_a,C_4b)$-doubling.

Let $Q,R$ be squares with $\ell(Q)\leq\ell(R)$. We 
denote 
$$D_\mu(Q,R) = \sum_{j: Q\subset 2^jQ \subset R_Q}\ve_a(2^jQ)^2,$$
where $R_Q$ denoted the smallest square of the form $2^jQ$ that contains $R$.  
The preceding 
lemma says that if $Q\subset R$ and there are no $h_a$-doubling squares of the form $2^jQ$ such that
$Q\subset2^jQ\subset R_Q$, then $D_\mu(Q,R)\leq C\ve_a(R)^2$. 

The definition of $D_\mu(Q,R)$
can be extended in a natural way to the case where $Q$ is replaced by a point. In this case
the sum above runs over all squares centered at $x$ with side length $2^j$, $j\in\bZ$, which
are
contained in $R_x$, where $R_x$ is the smallest square centered at $x$ that contains $R$.

\begin{remark} \label{rem53}
Let $\mu$ be any Radon measure on $\C$, and let $d$ be big enough. Then, for $\mu$-almost all $x\in E$, there exists a sequence of  $(2,d)$-doubling squares $\{Q_n\}_n$ centered
at $x$ such that $\ell(Q_n)\to 0$. However, this statement 
is false if we replace $(2,d)$-doubling squares by $(h_a,d)$-doubling squares when $a$ is small. The reader can check that this is the case for planar Lebesgue measure, for instance.
\end{remark}

\bigskip

\subsection{The family $\maxbad(R)$} \label{secstop}

Let $R$ be some fixed $4$-dyadic square such that $\frac12 R$ is $h_a$-doubling. In this subsection we will explain
the construction of a family of $4$-dyadic squares called $\maxbad(R)$. 

Let $A>10$ be some big constant to be chosen below, $\delta$ some small positive constant ($\delta<1/10$, say) which depends on $A$; and $\ve_0$ another small constant with  $0<\ve_0<1/100$ (depending on $A$ and $\delta$). Let $Q$ be a square centered at some point in
$3R\cap\supp(\mu)$, with $\ell(Q) =2^{-n} \ell(R)$, $n\geq5$. We
introduce the following notation:
\begin{itemize}
\item[(a)] If $\theta_\mu(Q) \geq A\theta_\mu(R)$, then we write $Q\in
HD_{c}(R)$ (high density).

\item[(b)] If $Q\not\in HD_{c}(R)$ and
$$\mu\bigl\{x\in Q:\,K_{\mu,J(Q)+10}\chi_E(x) - K_{\mu,
J(R)-2}\chi_E(x)\geq \ve_0\theta_\mu(R)^2\bigr\} \geq
\frac12\,\mu(Q),$$ then we set $Q\in HC_{c}(R)$ (high
curvature).

\item[(c)] If $Q\not\in HD_{c}(R)\cup HC_{c}(R)$ and
there exists some square $S_Q$ such that $Q\subset \frac1{100}
S_Q$, with $\ell(S_Q)\leq \ell(R)/8$ and $\theta_\mu(S_Q)\leq
\delta\, \theta_\mu(R)$, then we set $Q\in LD_{c}(R)$ (low
density).
\end{itemize}

For each point $x\in 3R\cap\supp(\mu)$ which belongs to some
square from $HD_{c}(R)\cup HC_{c}(R)\cup LD_{c}(R)$ consider the
largest square $Q_x\in HD_{c}(R)\cup HC_{c}(R)\cup
LD_{c}(R)$ which contains $x$. Let $\wh{Q}_x$ be a $4$-dyadic
square with side length $4\ell(Q_x)$ such that
$Q_x\subset\frac12\wh{Q}_x$. Now we apply Besicovitch's covering
theorem to the family $\{\wh{Q}_x\}_x$ (notice that this theorem
can be applied because $x\in\frac12 \wh{Q}_x$), and we obtain a
family of $4$-dyadic squares $\{\wh{Q}_{x_i}\}_i$ with finite
overlap such that the union of the squares from $HD_{c}(R)\cup HC_{c}(R)\cup LD_{c}(R)$
is contained (as a set in $\C$) in
$\bigcup_i \wh{Q}_{x_i}$. We define
$$\maxbad(R):=\{\wh{Q}_{x_i}\}_i.$$
Notice that the squares $Q\in \maxbad(R)$ satisfy $\ell(Q)\leq \ell(R)/8$.
If $Q_{x_i}\in HD_{c}(R)$, then we write $\wh{Q}_{x_i}\in
HD(R)$, and analogously with $HC_{c}(R)$, $LD_{c}(R)$ and
$HC(R)$, $LD(R)$. We also denote
\begin{equation}\label{defgg}
G(R) = 3R\setminus \bigcup_{Q\in\maxbad(R)} Q.
\end{equation}

\begin{remark}
The constants that we denote by $C$ (with or without subindex) in
the rest of Section \ref{seccorona0} do not depend on
$A$, $\delta$, or $\ve_0$, unless stated otherwise.
\end{remark}

To define the squares $\maxbad(R)$ we have followed
quite closely the arguments in \cite{tolsabilip}. However, there 
are a couple of small changes: in \cite{tolsabilip} we ask the square $R$ to be 
$(70,5000)$-doubling instead of $(h_a,b)$-doubling. Moreover, in \cite{tolsabilip} the squares from $HD_{c}(R)$,
$LD_{c}(R)$, and $HC_{c}(R)$ are
asked to be $(70,5000)$-doubling and then the resulting squares from $\maxbad(R)$ are
$(16,5000)$-doubling. Now, for convenience, we have not asked any doubling condition on
these squares, although below we will need other stopping squares to be doubling
(in fact, $h_a$-doubling).

The squares from $\maxbad(R)$ satisfy the following important properties:

\begin{lemma}\label{lemcorba}
Let $0<\rho<1$ be some fixed constant. Let $R$ be a $4$-dyadic square such that $\frac12 R$ is $h_a$-doubling.
 Given $A$ and $\delta$ as above, if 
$\ve_0$ is chosen small enough (depending on $A,\delta,\rho$), there are constants $C_5=C_5(A,\delta)>1$ and $C_6=C_6(A,\delta)>0$, and there are $N_0$ chord arc curves with constant $(1+\rho)$ whose union we denote
by $\Gamma_R$ with the following properties:
\begin{itemize}
\item[(a)] 
$G(R)\subset \Gamma_R$;
\item[(b)]any square $Q\in\maxbad(R)$ satisfies
$$C_5Q\cap \Gamma_R\neq\varnothing;$$
\item[(c)] if $P$ is a square concentric with $Q\in\maxbad(R)$ and 
$C_5\ell(Q)\leq \ell(P)\leq \ell(R)$, then
$$C_6^{-1}\theta_\mu(R)\leq \theta_\mu(P)\leq C_6\theta_\mu(R).$$

\end{itemize}
The constant $N_0$ depends only on $A$,$\delta$, and $\rho$.
\end{lemma}

For the proof of this lemma, see \cite[Section 4]{tolsabilip}.
One only needs to make very minor adjustments for that arguments to work in our situation.
See also \cite[Subsection 2.3]{Clop-Tolsa} concerning the fact that one can take chord arc
curves (in the original arguments in \cite{tolsabilip} $\Gamma_R$ turns out to be an 
AD regular curve). We will not go through the details.

\begin{remark} \label{remdob}
It is easy to check that the property (c) in the preceding lemma implies that the
squares $P$ from (c) are $(h_a,c)$-doubling, with $c$ depending on $A$ and~$\delta$.
\end{remark}

We also have:

\begin{lemma}\label{lemld}
Given $A>0$, if $\delta$ and $\ve_0$ are chosen small enough, then for any
 $4$-dyadic square $R$ with $\frac12 R$ $h_a$-doubling, we have 
$$\mu\biggl(\,\bigcup_{Q\in LD(R)}Q\biggr)\leq \frac1{100}\,\mu(R).$$
\end{lemma}

For the proof, see \cite[Section 7]{tolsabilip}. Again, the arguments there work with very
minor adjustments.

\bigskip

\subsection{The families $\sel(\mu)$,  $\sel_S(\mu)$, and $\sel_L(\mu)$.}

In the corona construction from \cite{tolsabilip} one constructs 
recursively the familiy of
squares $\ttt(\mu)$ mentioned in Section \ref{secstrat}. In this subsection we construct a quite analogous family which we will denote by $\sel(\mu)$ (the ``selected
squares''). We use another notation 
because the family $\sel(\mu)$ 
will have significant differences with respect the family $\ttt(\mu)$ of \cite{tolsabilip}.

First we have to distinguish two types of $h_a$-doubling squares:

\begin{definition}\label{def58}
Let $\eta>0$ be some constant to be fixed below (in Section \ref{secproof}), which will depend on $A,\delta,\ve_0,\rho,K$ (recall that $K$ is the distortion of the quasiconformal
mapping $\phi$).
Let $R$ be a square such that $\frac12 R$ is $h_a$-doubling. We say $R$ is of type S if
$$\mu\biggl(\,\bigcup_{\substack{Q\in \maxbad(R):\\ \ell(Q)\geq \eta \ell(R)}} 
Q\cap \frac12 R\biggr)
\geq \frac12\,\mu\Bigl(\frac12R\Bigr).$$
Otherwise, we say that $R$ is of type L. The letters S and L stand for ``short'' and ``long''
trees, respectively (this terminology will be more clear below).
\end{definition}

Before constructing the families $\sel(\mu)$,  $\sel_S(\mu)$, and $\sel_L(\mu)$, 
 we have to define the family of terminal squares $\cT(R)$.

\bigskip

\subsubsection{\bf Definition of $\cT(R)$ when $R$ is of type S} \label{subsels}

Let $R$ be a square of type $S$, so that $\frac 12 R$ is $h_a$-doubling.
For $x\in 3R$, consider 
the biggest $4$-dyadic square 
$Q_x$ of type $L$ containing $x$, such that $\frac12 Q_x$ is $h_a$-doubling, and such that $\ell(Q_x)\leq \ell(R)/8$, if it exists. Let $\cT_0(R)$ be the collection of these squares
$Q_x$.
We denote by $F(R)$ the subset of those points $x\in 3R$ such there does not exists such a square $Q_x$.

By Vitali's covering theorem there exists a subfamily $\cT(R)\subset \cT_0(R)$ such that
the squares 
$\{5Q\}_{Q\in\cT(R)}$ are pairwise disjoint and so that
$$\bigcup_{Q\in\cT_0(R)}5Q\subset \bigcup_{Q\in\cT(R)}15Q.$$
Since the squares $Q_x$ that intersect $\frac12 R$ are contained in $R$ and
they are doubling,
$$\mu\biggl(\,\bigcup_{Q\in\cT(R)}Q\cap R\biggr) \geq C_7^{-1}\mu\biggl(\frac12R\setminus F(R)\biggr).$$

\bigskip

\subsubsection{\bf Definition of $\cT(R)$ when $R$ is of type L}\label{subsell}

In this case
$$\mu(G(R)) + \mu\biggl(\,\bigcup_{\substack{Q\in \maxbad(R):\\ \ell(Q)< \eta \ell(R)}} 
Q\cap \frac12 R\biggr)
\geq\frac12\,\mu\biggl(\frac12R\biggr).$$

If $\mu(G(R))\geq \frac14\,\mu(\frac12 R)$, then we set $\cT(R)=\varnothing$.

Suppose now that $\mu(G(R))< \frac14\,\mu(\frac12 R)$. Then, 
$$\mu\biggl(
\,\bigcup_{\substack{Q\in \maxbad(R):\\ \ell(Q)< \eta \ell(R)}}
Q\cap \frac12 R\biggr)\geq \frac14\,\mu\biggl(\frac12 R\biggr).$$
Recall that, the squares $C_5Q$ in Lemma \ref{lemcorba} are doubling, by the statement (c) in the same lemma applied to $P=C_5Q$ (in fact,
$(h_a,c)$ doubling, with $c=c(A,\delta)$, by Remark \ref{remdob}).
We assume that the constant $\eta$ in the  Definition \ref{def58} of L squares is small
enough so that 
$$C_5\ell(Q)\leq \ell(R)/100 \qquad \mbox{if $\ell(Q)<\eta\ell(R)$},$$
say. By Vitali's covering theorem, there exists a subfamily 
\begin{equation}\label{eqsj7}
\{S_j\}_{j\in I_R}\subset \{C_5Q:\,Q\in \maxbad(R),\,\ell(Q)<\eta\ell(R)\}
\end{equation}
 such that the squares $5S_j$, $j\in I_R$, are pairwise
disjoint and contained in $R$ and, moreover, using the doubling property of the squares $C_5Q$,
$$\mu\biggl(\,\bigcup_{j\in I_R} S_j\cap R\biggr)\geq C^{-1}\mu\biggl(\frac12 R\biggr)
\geq C_8^{-1}\mu(R),$$
with $C_8$ depending on $A,\delta$ (but not on $\eta$).

Take a square $S_j$, $j\in I_R$, such that $S_j\cap R\neq \varnothing$. For each $x\in E\cap S_j$, consider the biggest 
square $P_x$ centered at $x$, with $\ell(P_x)\leq \ell(S_j)/16$, which is 
$(h_a,b)$-doubling, with $b$ as explained just above
Remark \ref{rem53}, in case such a square exists. 
We denote by $F_j(R)$ the subset of those points $x\in E\cap S_j$ such there does not exists such a square. Denote by $\wh P_x$ a $4$-dyadic square with side length $4\ell(P_x)$ such that
$P_x\subset \frac12 \wh P_x$. Notice that the squares $\wh P_x$ are $h_a$-doubling and they
are contained in $3S_j$.

By Vitali's covering theorem, there exists a subfamily $\{\wh P_{x_i}\}_i\subset 
\{\wh P_x\}_{x\in E\cap S_j\setminus F_j(R)}$ such that the squares $5\wh P_{x_i}$ are pairwise
disjoint, and 
$$\mu(S_j\setminus F_j(R)) \leq C\mu\biggl(\bigcup_i \wh P_{x_i}\biggr).$$
 We define
$\cT_j(R) :=\{\wh P_{x_i}\}_i,$
and finally
$$\cT(R) := \bigcup_{j\in I_R}\cT_j(R).$$
We also set
$$F(R) := \bigcup_{j\in I_R}F_j(R).$$

\bigskip

\subsubsection{\bf Definition of $\sel(\mu)$,  $\sel_S(\mu)$, and $\sel_L(\mu)$}

The family $\sel(\mu)$ is constructed recursively. 
 Let $R_0$ be a $4$-dyadic square with $\ell(R_0)\simeq
\diam(E)$ such that $E$ is contained in one of the four dyadic
squares in $\frac12R_0$ with side length $\ell(R_0)/4$.
The first square of $\sel(\mu)$ is $R_0$.
The next squares that we choose as elements of $\sel(\mu)$ are the ones from 
$\cT(R_0)$. And, now the ones that belong to $\cT(R)$ for some $R\in\cT(R_0)$,
an so on.

In other words, $\sel(\mu)$ is the smallest family of $4$-dyadic squares that contains
$R_0$ and which has the property that if $R\in\sel(\mu)$, then the squares from 
$\cT(R)$ also belong to $\sel(\mu)$.

The family $\sel_S(\mu)$ is made up of the squares from $\sel(\mu)$ of type S, while
 $\sel_L(\mu)$ is the subfamily of the squares from $\sel(\mu)$ of type L.

\bigskip

\subsection{\bf The packing condition for squares in $\tree(R)$, $R\in\sel_S(\mu)$}

\begin{definition}\label{def51}
For $R\in\sel_S(\mu)$, we denote by $\term(R)$ the collection of dyadic squares 
$Q$ such that $Q\subset 3P$ for some $P\in 
\cT_0(R)$, so that, moreover, $Q$ is maximal. We call them terminal squares.

We denote by $\tree(R)$ the family of dyadic squares that are contained in 
$R$ and that are not properly contained in any square from $\term(R)$.

We also set
$$\eee(R) = E\cap R\setminus \bigcup_{Q\in\term(R)}P.$$
\end{definition}

Notice that the points in $\eee(R)$ can be considered as terminal squares of
$\tree(R)$ with zero side length.

The main objective of this subsection consists in proving the following result.

\begin{lemma}\label{lempac0}
Let $R\in\sel_S(\mu)$. Then,
$$\sum_{Q\in\tree(R)}\ve_a(Q)^2\,\mu(Q)\leq C(A,\delta,\ve_0,\eta)\mu(R).$$
\end{lemma}

The main tool for the proof will be the corona construction of \cite{tolsabilip}. 
To state the  precise result that we will use, we need to introduce some notation. Let $R$ be a $4$-dyadic square such that $\frac12 R$ is $h_a$-doubling. The next lemma 
deals with a family $\ttt_R(\mu)$ of $4$-dyadic squares satisfying 
some precise properties.
Given $Q\in\ttt_R(\mu)$, we denote by $\sss(Q)$ the subfamily of squares $P\in\ttt_R(\mu)$
satisfying
\begin{itemize}
\item[(a)] $P\cap 3Q\neq\varnothing$,
\item[(b)] $\ell(P)\leq \ell(Q)/8$,
\item[(c)] $P$ is maximal, in the sense that there are not other squares $\{P_j\}_j
\subset\ttt_R(\mu)$ with $\ell(P_j)<\ell(P)$ such that $P\subset \bigcup_j P_j$.
\end{itemize}
We also denote 
$$\wt G(Q) = 3Q\cap E\setminus \bigcup_{P\in \sss(Q)} P.$$

\begin{lemma}[\cite{tolsabilip}]\label{lembilip}
Let $\mu$ be a Radon measure supported on $E\subset\C$ such that $\mu(B(x,r))\leq r$
for all $x\in\C$, $r>0$, and $c^2(\mu_{|40 R})<\infty$. Let $\wt A>10$ be big enough
and $\wt \delta,\wt \ve_0>0$ small enough. 
 Let $R$ be a
$4$-dyadic square such that $\frac12 R$ is $h_a$-doubling. There exists a family
$\ttt_R(\mu)$ of $4$-dyadic squares contained in $4R$ such that
\begin{equation}\label{eqcorona1}
\sum_{Q\in\ttt_R(\mu)}\theta_\mu(Q)^2\mu(Q)\leq C(\wt A,\wt \delta, \wt\ve_0)\bigl(\mu(R)
+ c^2(\mu_{|40 R})\bigr),
\end{equation}
and such that for $Q\in\ttt_R(\mu)$, if $P$ is a square 
with $\ell(P)\leq \ell(Q)$ such that either $P\cap \wt G(Q)\neq \varnothing$ or
there is another square $P'\in\sss(Q)$ satisfying $P\cap P'\neq\varnothing$ and 
$\ell(P')\leq \ell(P)$, then
\begin{itemize}
\item[(a)] $\theta_\mu(P)\leq C\wt A\,\theta_\mu(Q)$,

\item[(b)] every square $P''$ concentric
with $P$ such that $P\subset P''\subset 5R$ and $D_\mu(P,P'')\geq C_9(\wt A,\wt\delta)\theta_\mu(Q)^2$, satisfies
$$\theta_\mu(P'')\geq C^{-1}\wt\delta\,\theta_\mu(Q).$$

\item[(c)] every square $P''$ such that $\frac12 P''$ is $h_a$ doubling
and $P\subset \frac 34P''$,  $P''\subset 5R$ and $D_\mu(P,P'')\geq C(\wt A,\wt\delta)
\theta_\mu(Q)^2$, satisfies
$$\mu\bigl\{x\in P'':\,K_{\mu,J(P'')+10}\chi_E(x) - K_{\mu,
J(R)-4}\chi_E(x)\geq \wt\ve_0\theta_\mu(Q)^2\bigr\} \leq \beta
\,\mu(P''),$$
where $0<\beta<1$ is some fixed constant.
\end{itemize}
\end{lemma}

\medskip
Some remarks about the choice of 
the constants $\wt A,\wt\delta,\wt\ve_0,\beta$ in the preceding lemma are in order: first, $\wt A$ can be taken as big as desired. After
choosing $\wt A$, one has to take $\wt \delta\leq \wt\delta_1(\wt A)$, where $\wt\delta_1(\wt A)$ is some fixed small constant, and finally, one has to choose
$\wt\ve_0\leq \wt\ve_1(\wt A,\wt \delta,\beta)$. In particular, the preceding lemma holds for all
 $\wt\ve_0$ small enough, at the price of increasing the constant in the right side of 
 \rf{eqcorona1} as $\wt\ve_0\to0$.

In \cite{tolsabilip}, the reader will not find an exact statement
such as Lemma \ref{lembilip}. In fact, in \cite{tolsabilip}, every square $\frac12Q$,
with $Q\in\ttt(\mu)$, is $(32,5000)$-doubling, instead of $h_a$-doubling.
Also, Lemma \ref{lembilip} is proved only in the particular case where
$E\subset R$. However, the same arguments, with very 
minor changes, work with the assumptions above. 
On the other hand, the corona decomposition of \cite{tolsabilip} also states the existence
of curves $\Gamma_Q$ satisfying properties similar to the ones of Lemma \ref{lemcorba}.
However, this information is not useful to prove Lemma \ref{lempac0}, and so we have
skipped it. 

\medskip

\begin{lemma}\label{lem45}
Let $R\in\sel_S(\mu)$ and 
$Q_0\in \ttt_R(\mu)$. Also, let $Q\in\tree(R)$ be a $4$-dyadic square such that $Q\cap 3Q_0\neq \varnothing$, $\ell(Q)\leq \ell(Q_0)/8$, and so
that $\frac12Q$ is $h_a$-doubling. Then there exists a collection of 
squares or points $\{P_i\}_i$ contained in $Q$ such that
\begin{itemize}
\item[(a)] each $P_i$ is contained either in a union of squares from $P\in\term(R)\cup\eee(R)$, or in $3P$, for some $P\in \sss(Q_0)$, 
\item[(b)]  $D_\mu(P_i,Q)\leq M\theta_\mu(Q)^2$ if $\ell(P_i)\leq\ell(Q)$, with $M$ depending on 
$\wt A,\wt\delta$,
\item[(c)] and
$$\mu\biggl(Q\cap \bigcup_i P_i\biggr)\geq \tau\mu(Q),$$
\end{itemize}
assuming that the constants $A,\wt A, \delta, \wt \delta, \ve_0, \wt \ve_0,\beta$ are chosen 
appropriately.
\end{lemma}

In this lemma, by convenience we understand that the points in $\eee(R)$ are squares
with zero side length. To prove it, we will make
essential use of the fact that $R$ is of type S.

\begin{proof} If the square $Q$ is of type L, then $Q\in\term(R)$ by definition (since
$Q\in\tree(R)$ and $R$ is of type $S$).
Then we just take $P_i=Q$ and then the lemma holds.
If every square $T$ which intersects $\frac12 Q$ and such that
$D_\mu(T,Q)\geq M\theta_\mu(Q)^2$ is contained in $3P$, for some $P\in\sss(Q_0)\cup \wt G(Q_0)$, we are also done. Therefore, we may assume that $Q$ is of type $S$ and that there exists a square $T$
which intersects $\frac12 Q$ 
such that $D_\mu(T,Q)\geq M\theta_\mu(Q)^2$, satisfying $T\cap P\neq \varnothing$ for some $P\in\sss(Q_0)\cup \wt G(Q_0)$ with $\ell(P)<\ell(T)$ (otherwise, $T\subset 3P$).
This condition implies that 
$$C^{-1}\wt\delta\theta_\mu(Q_0)\leq \theta_\mu(Q)\leq C\wt A\theta_\mu(Q_0),$$
by conditions (a) and (b) of Lemma \ref{lembilip}, assuming $M$ big enough.

Since $Q$ is of type S, there
are squares $S_i\in\maxbad(Q)$ such that $\eta\ell(Q)\leq\ell(S_i)\leq\ell(Q)/8$, with
$S_i\cap\frac12 Q\neq\varnothing$, and
$$\mu\biggl(\bigcup_i S_i\biggr)\geq \frac12 \mu\Bigl(\frac12 Q\Bigr).$$
By Lemma \ref{lemld}, if $\delta$ is small enough, there are squares $\{S_i\}_{i\in I_{HD}}\subset HD(Q)$
and $\{S_i\}_{i\in I_{HC}}\subset HC(Q)$ such that
$$\mu\biggl(\,\bigcup_{i\in I_{HD}\cup I_{HC}} S_i\biggr)\geq \frac14 \mu\Bigl(\frac12 Q\Bigr).$$

Notice that if $S_i\in HD(Q)$, then 
$$\theta_\mu(S_i)\geq C^{-1} A\theta_\mu(Q)\geq C^{-1}A\wt \delta\theta_\mu(Q_0)\gg\theta_\mu(Q_0)$$
if we choose $A$ such that $A\wt \delta\gg \wt A$. Then it is easy to check that $S_i$ 
satisfies the conditions (a) and (b) of the lemma. Condition (c) also holds if 
$$\mu\biggl(\,\bigcup_{i\in I_{HD}} S_i\biggr)\geq \frac18 \mu\Bigl(\frac12 Q\Bigr).$$

If the latter condition fails, then we have
$$\mu\biggl(\,\bigcup_{i\in I_{HC}} S_i\biggr)\geq \frac18 \mu\Bigl(\frac12 Q\Bigr).$$
Let $\{\wh P_j\}_j$ be a family of $4$-dyadic squares or points such that
$\frac12 \wh P_j$ $h_a$-doubling for all $j$, 
which cover $\bigcup_{i\in I_{HC}}S_i$
with finite overlap, with $\ell(\wh P_j)\leq\ell(Q)/100$, so that 
$$D_\mu(\wh P_j,Q)\leq C(\eta)\theta_\mu(Q)^2.$$
By Tchebytchev, it is easy to check that there exists a subfamily $\{\wh P_j\}_{j\in J}\subset 
\{\wh P_j\}_j$ such that for each $j\in J$,
$$\mu\bigl\{x\in P_j:\,K_{\mu,J(P_j)+10}\chi_E(x) - K_{\mu,
J(Q)-4}\chi_E(x)\geq \ve_0\theta_\mu(Q)^2\bigr\} \geq
C_{10}^{-1}\,\mu(P_j),$$
with 
$$\mu\biggl(\,\bigcup_{j\in J} \wh P_j\biggr)\geq C^{-1}\mu(Q).$$
Notice that for $x$ in a big piece of each square $\wh P_j$, $j\in J$,
\begin{align*}
K_{\mu,J(\wh P_j)+10}\chi_E(x) - K_{\mu,
J(Q_0)-4}\chi_E(x) &\geq 
K_{\mu,J(\wh P_j)+10}\chi_E(x) - K_{\mu,
J(Q)-4}\chi_E(x) \\ & \geq
\ve_0\theta_\mu(Q)^2\geq C^{-1}{\wt\delta}^2\ve_0\theta_\mu(Q_0)^2.
\end{align*}
Thus if we choose $\wt\ve_0$ small enough so that $\wt\ve_0\ll{\wt\delta}^2\ve_0$, and
we also take $\beta\ll C_{10}^{-1}$, then 
one can find squares $\{P_i^j\}_i$ contained in $\frac34P_j$ which cover $\frac12
P_j$ with $D_\mu(P_i^j,P_j)=C\theta_\mu(Q_0)$, so that
the family $\bigcup_{j\in J}\{P_i^j\}_{i}$ satisfies all the required properties.
We leave the details for the reader.
\end{proof}
\medskip

For $Q\in\ttt_R(\mu)$,
we denote by $\tree_R(Q)$ the family of dyadic squares from $\tree(R)$ 
that are contained in 
$3Q$ and contain either some of the sixteen dyadic squares of equal length that form one square from $\sss(Q)$, or some point from $\wt G(Q)$.

\begin{lemma}\label{lem46}
Given $R\in\sel_S(\mu)$, 
for each $Q\in\ttt_R(\mu)$,
\begin{equation}\label{eqtree2}
\sum_{P\in\tree_R(Q)}\ve_a(P)^2\mu(P)\leq C(A,\delta,\ve_0,\eta)\ve_a(Q)^2\mu(Q).
\end{equation}
\end{lemma}

Recall that $\tree_R(Q)\subset\tree(R)$. Thus, in a sense all the relevant squares in 
the sum above are of type $S$, which originate short trees. 

\begin{proof}
For $x\in\C$ we define the function
\begin{equation}\label{eqsum84}
F(x) = \sum_{k\in\Z} \max_{P\sim (x,k)} \ve_a(P)^2,
\end{equation}
where the notation $P\sim (x,k)$ means that $P$ is a $4$-dyadic square containing $x$, with $\ell(P)=2^{-k}$
such that some of the $16$ dyadic squares of equal side length that form $P$ belongs
to $\tree(Q)$. 
From the definition, it is easy to check that $F(x)=0$ if $x\not\in CQ$, for some fixed
$C>1$. To prove the lemma we will show that $\|F\|_{L^1(\mu)}\leq C\ve_a(Q)^2\mu(Q)$.

For $\lambda>0$, denote $$\Omega_\lambda = \{x\in\C:F(x)>\lambda\ve_a(Q)^2\}.$$
For $x\in\Omega_\lambda$, let $k_x$ be the minimal integer such that
$$\sum_{k\leq k_x} \max_{P\sim (x,k)} \ve_a(P)^2>\lambda\ve_a(Q)^2,$$
and let $\wt S_x\sim (x,k_x)$ be such that $\ve_a(\wt S_x)$ is maximal.
Let $S_x$ be the smallest $4$-dyadic square such that $\frac12 S_x$ is $h_a$-doubling and
contains $\wt S_x$.
If $\ell(S_x)>\ell(Q)$, from  
 Lemma \ref{lemnodob}, it follows easily that 
$$D_\mu(\wt S_x,Q)\leq C_{11}\ve_a(Q)^2,$$
where $C_{11}$ may depend on $\wt A,\wt \delta$\ldots. This implies that
$$F(x)\leq C_{12}\ve_a(Q)^2,$$
with $C_{12}$ depending on $C_{11}$.

Assume that $\lambda>C_{12}$. In this case,
  $\ell(S_x)\leq\ell(Q)$. From Lemma \ref{lemnodob} and the fact
that for all $P\in\tree_R(Q)$, $\ve_a(P)\leq C(\wt A)\ve_a(Q)$, one infers that 
$$D_\mu(\wt S_x,S_x)\leq C_{11}\ve_a(Q)^2.$$
From this estimate, one deduces that
$$F(y)>(\lambda - C_{13})\ve_a(Q)^2\qquad \mbox{ for all $y\in S_x$,}$$
with $C_{13}$ depending on $C_{11}$. So we have
$$\Omega_\lambda \subset \bigcup_x S_x\subset \Omega_{\lambda - C_{13}}.$$
From the doubling properties of the squares $S_x$, there exists a subfamily
$\{S_{x_i}\}$ such that the squares from this family are pairwise disjoint
and 
$$\mu\biggl(\bigcup_iS_{x_i}\biggr)\geq C^{-1}\mu\biggl(\bigcup_x S_x\biggr).$$
We may cover each square $\frac12 S_{x_i}$ with a family of squares $\{T^i_j\}_j$
such that each $\frac12 T^i_j$ is $h_a$-doubling. By Lemma \ref{lem45},
for each $T^i_j$ there exists some subset $A^i_j$ such that
$\mu(A^i_j)\geq C^{-1}\mu(T^i_j)$ and
$$F(x)\leq \lambda + C$$
(because of (b) in Lemma \ref{lem45} and because $D_\mu(T^i_j,S_{x_i})\leq C$). Using some 
appropriate covering theorem (like Vitali), one
infers that for each $i$ there exists $A_i\subset S_{x_i}$ such that
$\mu(A_i)\geq C^{-1}\mu(S_{x_i})$ and
$F(x)\leq \lambda + C_{14}$ on $A_i$. 

Since $A_i\subset \Omega_{\lambda - C_{13}} \setminus
 \Omega_{\lambda + C_{14}}$, we deduce
$$\mu(\Omega_{\lambda - C_{13}}) - \mu(\Omega_{\lambda + C_{14}})
\geq \sum_i\mu(A_i)\geq C^{-1} \sum_i\mu(S_{x_i})\geq C_{15}^{-1}\mu(\Omega_\lambda).$$
Thus,
\begin{equation}\label{eqgeom1}
\mu(\Omega_{\lambda + C_{14}})\leq \mu(\Omega_{\lambda - C_{13}}) - C_{15}
^{-1}\mu(\Omega_\lambda).
\end{equation}
We have
$$\|F\|_{L^1(\mu)} \leq \ve_a(Q)^2\int_0^\infty \mu(\Omega_\lambda)\,d\lambda\leq
C_{12}\ve_a(Q)^2\mu(CQ) + \ve_a(Q)^2\int_{C_{12}}^\infty \mu(\Omega_\lambda)\,d\lambda.$$
We may assume that $C_{12},C_{13},C_{14}$ are integer constants. Then,
$$\int_{C_{12}}^\infty \mu(\Omega_\lambda)\,d\lambda  \leq \sum_{k\geq C_{12}}\mu(\Omega_k).$$
From \rf{eqgeom1},  one can easily check that $\mu(\Omega_k)$ decreases geometrically 
as $k\to\infty$. Indeed, \rf{eqgeom1} this implies that
$$\mu(\Omega_{\lambda+C_{14}}) (1+C_{15}^{-1})\leq \mu(\Omega_{\lambda - C_{13}}).$$
That is,
$\mu(\Omega_{\lambda+C_{14}+C_{13}}) \leq \alpha\,\mu(\Omega_\lambda),$
with $\alpha=(1+C_{15}^{-1})^{-1}$.
Then we deduce that
$$\sum_{k\geq C_{12}}\mu(\Omega_k)\leq C\mu(Q),$$
and then the lemma follows.
\end{proof}

\medskip

\begin{proof}[\bf Proof of Lemma \ref{lempac0}]
Let $\ttt_R(\mu)$ be the family described in Lemma \ref{lembilip}.
Since
$c^2_\mu(x)\leq 1$ for all $x\in\C$, we have
\begin{equation}\label{eqsum83}
\sum_{Q\in\ttt_R(\mu)}\theta_\mu(Q)^2\mu(Q)\leq C(\wt A,\wt \delta, \wt\ve_0)\mu(R).
\end{equation}

Notice that 
$$\tree(R)\subset \bigcup_{Q\in\ttt_R(\mu)}\tree_R(Q).$$
By the preceding lemma, for each $Q\in\ttt_R(\mu)$,
$$
\sum_{P\in\tree_R(Q)}\ve_a(P)^2\mu(P)\leq C(A,\delta,\ve_0,\eta)\ve_a(Q)^2\mu(Q).
$$
 Together with \rf{eqsum83} and the fact that
$\ve_a(Q)\approx \theta_\mu(Q)$ for $Q\in\ttt_R(\mu)$, this yields
\begin{align*}
\sum_{P\in\tree(R)}\ve_a(P)^2\,\mu(P)& \leq \sum_{Q\in\ttt_R(\mu)}\sum_{P\in\tree_R(Q)}\ve_a(P)^2\mu(P)\\ & \leq C\sum_{Q\in\ttt_R(\mu)}\ve_a(Q)^2\,\mu(Q)\leq 
C_{16}\mu(R),
\end{align*}
with $C_{16}$ depending on all the parameters $\eta,A, \delta,  \ve_0$.
\end{proof}



\bigskip

\section{Construction of the measure $\nu$ for the proof of Theorem \ref{distorgamma}}
\label{secconsnu}

In this section we will prove the estimate \rf{eqdistor} following the ideas explained in
Section \ref{secstrat}. To this end, using the corona decomposition of 
the preceding section 
we will construct a measure $\nu$ supported on $\phi(E)$ such that
$\nu(\phi(E))\approx \gamma(E)^{\frac{2K}{K+1}}$ and 
$\dot W_{ \frac{2K}{2K+1},\frac{2K+1}{K+1}}^\nu(x)\lesssim 1$ for all $x\in\phi(E)$.

\subsection{Preliminaries}

Next lemma is just a rescaled version of Lemma \ref{mainlem}

\begin{lemma}\label{mainlemrescal}
Let $\mu$ be a finite continuous (i.e. without point masses) Borel measure on $\C$.
For $a>0$ small enough (depending only on $K$), denote
$$\ve_a(x,t) = \ve_a(B) = \frac1t \int \psi_a\Bigl(\frac{y-x}t\Bigr)d\mu(y),\qquad h_a(x,t) = t\,\ve_a(x,t),$$
with $\psi_a$ as in \rf{eqpsia}. Let $\phi:\C\to\C$ be a $K$-quasiconformal mapping and set
\begin{equation}\label{defd43}
\ve(x,t) = \ve_a(\phi^{-1}(B(x,t)))^{\frac{2K}{K+1}},\qquad
h(x,t) = t^{\frac2{K+1}}\, \ve(x,t).
\end{equation}
If $E\subset \C$ is a compact subset contained in a ball $B$,
 we have
$$\frac{M^{h_a}(E)}{\diam(B)} \leq C(K) \,\biggl(\frac{M^h(\phi(E))}{\diam(\phi(B))^{\frac2{K+1}}}\biggr)^{\frac{K+1}{2K}}.$$
\end{lemma}

\vspace{3mm}
Obviously, an analogous version holds with squares instead of balls.

\begin{lemma}\label{lemtec53}
Under the same hypotheses and notation of Lemma \ref{mainlemrescal}, given any square $Q\subset\C$, if
$$M^{h_a}(Q\cap E)\geq C_{17} h_a(Q)$$
with $C_{17}>0$,
then
$$M^h(\phi(Q\cap E))\geq C_{18} h(\phi(Q)),$$
with $C_{18}>0$ depending only on $C_{17}$ and $K$.
\end{lemma}

\begin{proof}
We have
$$\frac{M^{h_a}(Q\cap E)}{\ell(Q)}\geq C\ve_a(Q).$$
Then, by Lemma \ref{mainlemrescal} applied to $Q\cap E$ and $Q$ instead of $B$,
$$\frac{M^h(\phi(Q\cap E))}{\diam(\phi(Q))^{\frac2{K+1}}}\geq C\ve_a(Q)^{\frac{2K}{K+1}},$$
which is equivalent to $M^h(\phi(Q\cap E))\geq C_{18} h(\phi(Q)).$
\end{proof}

\bigskip

%

Notice that the assumption $M^{h_a}(Q\cap E)\geq C_{17} h_a(Q)$ is satisfied by the squares
from $\sel(\mu)$ in the corona construction in Section \ref{seccorona0}, since
$$M^{h_a}(Q\cap E)\geq C\,\mu(Q)\geq C\,h_a(Q),$$
by Lemma \ref{lem2.2} and $Q$ is $h_a$-doubling.

\bigskip


To construct $\nu$ we will use the structure of $4$-dyadic squares from $\sel(\mu)$
introduced in the preceding section.
We denote $\sel(\nu):=\phi(\sel(\mu))$, and analogously for other families of squares such as 
$\sel_S(\nu)$, $\sel_L(\nu)$, etc.
Given a $4$-dyadic $\phi$-square $R\in\sel(\nu)$, we denote $\cT_\nu(R):=\phi(\cT(\phi^{-1}(R)))$ and $F_\nu(R) := \phi(F(\phi^{-1}(R))$ (see Subsections \ref{subsels} and \ref{subsell}),
 and also
$G_\nu(R):= \phi(G(\phi^{-1}(R))$ (see \rf{defgg}).

We will define the values of $\nu$ on the squares of $\sel(\nu)$ (and/or other subsets
like $G_\nu(R)$ or $F_\nu(R)$, for $R\in\sel(\nu)$) inductively. To start with, we set
$$\nu(\phi(R_0))= M^h(\phi(E)).$$
Recall that $R_0$ is the biggest $4$-dyadic square from $\sel(\mu)$, so that $E$ is contained
 in one of the $4$ dyadic squares that form $\frac12 R_0$.  

In the algorithm of construction of $\nu$,
after fixing $\nu(R)$ for some $R\in\sel(\nu)$, then one defines the values 
of $\nu(P)$ for all $P\in\cT_\nu(R)$ , as well as in $G(R)\cup F(R)$.
To this end, it is necessary to distinguish two cases, according to wether $R$ is 
of type L or S. In Subsection \ref{sub55} we consider the case where $R$ is of 
type L, and in Subsection \ref{sub56}, the one where $R$ is of type S.

To simplify notation, in the rest of the paper given a square $Q$, we denote
$Q'=\phi(Q)$. Usually, the letters $P,Q,R$ will be reserved for squares, and $P',Q',R'$
for $\phi$-squares.

\bigskip


\subsection{Definition of $\nu$ on $\cT_\nu(R')$ when $R'\in \sel_L(\nu)$}\label{sub55}

Suppose first that 
$$\mu(G(R))< \frac14\,\mu\Bigl(\frac12R\Bigr).$$


\subsubsection{First step: definition of $\nu(3 S_j'))$, $j\in I_R$}\label{susu3}

Recall the definition of the squares $S_j$, $j\in I_R$, in \rf{eqsj7}.
In particular, recall that the squares $5S_j$, $j\in I_R$, are pairwise disjoint,
contained in $R$, so that $S_j\cap\Gamma_R\neq\varnothing$ 
(where $\Gamma_R$ is a chord arc curve or a union of at most $N_0$ chord arc curves), and moreover,
$$\sum_{j\in I_R}\mu(S_j)\geq C^{-1}\mu(R).$$

We have
\begin{equation}\label{eqdak33}
\sum_{j\in I_R}\ell(S_j)\geq C_{19}\diam(\Gamma_R),
\end{equation}
with $C_{19}$ depending on $A,\delta$ (but not on $\eta$), because by the property (c) from Lemma 
\ref{lemcorba},
\begin{align*}
\sum_{j\in I_R} \ell(S_j) &= \sum_{j\in I_R} \mu(S_j)\, \theta_\mu(S_j)^{-1} 
\approx \sum_{j\in I_R}\mu(S_j) \,\theta_\mu(R) \geq C^{-1}\,\mu(R)\,\theta_\mu(R) = C\,\ell(R).
\end{align*}
Then,
from Lemma \ref{lemdistca4} applied to suitable arcs contained in $3S_j\cap\Gamma_R$
and quasisymmetry we deduce
\begin{equation}\label{eqdak34} 
\sum_{j\in I_R}\ell(S_j')^\alpha \geq C_{20}\diam(\phi(\Gamma))^\alpha\approx \ell(R')^\alpha,
\end{equation}
where $\alpha>2/(K+1)$ depends only on $K$, and $C_{20}$ depends on $C_{19}$, $K$, and the parameters of the corona construction 
(except $\eta$).
In fact, a similar argument shows that the set $G':=\bigcup_{j\in I_R}3S_j'$ satisfies
$$H_\infty^\alpha(G')\geq C_{21}\diam(\phi(\Gamma))^\alpha.$$
To see this, just take into account that any union of sub-arcs of $\Gamma_R$ that contains
$\Gamma_R\cap \bigcup_{j\in I_R}3S_j$ satisfies an estimate analogous to \rf{eqdak33}, and
thus we would get an estimate similar to \rf{eqdak34} for the corresponding images.

By Frostman Lemma, we deduce that there exists some measure $\sigma$ supported on
$G'$ such that
$\sigma(G') = H^\alpha_\infty (G')$
and
\begin{equation}\label{condsigma}
\sigma(B(x,r))\leq Cr^\alpha\qquad \mbox{for all $x\in\C,\,r>0.$}
\end{equation}
We define
$$\nu(3S_j') = \frac{\sigma(3S_j')}{\sigma(G')}\,\nu(R')$$
(recall that we assume that $\nu(R')$ has already been fixed),
and moreover, 
$$\nu(5S_j'\setminus 3S_j')=0.$$
It is easy to check then that if $P'$ is a $\phi$-square concentric with $S_j'$ which contains $3S_j'$ and is contained in
$3R'$, then
$$\nu(P') \leq C\,\frac{\ell(P')^\alpha}{\sigma(G')}\,\nu(R')\approx
\frac{\ell(P')^\alpha}{\ell(R')^\alpha}\,\nu(R').$$
This can be proved using the condition \rf{condsigma} and the fact that the $\phi$-squares
$5S_j'$ are disjoint, for instance.
Therefore,
\begin{equation}\label{eqf36}
\frac{\nu(P')}{\ell(P')^{\frac2{K+1}}} \leq C\,\biggl(\frac{\ell(P')}{\ell(R')}\biggr)^{\alpha-\frac2{K+1}}\,\frac{\nu(R')}{\ell(R')^{\frac2{K+1}}}.
\end{equation}

\bigskip


\subsubsection{Second step: definition of $\nu(P')$ for $P'\in \cT_\nu(R')$}
\label{subsub69}

Recall that for each $j\in I_R$, there is a family $\cP'=\cT_{j,\nu}(R')\cup F_j(R')$ of $\phi$-squares or points $P'$ which are contained in 
$3S_j'$, such that different $\phi$-squares $5P'$ are pairwise disjoint and 
$$\mu\biggl(\,\bigcup_{P\in \cP} P\biggr) \geq C^{-1}\mu(5S_j).$$

We denote 
$$G_j' = \bigcup_{P'\in \cP'} P'.$$
Our next task consists in distributing the measure $\nu_{| 3S_j'}$ among the $\phi$-squares
$P'$ above. In particular,  we will have
$$\nu(3S_j') = \nu(G_j').$$
To define the appropriate values of $\nu(P')$, for $P'\in \cP'$, we will follow an algorithm
inspired by the proof of Frostman Lemma ``from above''. Let $Q_0'$ be a 
dyadic $\phi$-square contained in $3S_j'$, with $\ell(Q_0)=\ell(S_j')$, such that 
$\mu(Q_0'\cap G_j')$ is maximal.
We set
\begin{equation}\label{defq0}
\tau(Q_0')=\nu(3S_j'),
\end{equation}
where $\tau$ should be considered as a preliminary version of $\nu$ on some $\phi$-squares contained in $3S_j'$.
If $Q'$ is a dyadic $\phi$-square contained in $Q_0'$ such that $\tau(Q')$ has already been
defined and $Q'$ is not contained in any $\phi$-square from $\cP'$ (in particular, 
$Q'\not\in \cP'$), then we define $\tau$ on the sons $Q_1',\ldots,Q_4'$ of $Q'$ as follows:
\begin{equation}\label{eqtau}
\tau(Q_i') = \frac{M^h(Q_i'\cap G_j')}{\sum_{k=1}^4 M^h(Q_k'\cap G_j')}\,
\tau(Q').
\end{equation}
Clearly, we have $\sum_{1\leq i \leq 4}\tau(Q_i')=\tau(Q').$
At the end of the algorithm,
for each $P'\in\cP'$ there is a pairwise disjoint family of dyadic $\phi$-squares $T_1',\ldots,T_m'$ such that $P'=\bigcup_{1\leq i \leq m} T_i'$ so that $\tau(T_i')$ has been
defined. We set
$$\nu(P') = \sum_{1\leq i\leq m}\tau(T_i').$$

\bigskip


\subsubsection{The case $\mu(G(R))\geq \frac14\,\mu(\frac12R)$.}
The arguments for this case are very similar  to the ones of Subsection \ref{susu3}.
Instead of $\phi$-squares $S_j'$, we have now points from $G_\nu(R')$.
So we construct $\nu_{|3R'}$ so that it is concentrated on $G_\nu(R')$, arguing as 
in Subsection \ref{susu3}. We leave
the details for the reader.

\bigskip


\subsection{Definition of $\nu$ on $\cT_\nu(R')$ when $R'\in \sel_S(\nu)$}\label{sub56}

\subsubsection{The case $\mu(F(R))\leq \frac14\,\mu(\frac12R)$.}

Recall the definition of the family of squares $\cT(R)$. For $P\in\cT(R)$,
  Set
$$U(P) = \sum_{Q\in\cD:P\subset Q\subset R}\ve_a(Q)^2
= \sum_{Q'\in\phi\cD:P'\subset Q'\subset R'}\ve(Q')^{\frac{K+1}K}.$$
By Lemma \ref{lempac0},
\begin{equation}\label{eqf56}
\sum_{P\in \cT(R)}U(P)\,\mu(P)\leq C(\eta)\mu(R).
\end{equation}
Since $\mu\left(\bigcup_{P\in\cT(R)}P\right)\approx \mu(R)$,
by Tchebytchev there is a subfamily $\cT_1(R)\subset \cT(R)$
such that 
\begin{equation}\label{eqft89}
\mu\biggl(\bigcup_{P\in\cT_1(R)}P\biggr)\approx \mu(R) \quad\mbox{ and } \quad
U(P)\leq 2C(\eta)\quad\mbox{ for every $P\in\cT_1(R)$.}
\end{equation}

For $P'\in\cT_\nu(R')\setminus \cT_{1,\nu}(R')$, we set
$$\nu(P')=0.$$

To define $\nu$ on the $\phi$-squares from $\cT_{1,\nu}(R')$ we follow the same 
algorithm of Subsection \ref{subsub69}: we denote 
$$G' = \bigcup_{P'\in \cT_{1,\nu}(R')} P'.$$
Let $Q_0$ one of the $16$ dyadic squares that form $R$ such that $\mu(Q_0\cap G)$ is maximal.
We set
$$\tau(Q_0')=\nu(R'),$$
where $\tau$ should be considered as a preliminary version of $\nu$ on some $\phi$-squares contained in $R$.
If $Q'$ is a dyadic $\phi$-square contained in $Q_0'$ such that $\tau(Q')$ has already been
defined and $Q'$ is not contained in any $\phi$-square from $\cT_{\nu,1}(R')$ 
(in particular, 
$Q'\not\in \cT_{\nu,1}(R')$), then we define $\tau$ on the sons $Q_1',\ldots,Q_4'$ of $Q'$ as follows:
$$\tau(Q_i') = \frac{M^h(Q_i'\cap G')}{\sum_{k=1}^4 M^h(Q_k'\cap G')}\,
\tau(Q').$$
Clearly, we have $\sum_{1\leq i \leq 4}\tau(Q_i')=\tau(Q').$
At the end of the algorithm,
for each $P'\in\cT_{\nu,1}(R')$ there is a pairwise disjoint family of dyadic $\phi$-squares $T_1',\ldots,T_m'$ such that $P'=\bigcup_{1\leq i \leq m} T_i'$ so that 
$\tau(T_i')$ has been
defined. We set
$$\nu(P') = \sum_{1\leq i\leq m}\tau(T_i').$$

\bigskip
\subsubsection{The case $\mu(F(R))\geq \frac14\,\mu(\frac12R)$.}

This case is treated as the preceding one, with the convention that the points from 
$F(R)$ are the same as squares with zero side length.

\bigskip

\subsection{Estimate of the Wolff potential of $\nu$ on trees of type L} \label{subsub67}

Recall that $\eta$ is the constant in the Definition \ref{def58} of short and long trees, and
that $\eta\ll1$.

\begin{lemma}\label{lemenertree}
Let $R'\in\sel_L(\nu)$. If $\nu(R')\leq bh(R')$, then
\begin{equation}\label{eqff9}
\nu(P')\leq C_{22}\,b\eta^{\alpha-\frac2{K+1}} h(P')\qquad\mbox{for all $P'\in\cT_\nu(R')$}.
\end{equation}
Also, if $Q'$ is a $\phi$-square such that $P'\subset Q'\subset 3R'$ for some
$P'\in\cT_\nu(R')
\cup G_\nu(R')\cup F_\nu(R')$, then
\begin{equation}\label{eqff8}
\nu(Q')\leq C_{22}bh(Q').
\end{equation}
Moreover, for each $P'\in\cT_\nu(R')
\cup G_\nu(R')\cup F_\nu(R')$
\begin{equation}\label{eqff10} 
\sum_{Q'\in\phi\cD:P'\subset Q'\subset R'}\Biggl(\frac{\nu(3Q')}{\ell(Q')^{\frac2{K+1}}}\Biggr)^{\frac{K+1}K}\leq
Cb^{\frac{K+1}K}.
\end{equation}
\end{lemma}

Let us remark that the constant $C_{22}$ is independent from $\eta$.

One of the key points in this lemma is that, by \rf{eqff9},
$$\frac{\nu(P')}{h(P')} \ll \frac{\nu(R')}{h(R')}$$
if $P'\in\cT_\nu(R')$, for $R'\in\sel_L(\nu)$, assuming that $\eta$ is chosen small enough.
This is due to improved distortion estimates for sub-arcs of chord arc curves. This point
plays an essential role in our proof of Theorem \ref{distorgamma}.

\begin{proof}
For simplicity we will only consider the case where $\mu(G(R))<\mu(\frac12R)/4$, and that
$\mu(F_j(R))\leq\mu(S_j)/2$ for all $j$. By arguments quite similar to the ones below, one can
deal with the other cases, considering points as squares of zero side length.

Recall that if $Q'$ is a $\phi$-square concentric with $S_j'$ which contains $5S_j'$ and is contained in $3R'$, by \eqref{eqf36},
\begin{equation}\label{eqff89}
\frac{\nu(Q')}{\ell(Q')^{\frac2{K+1}}} \leq C\,\biggl(\frac{\ell(Q')}{\ell(R')}\biggr)^{\alpha-\frac2{K+1}}\,\frac{\nu(R')}{\ell(R')^{\frac2{K+1}}}
\leq  C\,b\biggl(\frac{\ell(Q')}{\ell(R')}\biggr)^{\alpha-\frac2{K+1}}\ve(R'),
\end{equation}
since $\nu(R')\leq b\ell(R')^{\frac2{K+1}}\ve(R')$.
By construction, $\ve(Q')\approx\ve(R')$ (by Remark \ref{remdob}) and so we get
\begin{equation}\label{eqft6}
\nu(Q')\leq C\,b\biggl(\frac{\ell(Q')}{\ell(R')}\biggr)^{\alpha-\frac2{K+1}}\ell(Q')^{\frac2{K+1}}\ve(Q') = C\,b\biggl(\frac{\ell(Q')}{\ell(R')}\biggr)^{\alpha-\frac2{K+1}}h(Q').
\end{equation}
Recall that the subset $G_j' =\bigcup_{P'\in \cP'} P'$ of $5S_j'$ and the $\phi$-square $Q_0'$ in \eqref{defq0} satisfy
$$M^{h_a}(Q_0\cap G_j) \geq C^{-1} \mu(G_j)
 \geq C^{-1}\mu(5S_j).$$
Since $\ve_{a}(R) \approx \ve_{a}(5S_j)\approx \theta_{\mu}(5S_{j})$, this implies
$$M^{h_a}(Q_0\cap G_j)  \geq C h_{a}(5S_{j}),$$
and then, by Lemma \ref{lemtec53},
$$M^h(Q_0'\cap G_j')\geq Ch(5S_j').$$
Thus, by \eqref{eqft6}, 
$$\tau(Q_0')=\nu(5S_{j}')\leq
 C_{23}\,b\eta^{\alpha-\frac2{K+1}}M^h(G_j'\cap Q_0').$$
We claim that all the numbers $\tau(Q')$ in \eqref{eqtau} satisfy the analogous inequality
\begin{equation}\label{eqf447}
\tau(Q')\leq
 C_{23}\,b\eta^{\alpha-\frac2{K+1}}M^h(G_j'\cap Q').
\end{equation}
To prove this, it is enough to show that if this hods for some $\phi$-square $Q'$, then it also holds for its
sons $Q_i'$, $1\leq i \leq 4$, assuming that $Q'$ is not contained in any $\phi$-square from $\cP'$ (this was the necessary condition to define $\tau(Q_i')$, $1\leq i \leq 4$).
By \eqref{eqtau}, we get
\begin{align*}
\tau(Q_i') & = \frac{M^h(Q_i'\cap G_j')}{\sum_{k=1}^4 M^h(Q_k'\cap G_j')}\,
\tau(Q')\\ & \leq \frac{M^h(Q_i'\cap G_j')}{M^h(Q'\cap G_j')}\,
\tau(Q') \leq C_{23}\,b\eta^{\alpha-\frac2{K+1}}M^h(Q_i'\cap G_j'),
\end{align*}
and so \eqref{eqf447} holds. From this estimate one easily obtains
$$\nu(Q')\leq C\,b\eta^{\alpha-\frac2{K+1}}M^h(Q_i'\cap G_j')$$
for $Q'$ contained in $5S_j'$ and containing some $P_0'\in\cT(R')$.
Indeed,
$$\nu(Q')\leq \nu\Biggl(\bigcup_{P'\in\cT_\nu(R'):P'\cap Q'\neq\varnothing} P'\Biggr).$$
From the fact that the $\phi$-squares $5P'$, $P'\in\cT_\nu(R')$ are pairwise disjoint, it follows that if $Q'$ intersects another $\phi$-square $P'\in\cT_\nu(R')$, 
then $\ell(Q)\geq \ell(P)$. As a consequence, all $\phi$-squares $P'\in\cT_\nu(R')$ 
intersecting
$Q'$ are contained in $3Q'$. Thus, there are at most four dyadic $\phi$-squares 
$L_1',\ldots,L_4'$
with $\ell(L_i)\leq 2\ell(3Q)$ that contain all $\phi$-squares $P'\in\cT_\nu(R')$ intersecting
$Q'$. Then, by construction we have
$$\nu(Q') \leq \sum_{i=1}^4 \tau(L_i')\leq \sum_{i=1}^4
C_{23}\,b\eta^{\alpha-\frac2{K+1}}M^h(L_i'\cap G_j') \leq C\,b\eta^{\alpha-\frac2{K+1}}
h(Q').$$
From \eqref{eqft6} and the preceding inequality, one easily deduces \rf{eqff9} and \rf{eqff8}.

\medskip

To prove \rf{eqff10}, it is enough to show that for each $P'\in\cT_\nu(R')$
$$\sum_{Q'\in\phi\cD:P'\subset Q'\subset R'}\Biggl(\frac{\nu(3Q')}{\ell(Q')^{\frac2{K+1}}}\Biggr)^{\frac{K+1}K}\leq
Cb^{\frac{K+1}K}.$$
Suppose that $P'\subset S_j'$. Then we split the sum above as follows:
$$\sum_{Q'\in\phi\cD:P'\subset Q'\subset R'}\Biggl(\frac{\nu(3Q')}{\ell(Q')^{\frac2{K+1}}}\Biggr)^{\frac{K+1}K}
= \sum_{Q'\in\phi\cD:S_j'\subsetneq Q'\subset R'}\cdots + \sum_{Q'\in\phi\cD:P'\subset Q'\subset S_j'}\cdots
=:T_1 + T_2.$$
To estimate the first sum recall that by \rf{eqff89} we have
$$\frac{\nu(3Q')}{\ell(Q')^{\frac2{K+1}}} \leq  Cb\biggl(\frac{\ell(Q')}{\ell(R')}
\biggr)^{\alpha-\frac2{K+1}}\ve(R').$$
Then it follows that
$T_1\leq C\bigl(b \,\ve(R')\bigr)^{\frac{K+1}K}.$
Recalling that $\ve(R') = 
\ve_a(R)^{\frac{2K}{K+1}}\leq C$, we infer that
$$T_1\leq Cb^{\frac{K+1}K}.$$

To estimate $T_2$ we use that 
$$\frac{\nu(3Q')}{\ell(Q')^{\frac2{K+1}}}\leq Cb \,\ve(Q')$$
 and the fact that $D_\mu(P,S_j)\leq C\ve_a(S_j)^2\leq C\ve_a(R)^2$, by construction, and so
$$\sum_{Q'\in\phi\cD:P'\subset Q'\subset S_j'}\ve(Q')^{\frac{K+1}K}
\approx D_\mu(P,S_j)
\leq C\ve(R')^{\frac{K+1}K}\leq C,$$
and then we deduce that
$T_2\leq Cb^{\frac{K+1}K}.$
\end{proof}

\bigskip


\subsection{Estimates for the Wolff potential of $\nu$ on trees of type S} \label{subsub68}

Recall Definition \ref{def51} of $\tree(R)$ for $R\in \sel_S(\mu)$.
We denote $\tree_\nu(R')= \phi(\tree(R))$.

\begin{lemma}\label{lemenertree2}
Let $R'\in\sel_S(\nu)$. If $\nu(R')\leq bh(R')$, then
\begin{equation}\label{eqfff8}
\nu(Q')\leq C_{24}bh(Q')\qquad\mbox{for all $Q'\in\tree_\nu(R')$}
\end{equation}
and, for each $P'\in\cT_\nu(R')\cup F_\nu(R')$,
\begin{equation}\label{eqfff10}
\sum_{Q'\in\phi\cD:P'\subset Q'\subset R'}\Biggl(\frac{\nu(3Q')}{\ell(Q')^{\frac2{K+1}}}\Biggr)^{\frac{K+1}K}\leq
C(\eta)b^{\frac{K+1}K}.
\end{equation}
\end{lemma}

The constant $C_{24}$ above is independent of $\eta$ in the definition of 
``long trees''.

\begin{proof}
The arguments to prove \rf{eqfff8} are very similar to the ones used in 
Lemma \ref{lemenertree} to show that analogous estimates hold for the squares contained in the squares $S_j$, taking into account that
$\mu\left(\bigcup_{P\in\cT_1(R)}P\right)\approx \mu(R)$, by \rf{eqft89}.
So we skip the details.

On the other hand, from \rf{eqfff8} we also infer that
$$\nu(3Q')\leq Cbh(Q') \qquad\mbox{for all $Q'\in\tree_\nu(R')$}.$$
Then, \rf{eqfff10} follows from this estimate and the 
fact that for every $P'\in\cT_\nu(R')$,
$$\sum_{Q'\in\phi\cD:P'\subset Q'\subset R'}\ve(Q')^{\frac{K+1}K} \leq C(\eta),$$
by \rf{eqft89}.
\end{proof}

\bigskip


\section{Proof of Theorem \ref{distorgamma}} \label{secproof}

Recall that the measure $\nu$ supported on $\phi(E)$ that we have constructed in 
Section~\ref{secconsnu} satisfies
$$\nu(\phi(E)) = M^h(\phi(E))\gtrsim \mu(E)^{\frac{2K}{K+1}} \approx \gamma(E)^{\frac{2K}{K+1}}.$$
Thus the theorem follows if we show that
\begin{equation}\label{eqfi99}
\sum_{Q'\in\phi\cD:x\in Q'} 
\Biggl(\frac{\nu(3Q')}{\ell(Q')^{\frac2{K+1}}}\Biggr)^{\frac{K+1}K}\leq C\quad \mbox{ for all $x\in \supp(\nu)$.}
\end{equation}
Let $\{R_n'\}_{n\geq 0}$ be the collection of $\phi$-squares from $\sel(\nu)$ which contain $x$.
We assume that $\ell(R_n)>\ell(R_{n+1})$ for all $n$. We split the preceding sum as follows:
\begin{align}\label{eqf51}
\sum_{Q'\in\phi\cD:x\in Q'} 
\left(\frac{\nu(3Q')}{\ell(Q')^{\frac2{K+1}}}\right)^{\frac{K+1}K} & =
\sum_{Q'\in\phi\cD: R_0'\subsetneq Q'} \Biggl(\frac{\nu(3Q')}{\ell(Q')^{\frac2{K+1}}}\Biggr)^{\frac{K+1}K} \nonumber \\
&\quad + \sum_{n\geq0} \sum_{Q'\in\phi\cD: R_{n+1}'\subsetneq Q'\subset R_n'}\Biggl(\frac{\nu(3Q')}{\ell(Q')^{\frac2{K+1}}}\Biggr)^{\frac{K+1}K}=: S_1 + S_2.
\end{align}
To estimate the sum $S_1$ on the right side, one only needs to take into account
that
$$
\Biggl(\frac{\nu(\phi(E))}{\ell(Q')^{\frac2{K+1}}}\Biggr)^{\frac{K+1}K} =
\frac{\ell(R_0')^{\frac2K}}{\ell(Q')^{\frac2K}}\,
\Biggl(\frac{\nu(\phi(E))}{\ell(R_0')^{\frac2{K+1}}}\Biggr)^{\frac{K+1}K} \leq
\frac{\ell(R_0')^{\frac2K}}{\ell(Q')^{\frac2K}}\,
\ve(R_0')^{\frac{K+1}K}\leq C\,\frac{\ell(R_0')^{\frac2K}}{\ell(Q')^{\frac2K}},$$
and summing over those $Q'\in\phi\cD$ containing $R_0'$, we get $S_1\leq C$.

To deal with $S_2$, observe that, by Lemmas \ref{lemenertree} and \ref{lemenertree2},
\begin{align*}
\sum_{n\geq0} \sum_{Q'\in\phi\cD: R_{n+1}'\subsetneq Q'\subset R_n'}\Biggl(\frac{\nu(3Q')}{\ell(Q')^{\frac2{K+1}}}\Biggr)^{\frac{K+1}K} \leq C(\eta) \sum_{n\geq0} 
\left(\frac{\nu(R_n')}{h(R_n')}\right)^{\frac{K+1}K}.
\end{align*}
 Lemma \ref{lemenertree} tells us
that if $R_n'\in\sel_L(\nu)$, then
$$\frac{\nu(R_{n+1}')}{h(R_{n+1}')} \leq C_{25}\,\eta^{\alpha-\frac2{K+1}} 
\frac{\nu(R_n')}{h(R_n')},$$
and if $R_n'\in\sel_S(\nu)$, then
$$\frac{\nu(R_{n+1}')}{h(R_{n+1}')} \leq C_{25}\, 
\frac{\nu(R_n')}{h(R_n')},$$
where $C_{25}$ is the maximum of the corresponding constants $C_{22}$ and $C_{24}$ 
in \rf{eqff8} and \rf{eqfff8}.
Notice that, by construction, for all $m$, it turns out 
that either $R_m'$ or $R_{m+1}'$ belongs to $\sel_L(\nu)$. As a consequence,
$$\sum_{n\geq0} \left(
\frac{\nu(R_n')}{h(R_n')}\right)^{\frac{K+1}K}\leq \sum_{n\geq0} 
\Bigl(C_{25}\,\eta^{\frac12
\left(\alpha-\frac2{K+1}\right)}\Bigr)^{\frac{K+1}K\,n} \leq C,$$
if $\eta$ is chosen small enough (recall that $C_{25}$ is independent of $\eta$).
Thus, $S_2\leq C(\eta)$ and \rf{eqfi99} follows.


\bigskip

\section{Examples showing sharpness of results}\label{SectionExamples}

\subsection{Some results from \cite{ACMOU}}

The state-of-the-art for largest ``metric" (or ``size") sufficient conditions for removability theorems for bounded $K$ quasiregular maps is given by Theorem 1.2 in \cite{ACMOU}.

\begin{theorem}[Astala, Clop, Mateu, Orobitg, Uriarte-Tuero]
\label{RemovabilityBoundedKQRACMOUT}
Let $K > 1$ and suppose $E \subset \C$ is a compact set with $\H^{\frac{2}{K+1}} (E)$ $\sigma$-finite. Then $E$ is removable for bounded $K$ quasiregular maps.
\end{theorem}

As a first remark, let us mention that from Theorem \ref{distorgamma} we recover this result. Indeed, if $E \subset \C$ and $\H^{\frac{2}{K+1}} (E) < \infty$, then 
$\dot C_{\frac{2K}{2K+1}, \frac{2K+1}{K+1}}(E)=0$. Also if $E_i \subset \C$, for $i = 1,2,\dots$ and $E = \bigcup_{i=1}^{\infty} E_i$ with $\H^{\frac{2}{K+1}} (E_i) < \infty$, then $\dot C_{\frac{2K}{2K+1}, \frac{2K+1}{K+1}}(E) 
\leq \sum_{i=1}^{\infty} \dot C_{\frac{2K}{2K+1}, \frac{2K+1}{K+1}}(E_i)=0$,
by the subadditivity of $\dot C_{\frac{2K}{2K+1}, \frac{2K+1}{K+1}}$ (see \cite[Proposition 2.3.6]{adamshedberg}, for example). Consequently, recalling that by Stoilow's factorization any $K$-quasiregular map $f$ can be factored as $f = h \circ g$, where $h$ is analytic and $g$ is $K$-quasiconformal, by Theorem \ref{distorgamma} in the present paper, $E$ is removable. 

Of course, in order to prove Theorem \ref{distorgamma}, we used many of the ideas in \cite{ACMOU}, so we are not claiming any novelty.

To contextualize some of our examples below, we recall the next result from \cite{ACMOU}.

\begin{theorem}[Astala, Clop, Mateu, Orobitg, Uriarte-Tuero]
\label{NonRemovabilityBoundedKQRACMOUT}
Let $K\geq 1$. Suppose  $h(t)=t^\frac{2}{K+1}\,\varepsilon(t)$ is a measure function such that
 \begin{equation} \label{kesa2}
  \int_0 \frac{\varepsilon(t)^{1+1/K}}{t} dt < \infty.
 \end{equation}
 Then there is a compact set $E\, $ which is not $K$-removable and 
 such that $0<\H^h(E)<\infty$.
 In particular, whenever  $\varepsilon(t)$ is chosen so that in addition for every  $\alpha >0$ we have  $t^\alpha/\varepsilon(t) \to 0$  as  $t \to 0$,   
then  the construction gives  a non-$K$-removable set $E$ with $\dim(E)=\frac{2}{K+1}$.
\end{theorem}

\subsection{Example 1}\label{secex1}

Our next example shows that Theorem \ref{distorgamma} is strictly stronger than Theorem \ref{RemovabilityBoundedKQRACMOUT}. Indeed, let us recall Theorem 5.4.2 in \cite{adamshedberg}, adapted to our situation.

\begin{theorem}\label{Theorem5.4.2AdamsHedberg}
Let $h$ be an increasing nonnegative function on $[0,\infty)$. If
$$
\int_0^1 \left( \frac{h(r)}{r^{\frac{2}{K+1}}}  \right)^{1+\frac{1}{K}} \; \frac{dr}{r} = \infty \ ,
$$
then there is a compact set $E \subset \C$ such that $\H^h(E)>0$ and $\dot C_{\frac{2K}{2K+1}, \frac{2K+1}{K+1}}(E) =0$.
\end{theorem}

If we choose $h(r)$ so that it satisfies the conditions in Theorem \ref{Theorem5.4.2AdamsHedberg} but $\frac{h(r)}{r^{\frac{2}{K+1}}} \to 0$ as $t \to 0$, then the set $E$ obtained in Theorem \ref{Theorem5.4.2AdamsHedberg} will be non-$\sigma$-finite for $\H^{\frac{2}{K+1}}$, but will be removable for bounded $K$-quasiregular maps due to Theorem \ref{distorgamma} and Stoilow's factorization. For this purpose it is enough to choose $h(r) = \frac{r^{\frac{2}{K+1}}}{\log \left( \frac{1}{r} \right)^\beta}$ when $r$ is small enough, so that $\beta >0$ and $\beta \left( 1 + \frac{1}{K} \right) \leq 1$.

\subsection{Basic construction for the subsequent examples}\label{BasicConstructionCantorSets}

For our subsequent examples we need to refine the construction from Theorem \ref{NonRemovabilityBoundedKQRACMOUT}.
To this end we argue as in \cite{uriartesharpqcstretching}. We assume the reader is familiar with that paper and we will use the notation from it without further reference. The formulae look slightly nicer if we assume in the construction that $\varepsilon_n = 0$ for all $n$, i.e. that we take infinitely many disks in each step, completely filling the area of the unit disk $\mathbb D$ (see equations (2.1), (2.2) and (2.3) in \cite{uriartesharpqcstretching}.) It is not strictly needed to set in that construction $\varepsilon_n = 0$ for all $n$, and we will later indicate the corresponding formulae if $\varepsilon_n > 0$ for all $n$ (which is the case in \cite{uriartesharpqcstretching}.) The construction in \cite{uriartesharpqcstretching} works as well if we set $\varepsilon_n = 0$ for all $n$, the only point that the reader might wonder about is whether the resulting map is $K$-quasiconformal. 
However, this can be seen easily by a compactness argument (approximating the desired map by maps with finitely many circles in each step which are $K$-quasiconformal and have more and more disks in each step of the construction).

So we get (see equations (2.5) and (2.6) in \cite{uriartesharpqcstretching}) a Cantor type set $E$ and a $K$-quasiconformal map $\phi$ so that a building block in the $N^{th}$ step of the construction of the source set $E$ is a disk with radius given by 
\begin{equation}\label{RadiusSourceNthStep}
s_{j_1,...,j_N}=\left( (\sigma_{1,j_1})^K \, R_{1,j_1} \right) \dots \left( (\sigma_{N,j_N})^K R_{N,j_N} \right) ,
\end{equation}
and a building block in the $N^{th}$ step of the construction in the target set $\phi(E)$ is a disk with radius given by
\begin{equation}\label{RadiusTargetNthStep}
t_{j_1,...,j_N}=\left( \sigma_{1,j_1} \, R_{1,j_1} \right) \dots \left( \sigma_{N,j_N} \, R_{N,j_N} \right) .
\end{equation}

Now we consider a measure $\mu$ supported on $\phi(E)$ (which will be the ``large" set of dimension $d' = 1 $) and its image measure $\nu = \phi^{-1}_\ast \mu$ supported on $E$ (which will be the ``small" set of dimension $d = \frac{2}{K+1}$) given by splitting the mass according to area. More explicitly,

\begin{equation}\label{DefinitionOfMuStep0}
\mu (\D) = 1,
\end{equation}
for any disk $B_{1,j_1} = \psi^{i_1}_{1,j_1} \left( \, \overline{\D} \, \right)$ of the first step of the construction with radius $t_{j_1} = \left( \sigma_{1,j_1} \, R_{1,j_1} \right)$,
\begin{equation}\label{DefinitionOfMuStep1}
\mu (B_{1,j_1}) = \left( R_{1,j_1} \right)^2,
\end{equation}
and in general, for any disk $B_{N ; j_1, \dots , j_N}^{i_1, \dots , i_N} = \psi^{i_1}_{1,j_1}  \circ \dots \circ \psi^{i_N}_{N,j_N} \left( \, \overline{\D} \, \right)$ of the $N^{th}$ step of the construction with radius 
$t_{j_1,...,j_N}=\left( \sigma_{1,j_1} \, R_{1,j_1} \right) \dots \left( \sigma_{N,j_N} \, R_{N,j_N} \right) $,
\begin{equation}\label{DefinitionOfMuStepN}
\mu (B_{N ; j_1, \dots , j_N}^{i_1, \dots , i_N}) = \left( R_{1,j_1} \dots R_{N,j_N} \right)^2\ .
\end{equation}

Since we took $\varepsilon_N = 0 $ for all $N$, the total mass of $\mu$ is always $1$ on every step. (If one prefers to take $\varepsilon_N > 0 $ for all $N$, the definition should be changed to $\mu (B_{N ; j_1, \dots , j_N}^{i_1, \dots , i_N}) = \left( R_{1,j_1} \dots R_{N,j_N} \right)^2\ \ \prod_{n=N+1}^{\infty} \left( 1- \varepsilon_n \right) $, and the total mass of $\mu$ gets renormalized by the factor $\prod_{n=1}^{\infty} \left( 1- \varepsilon_n \right) >0
$, but otherwise the rest of the construction we are about to describe works, keeping in mind these renormalizations.)

Since $\nu$ is the image measure, for any disk $D_{N ; j_1, \dots , j_N}^{i_1, \dots , i_N} = \varphi^{i_1}_{1,j_1}  \circ \dots \circ \varphi^{i_N}_{N,j_N} \left( \, \overline{\D} \, \right) = \phi^{-1}( B_{N ; j_1, \dots , j_N}^{i_1, \dots , i_N} = \varphi^{i_1}_{1,j_1}  \circ \dots \circ \varphi^{i_N}_{N,j_N} \left( \, \overline{\D} \, \right) )$
we get
\begin{equation}\label{DefinitionOfNuStepN}
\nu (D_{N ; j_1, \dots , j_N}^{i_1, \dots , i_N}) = \left( R_{1,j_1} \dots R_{N,j_N} \right)^2\ .
\end{equation}

\begin{lemma}\label{ComputingWolffPotentialInOurCantorSets}
For the Cantor type sets just described (in subsection \ref{BasicConstructionCantorSets}), for any $\alpha, p >0$ with $\alpha p < 2$, and for $x \in E$, the Wolff potentials satisfy

$$\dot W^\mu_{\alpha,p}(x) \approx \sum_n \biggl(\frac{\mu(B(x,2^{n}))}{2^{n(2-\alpha p)}}\biggr)^{p'-1}\, \approx 
\sum_{ N : x \in B_{N ; j_1, \dots , j_N}^{i_1, \dots , i_N} } \biggl(\frac{\mu( B_{N ; j_1, \dots , j_N}^{i_1, \dots , i_N} ) }{ \left( t_{j_1,...,j_N} \right)^{(2-\alpha p)}}\biggr)^{p'-1} \ ,
$$
and analogously for $\nu$, $D_{N ; j_1, \dots , j_N}^{i_1, \dots , i_N}$ and $s_{j_1,...,j_N}$.

\end{lemma}

\begin{proof}

We first introduce some convenient notation. 
For any multiindexes $I=(i_1,...,i_N)$ and  $J=(j_1,...,j_N)$, where $1\leq i_k , 
j_k \leq \infty $ (since we are taking infinitely many disks in each step of the construction), 
we will denote by
\begin{equation}\label{DefinitionProtectingDisk}
P^{N}_{I;J} = \frac{1}{\sigma_{N,j_N}}\, \psi^{i_1}_{1,j_1} \circ \dots \circ \psi^{i_N}_{N,j_N}(\D)
\end{equation}
a {\it{protecting}} disk of generation $N$. Then, $P^N_{I;J}$ has radius 
$$r(P^{N}_{I;J}) = \frac{1}{ \sigma_{N,j_N} } t_{j_1,...,j_N}=\left( \sigma_{1,j_1} \, \dots \sigma_{N-1,j_{N-1}} \right) \left(   R_{1,j_1} \dots   R_{N,j_N} \right).$$ 
Analogously, we will write
\begin{equation}\label{DefinitionGeneratingDisk}
G^{N}_{I;J} = \psi^{i_1}_{1,j_1}  \circ \dots \circ \psi^{i_N}_{N,j_N}(\D)
\end{equation}
in order to denote a {\it{generating}} disk of generation $N$, which has radius $$r(G^{N}_{I;J}) = t_{j_1,...,j_N}=\left( \sigma_{1,j_1} \, \dots \sigma_{N,j_N} \right)  \left(   R_{1,j_1} \dots   R_{N,j_N} \right).$$

Notice that, since all values of $\sigma_{n,j_n}$ and $R_{n,j_n}$ are $\leq \frac{1}{100}$, then $\mu (G^{N}_{I;J}) = \mu (2 G^{N}_{I;J})$, so we can pretend without loss of generality that the radii $t_{j_1,...,j_N}$ are dyadic numbers.

Now, if $ r(G^{N}_{I;J}) \lesssim t \lesssim r(P^{N}_{I;J}) $, and $x \in E$ so that $B(x, t) \subseteq P^{N}_{I;J}$, then $\mu ( B(x,t)) = \mu ( G^{N}_{I;J} )$, so that 
$$
\sum_{n : G^{N}_{I;J} \subseteq B(x,2^{n}) \subseteq P^{N}_{I;J}} \biggl(\frac{\mu(B(x,2^{n}))}{2^{n(2-\alpha p)}}\biggr)^{p'-1}
$$
is a geometric series with sum comparable (with constants of comparison only depending on $\alpha$ and $p$) to its largest term, namely, up to universal constants, $ \biggl(\frac{\mu( G^{N}_{I;J} ) }{ \left( t_{j_1,...,j_N} \right)^{(2-\alpha p)}}\biggr)^{p'-1}  $.

And if $ r(P^{N}_{I;J}) \lesssim t \lesssim r(G^{N-1}_{I';J'}) $, where $G^{N-1}_{I';J'}$ is the unique generating disk of generation $N-1$ containing $P^{N}_{I;J}$, and $x \in E$ so that $P^{N}_{I;J} \subseteq B(x, t) \subseteq G^{N-1}_{I';J'}$, then 
\begin{equation}\label{UpperBoundForMuOfBallIfBallInBetweenProtectingNAndGeneratingNMinus1}
\mu ( B(x,t)) \lesssim \frac{ t^2}{ \left( \sigma_{1,j_1} \, \dots \sigma_{N-1,j_{N-1}} \right)  \left(   R_{1,j_1} \dots   R_{N-1,j_{N-1}} \right)}  \left(   R_{1,j_1} \dots   R_{N-1,j_{N-1}} \right)^2 \ , 
\end{equation}
i.e. the mass that $\mu$ assigns to $B(x,t)$ is proportional to its area once $G^{N-1}_{I';J'}$ is renormalized to $\D$, but multiplied by the mass that $\mu$ assigns to $G^{N-1}_{I';J'}$, namely $\left(   R_{1,j_1} \dots   R_{N-1,j_{N-1}} \right)^2$. Hence 
$$
\sum_{n : P^{N}_{I;J} \subseteq B(x,2^{n}) \subseteq G^{N-1}_{I';J'}} \biggl(\frac{\mu(B(x,2^{n}))}{2^{n(2-\alpha p)}}\biggr)^{p'-1} 
$$
is dominated by a geometric series (if $n$ appears in the above sum and $2^{n} = \frac{ r(G^{N-1}_{I';J'}) }{2^k}$ with $k >0$, then $ \biggl(\frac{\mu(B(x,2^{n}))}{2^{n(2-\alpha p)}}\biggr)^{p'-1} \lesssim  \biggl(\frac{\mu( G^{N-1}_{I';J'})}{r(G^{N-1}_{I';J'})^{(2-\alpha p)}} \ \frac{ 2^{k(2-\alpha p)} }{2^{2k} }  \biggr)^{p'-1} $ , and hence the above sum is $ \lesssim  \biggl(\frac{\mu( G^{N-1}_{I';J'} )}{r(G^{N-1}_{I';J'})^{(2-\alpha p)}}\biggr)^{p'-1}$, with constants of comparison only depending on $\alpha$ and $p$.)
\end{proof}


\subsection{Example 2}\label{secex2}

In view of example 1, it is natural to wonder whether all compact sets $E$ with $\H^h(E)=0$ for some gauge function $h(r)$ satisfying $\frac{h(r)}{r^{\frac{2}{K+1}}} \to 0$ are 
removable for bounded $K$-quasiregular maps, i.e. whether there is some condition strictly weaker than $\H^{\frac{2}{K+1}}(E)$ being $\sigma$-finite in terms of the gauge function $h$, which guarantees removability. Our next example shows that this is not the case. Notice the resemblance to Theorem 5.4.1 in \cite{adamshedberg}.

\begin{theorem}\label{ForAnyGaugeFunctionMeasuringSetsStrictlyLargerThan2/K+1ThereIsANonRemovableSet}
Let $h$ be a positive function on $(0,\infty)$ such that
$$\varepsilon (r) = \frac{h(r)}{r^{\frac{2}{K+1}}} \to 0 \quad\mbox{ as $r\to0$.}$$ 
 Then there is a compact set $E \subset \C$ such that $\H^h(E)=0$ and a $K$-quasiconformal map $\phi$ such that $\gamma(\phi E) >0$ (and hence $\dot C_{ \frac{2K}{2K+1},\frac{2K+1}{K+1} }(E) >0$, due to Theorem \ref{distorgamma}.)

\end{theorem}

\begin{proof}

For $E$ and $\phi$ as in Subsection \ref{BasicConstructionCantorSets}, 
notice that by Lemma \ref{ComputingWolffPotentialInOurCantorSets}, for $x \in \phi E$

$$\dot W^{\mu}_{ \frac{2}{3} , \frac{3}{2}  } (x) \approx \sum_{ N : x \in B_{N ; j_1, \dots , j_N}^{i_1, \dots , i_N} } \biggl(\frac{\mu( B_{N ; j_1, \dots , j_N}^{i_1, \dots , i_N} ) }{ \left( t_{j_1,...,j_N} \right)}\biggr)^{2}  =  \sum_{ N : x \in B_{N ; j_1, \dots , j_N}^{i_1, \dots , i_N} } \biggl(\frac{ R_{1,j_1} \dots R_{N,j_N} }{\sigma_{1,j_1} \dots \sigma_{N,j_N} }\biggr)^{2}.$$

Since on the one hand $E$ is very ``close" to satisfying $0 < \H^{\frac{2}{K+1}}(E) < \infty$ 
and $0 < \H^{1}(\phi E) < \infty$ (see (3.11) and (4.5) in \cite{uriartesharpqcstretching}), and on the other hand an important element in 
the proof of the semiadditivity of analytic capacity is that the potential is ``approximately 
constant" on each scale (see \cite{tolsasemiadditivityanalyticcapacity}), the above equation 
suggests the choice
\begin{equation}\label{ChoiceOfSigma}
\sigma_{N,j_N} = R_{N,j_N} \, d_N\quad\mbox{ for all $N$,}
\end{equation} 
where $d_N \in [1,2]$ is a parameter to be determined, independent of $j_N$.

If we take
\begin{equation}\label{ChoiceOfDjForExample2}
d_j = \frac{j+1}{j},
\end{equation} 
then, for $x \in \phi E$, we have
$$ \dot W^{\mu}_{ \frac{2}{3} , \frac{3}{2}  } (x) \approx \sum_n \left\{ \prod_{j=1}^{n} \frac{1}{\left( d_j \right)^2} \right\} = \sum_{n=2}^{\infty} \frac{1}{n^2} < \infty,$$
so that $ \dot C_{ \frac{2}{3} , \frac{3}{2} } (\phi E) >0$, and $\gamma (\phi E) >0$.

Let us denote $\ve^{N}_{max} = \max \left\{\ve(s_{j_1,...,j_N}) \right\}$.
For each $N$, substituting $\sigma_{N,j_N} = R_{N,j_N} \, d_N$, recalling that $\sum_{j_1, \dots , j_N} \left(  R_{1,j_1} \dots R_{N,j_N}  \right)^2 = 1$, and that $d_n = \frac{n+1}{n}$, we obtain
\begin{eqnarray}\label{UpperEstimateHausdorffMeasureExample4}
\sum_{j_1, \dots , j_N} h( r(D_{N ; j_1, \dots , j_N}^{i_1, \dots , i_N} ) ) & = & \sum_{j_1, \dots , j_N} h( s_{j_1,...,j_N} ) \nonumber \\
& = & \sum_{j_1, \dots , j_N} \varepsilon (s_{j_1,...,j_N}) \left(  (\sigma_{1,j_1})^K \, R_{1,j_1} \right) \dots \left( (\sigma_{N,j_N})^K R_{N,j_N} \right)^{\frac{2}{K+1}}  \nonumber \\
& \leq & \ve^{N}_{max} \ \left(  d_1 \dots d_N \right)^{\frac{2K}{K+1}} \sum_{j_1, \dots , j_N} \left(  R_{1,j_1} \dots R_{N,j_N}  \right)^2  \nonumber \\
& = & \ve^{N}_{max} \ \left(  d_1 \dots d_N \right)^{\frac{2K}{K+1}} = \ve^{N}_{max} \ \left(  N+1 \right)^{\frac{2K}{K+1}}.
\end{eqnarray}
Choosing $R_{1,j_1} \dots R_{N,j_N}$  small enough in the construction so that 
$\ve^{N}_{max} \ \left(  N+1 \right)^{\frac{2K}{K+1}} \to 0$ as $N \to \infty$, 
one infers that $\H^h(E)=0$.
\end{proof}

\subsection{Example 3}\label{secex5}

The preceding example can be modified (notice the analogies with Theorem 5.6.4 in \cite{adamshedberg}) to show that

\begin{theorem}\label{ThereIsAUniversalNonRemovableSetForAnyGaugeFunctionMeasuringSetsStrictlyLargerThan2/K+1WithConvergentIntegral}
There is a compact set $E \subset \C$ such that $\gamma(\phi E) >0$ (and hence $\dot C_{ \frac{2K}{2K+1},\frac{2K+1}{K+1} }(E) >0$, due to Theorem \ref{distorgamma}), but $\H^h(E)=0$ for every positive function $h$ such that 
$$\varepsilon (r) = \frac{h(r)}{r^{\frac{2}{K+1}}} \,\text{ is non decreasing},$$
and 
$$
\int_0^1 \left( \frac{h(r)}{r^{\frac{2}{K+1}}}  \right)^{a} \; \frac{dr}{r} < \infty,\;
\mbox{ for some $a>0$.}
$$
\end{theorem}

\begin{proof} In the preceding construction, denote  $S^{N}_{max} = \max \left\{ s_{j_1,...,j_N} \right\}$ and
choose $S^{N}_{max} \leq 
e^{-e^N}$. Since $\ve$ is non decreasing, $\ve(s_{j_1,...,j_N})\leq \ve(e^{-e^N})$, and 
from \eqref{UpperEstimateHausdorffMeasureExample4} we deduce
$$\left[ \H^h(E)  \right]^{a}   \lesssim  \liminf_{N \to \infty} 
\left\{ \left[ \ve(S^{N}_{max}) \right]^{a} \ N^{\frac{2Ka}{K+1}} \right\} \lesssim \liminf_{N \to \infty} \left\{  \sum_{n=N}^{\infty}  \left[ \varepsilon (e^{-e^{n}})  \right]^{a} \ n^{\frac{2Ka}{K+1}} \right\}.$$
Using that again that $\ve$ is non decreasing and setting $s = e^{- \frac{1}{t} }$, we obtain
\begin{eqnarray}\label{UpperEstimateHausdorffMeasureExample5}
\left[ \H^h(E)  \right]^{a}  & \lesssim  & \liminf_{N \to \infty} \left\{ 
\sum_{n=N}^\infty\int_{e^{-n}}^{e^{-n+1}} \left[ \varepsilon 
(e^{- e^{n }})  \right]^{a} \ 
\left[  \log \left( \frac{1}{t} \right)  \right]^{\frac{2Ka}{K+1}} \ \frac{dt}{t} \right\} \nonumber \\
& \leq & \liminf_{N \to \infty} \left\{ \int_{0}^{e^{-N+1}} \left[ \varepsilon 
(e^{- \frac{1}{t} } ) \right]^{a} \ 
\left[  \log \left( \frac{1}{t} \right)  \right]^{\frac{2Ka}{K+1}} \ \frac{dt}{t} \right\} \nonumber \\
& = & \liminf_{N \to \infty} \left\{  \int_{0}^{e^{-e^{N-1}}} \left[ \varepsilon (s) 
\right]^{a} \ \frac{ \left[ \log \log \left( \frac{1}{s} \right) \right]^{\frac{2Ka}{K+1}} }{  \log \left( \frac{1}{s} \right) } \ \frac{ds}{s} \right\} \nonumber \\
& \lesssim & \liminf_{N \to \infty} \left\{  \int_{0}^{e^{-e^{N-1}}} \left[ \varepsilon (s) 
\right]^{a} \ \frac{ds}{s} \right\} = 0.
\end{eqnarray}
\end{proof}

\subsection{Example 4}\label{secex4}

Examples 2 and 3 strongly suggest that the language of capacities $\dot C_{\alpha, p}$ is better suited to understand the removability for bounded $K$-quasiregular maps than the language of Hausdorff measures. Hence it is natural to wonder how sharp Theorem \ref{distorgamma} is in the category of capacities $\dot C_{\alpha, p}$. To that effect, it is useful to recall Theorem 5.5.1 (b) in \cite{adamshedberg} adapted to our situation (and combined with Proposition 5.1.4):

\begin{theorem}\label{TheoremRelatingCapacities}
Let $E \subset \C$. Then there is a constant $A$ such that 
$$
\dot C_{\beta, q} (E) \leq A \dot C_{\alpha, p} (E) \ ,
$$
for $\beta q = \alpha p = 2-\frac{2}{K+1} = \frac{2K}{K+1}$, $p<q$.

Moreover, there exist sets $E$ such that $\dot C_{\beta, q} (E) =0$ but $\dot C_{\alpha, p} (E) >0$.
\end{theorem}

Hence it is conceivable that Theorem \ref{distorgamma} might be strengthened to a statement of the form 
$$
\frac{\dot C_{ \beta, q}(E)}{\diam(B)^{\frac2{K+1}}}  \geq 
c^{-1} \left(\frac{\gamma(\phi(E))}{\diam(\phi(B))}\right)^{\frac{2K}{K+1}}
$$
for some $\beta, q$ such that $\beta q = \frac{2K}{K+1}$ and $ \frac{2K+1}{K+1} < q$, i.e. for $q'-1 < 1 + \frac{1}{K}$. The following theorem shows that the answer to this question is negative.

\begin{theorem}\label{teosharp}
For any $\beta, q >0$ such that $\beta q = \frac{2K}{K+1}$ and $ \frac{2K+1}{K+1} < q$, there exists a compact $E \subset \C$ and a $K$-quasiconformal map $\phi$ such that $\gamma(\phi E) >0$ (and hence $\dot C_{ \frac{2K}{2K+1},\frac{2K+1}{K+1} }(E) >0$, due to Theorem \ref{distorgamma}), but $ \dot C_{ \beta, q}(E) = 0$.
\end{theorem}

\begin{proof}
As in the construction in Example 2, we choose $\sigma_{N,j_N} = R_{N,j_N} \, d_N$. Then,
for $y \in \phi E$,
$$
\dot W^{\mu}_{ \frac{2}{3} , \frac{3}{2}  } (y) \approx \sum_n \left\{ \prod_{j=1}^{n} \frac{1}{\left( d_j \right)^2} \right\} \ ,
$$
while by Lemma \ref{ComputingWolffPotentialInOurCantorSets} and
\rf{RadiusSourceNthStep}, for $x \in E$,
$$
\dot W^{\nu}_{ \beta , q } (x) \approx \sum_{ N : x \in D_{N ; j_1, \dots , j_N}^{i_1, \dots , i_N} } \biggl(\frac{\nu( D_{N ; j_1, \dots , j_N}^{i_1, \dots , i_N} ) }{ \left( s_{j_1,...,j_N} \right)^{ \frac{2}{K+1} }}\biggr)^{q'-1}  =  \sum_{ N : x \in D_{N ; j_1, \dots , j_N}^{i_1, \dots , i_N} } \biggl(\frac{ R_{1,j_1} \dots R_{N,j_N} }{\sigma_{1,j_1} \dots \sigma_{N,j_N} }\biggr)^{\frac{2K}{K+1}\left( q'-1 \right) } \ ,
$$
so that, substituting $\sigma_{N,j_N} = R_{N,j_N} \, d_N$ we get, for $x \in E$,
$$
\dot W^{\nu}_{ \beta , q } (x) \approx \sum_n \left\{ \prod_{j=1}^{n} \frac{1}{\left( d_j \right)^2} \right\}^{\left( q'-1 \right) \left( \frac{K}{K+1} \right)  } \ .
$$

Now choose $\left( d_j \right)^{2 \left( q'-1 \right) \left( \frac{K}{K+1} \right)  }= \dfrac{j+1}{j}$, so that for $x \in E$, $\dot W^{\nu}_{ \beta , q } (x) \approx \sum_{n=2}^{\infty}  \dfrac{1}{n} = \infty$, while for $y \in \phi E$,
$$\displaystyle \dot W^{\mu}_{ \frac{2}{3} , \frac{3}{2}  } (y) \approx \sum_{n=2}^{\infty}  \frac{1}{\left\{  n^{ \frac{1}{ \left( q'-1 \right) \left( \frac{K}{K+1} \right) }  }   
\right\}
} < \infty,$$ since $\left( q'-1 \right) \left( \frac{K}{K+1} \right) <1$.

The fact that $\dot W^{\mu}_{ \frac{2}{3} , \frac{3}{2}  } (y)<\infty$ for all $y\in\phi(E)$
implies that $\dot C_{ \frac{2}{3} , \frac{3}{2} } (\phi E) >0$, and hence $\gamma (\phi E) >0$ and $\dot C_{ \frac{2K}{2K+1},\frac{2K+1}{K+1} }(E) >0$),
while from the  fact that $\dot W^{\nu}_{ \beta , q } (x) =\infty$ for all $x\in E$ one
infers that $ \dot C_{ \beta, q}(E) = 0$ (see 
 Proposition 6.3.12 and (6.3.4) in \cite{adamshedberg}, adapted for the potential $\dot W^{\nu}_{ \beta , q }$.)
\end{proof}

Let us remark that the above example also gives that $\dot C_{\gamma,r}(E)=0$ if 
$\gamma\,r<\beta\,q=2K/(K+1)$. This due to the fact that there is some constant $A$ 
indendent of $E$ such that
$$\dot C_{\gamma,r}(E)^{1/(2-\gamma r)}\leq A \,\dot C_{\beta,q}(E)^{1/(2-\beta q)}.$$
See Theorem 5.5.1 of \cite{adamshedberg}.


\section{Final remarks}\label{sec9}

The Main Lemma \ref{mainlem} on the distortion of $h$-contents can also be proved using 
arguments based on the ideas in \cite{Lacey-Sawyer-Uriarte}, instead of 
\cite{ACMOU}. 
Following this new approach one can extend the Main Lemma \ref{mainlem} to $h$-contents
$M^h$, with $h$ of the form
$h(B(x,r))=r^t\,\ve(B(x,r)),$
for all $0<t<2$.
As a consequence, one can extend
Theorem \ref{teocap} (a) to 
all capacities $\dot C_{\alpha,p}$, with $\alpha>0$, $1<p<\infty$, such that
$\alpha p <2$. Then, one obtains the following:

\begin{theorem} \label{teocap2}
Let $1<q<\infty$ and $0<\beta q<2$. Let $t'=2-\beta q$, and $t$ be such that
$$
\frac1t-\frac12 = K\left(\frac{1}{t'}-\frac12\right).
$$
Let $E\subset\C$ be compact, and let $\phi:\C\to\C$ be a $K$-quasiconformal map. If $E$ is contained in a ball $B$, then
\begin{equation}\label{eq111}
\frac{\dot{{\cal C}}_{\beta,q}(\phi(E))}{\diam(\phi(B))^{t'}}\leq C(\beta,q,K)\,\left(\frac{\dot{{\cal C}}_{\alpha,p}(E)}{\diam(B)^t}\right)^\frac{t'}{Kt}
\end{equation}
where
$$
p=1+\frac{Kt}{t'}\,(q-1)\hspace{1cm}\text{and}\hspace{1cm} 2-\alpha p=t.
$$
The constant in \rf{eq111} depends only on $\beta$, $q$, $K$.
\end{theorem}

The proof will appear in \cite{ACTUV}.

\bigskip

\bigskip

\emph{Acknowledgements.} The second named author would like to thank K. Astala, J. Mateu, J. Orobitg, and J. Verdera for fruitful conversations regarding some parts this paper.

\bigskip


\bibliographystyle{alpha}
\bibliography{./references25G}

\newcommand{\etalchar}[1]{$^{#1}$}
\def\cprime{$'$} \def\cprime{$'$} \def\cprime{$'$} \def\cprime{$'$}
\begin{thebibliography}{ACM{\etalchar{+}}08}

\bibitem[ACM{\etalchar{+}}08]{ACMOU}
Kari Astala, Albert Clop, Joan Mateu, Joan Orobitg, and Ignacio Uriarte-Tuero.
\newblock Distortion of {H}ausdorff measures and improved {P}ainlev\'{e}
  removability for bounded quasiregular mappings.
\newblock {\em Duke Math. J.}, 141(3):539--571, 2008.

\bibitem[ACT{\etalchar{+}}]{ACTUV}
Kari Astala, Albert Clop, Xavier Tolsa, Ignacio Uriarte-Tuero, and Joan
  Verdera.
\newblock Quasiconformal distortion of {R}iesz capacities and {H}ausdorff
  measures in the plane.
\newblock {\em Preprint, 2010. To appear in American Journal of Mathematics.
  See http://arxiv.org/pdf/1002.1038.pdf}.

\bibitem[AH96]{adamshedberg}
David~R. Adams and Lars~Inge Hedberg.
\newblock {\em Function spaces and potential theory}, volume 314 of {\em
  Grundlehren der Mathematischen Wissenschaften [Fundamental Principles of
  Mathematical Sciences]}.
\newblock Springer-Verlag, Berlin, 1996.

\bibitem[AIM09]{astalaiwaniecmartin}
Kari Astala, Tadeusz Iwaniec, and Gaven Martin.
\newblock {\em Elliptic partial differential equations and quasiconformal
  mappings in the plane}, volume~48 of {\em Princeton Mathematical Series}.
\newblock Princeton University Press, Princeton, NJ, 2009.

\bibitem[AIS01]{astalaiwaniecsaksman}
Kari Astala, Tadeusz Iwaniec, and Eero Saksman.
\newblock Beltrami operators in the plane.
\newblock {\em Duke Math. J.}, 107(1):27--56, 2001.

\bibitem[AN03]{astalanesi}
Kari Astala and Vincenzo Nesi.
\newblock Composites and quasiconformal mappings: new optimal bounds in two
  dimensions.
\newblock {\em Calc. Var. Partial Differential Equations}, 18(4):335--355,
  2003.

\bibitem[Ast94]{astalaareadistortion}
Kari Astala.
\newblock Area distortion of quasiconformal mappings.
\newblock {\em Acta Math.}, 173(1):37--60, 1994.

\bibitem[CT08]{Clop-Tolsa}
Albert Clop and Xavier Tolsa.
\newblock Analytic capacity and quasiconformal mappings with {$W\sp {1,2}$}
  {B}eltrami coefficient.
\newblock {\em Math. Res. Lett.}, 15(4):779--793, 2008.

\bibitem[Dav98]{davidunrectifiable1setszeroanalyticcapacity}
Guy David.
\newblock Unrectifiable {$1$}-sets have vanishing analytic capacity.
\newblock {\em Rev. Mat. Iberoamericana}, 14(2):369--479, 1998.

\bibitem[ENV10]{ENV}
V.~Eiderman, F.~Nazarov, and A.~Volberg.
\newblock Vector-valued {R}iesz potentials: {C}artan-type estimates and related
  capacities.
\newblock {\em Proc. Lond. Math. Soc. (3)}, 101(3):727--758, 2010.

\bibitem[LSUT10]{Lacey-Sawyer-Uriarte}
Michael~T. Lacey, Eric~T. Sawyer, and Ignacio Uriarte-Tuero.
\newblock Astala's conjecture on distortion of {H}ausdorff measures under
  quasiconformal maps in the plane.
\newblock {\em Acta Math.}, 204(2):273--292, 2010.

\bibitem[Mat95]{mattila}
Pertti Mattila.
\newblock {\em Geometry of sets and measures in {E}uclidean spaces}, volume~44
  of {\em Cambridge Studies in Advanced Mathematics}.
\newblock Cambridge University Press, Cambridge, 1995.
\newblock Fractals and rectifiability.

\bibitem[Mat96]{mattilaanalyticcapacitycantorsets}
Pertti Mattila.
\newblock On the analytic capacity and curvature of some {C}antor sets with non
  $\sigma$-finite length.
\newblock {\em Pub. Mat.}, 40, 1996.

\bibitem[Mel95]{Melnikov}
M.~S. Mel{\cprime}nikov.
\newblock Analytic capacity: a discrete approach and the curvature of measure.
\newblock {\em Mat. Sb.}, 186(6):57--76, 1995.

\bibitem[MPV05]{mateupratverdera}
Joan Mateu, Laura Prat, and Joan Verdera.
\newblock The capacity associated to signed {R}iesz kernels, and {W}olff
  potentials.
\newblock {\em J. Reine Angew. Math.}, 578:201--223, 2005.

\bibitem[MTV03]{MTV}
Joan Mateu, Xavier Tolsa, and Joan Verdera.
\newblock The planar {C}antor sets of zero analytic capacity and the local
  {$T(b)$}-theorem.
\newblock {\em J. Amer. Math. Soc.}, 16(1):19--28 (electronic), 2003.

\bibitem[Par90]{Paramonov}
P.~V. Paramonov.
\newblock Harmonic approximations in the {$C^1$}-norm.
\newblock {\em Mat. Sb.}, 181(10):1341--1365, 1990.

\bibitem[Pom92]{pommerenke}
Ch. Pommerenke.
\newblock {\em Boundary behaviour of conformal maps}, volume 299 of {\em
  Grundlehren der Mathematischen Wissenschaften [Fundamental Principles of
  Mathematical Sciences]}.
\newblock Springer-Verlag, Berlin, 1992.

\bibitem[Pra04]{prat1}
Laura Prat.
\newblock Potential theory of signed {R}iesz kernels: capacity and {H}ausdorff
  measure.
\newblock {\em Int. Math. Res. Not.}, (19):937--981, 2004.

\bibitem[Rei74]{reimann}
Hans~Martin Reimann.
\newblock Functions of bounded mean oscillation and quasiconformal mappings.
\newblock {\em Comment. Math. Helv.}, 49:260--276, 1974.

\bibitem[Tol03]{tolsasemiadditivityanalyticcapacity}
Xavier Tolsa.
\newblock Painlev\'e's problem and the semiadditivity of analytic capacity.
\newblock {\em Acta Math.}, 190(1):105--149, 2003.

\bibitem[Tol05]{tolsabilip}
Xavier Tolsa.
\newblock Bilipschitz maps, analytic capacity, and the {C}auchy integral.
\newblock {\em Ann. of Math. (2)}, 162(3):1243--1304, 2005.

\bibitem[UT08]{uriartesharpqcstretching}
Ignacio Uriarte-Tuero.
\newblock Sharp examples for planar quasiconformal distortion of {H}ausdorff
  measures and removability.
\newblock {\em International Mathematics Research Notices}, 2008: rnn047-43,
  2008.

\bibitem[V{\"a}i86]{vaisala}
Jussi V{\"a}is{\"a}l{\"a}.
\newblock Bi-{L}ipschitz and quasisymmetric extension properties.
\newblock {\em Ann. Acad. Sci. Fenn. Ser. A I Math.}, 11(2):239--274, 1986.

\bibitem[Vol03]{volberg}
Alexander Volberg.
\newblock {\em Calder\'on-{Z}ygmund capacities and operators on nonhomogeneous
  spaces}, volume 100 of {\em CBMS Regional Conference Series in Mathematics}.
\newblock Published for the Conference Board of the Mathematical Sciences,
  Washington, DC, 2003.

\end{thebibliography}

\enlargethispage{2cm}
\end{document}